\documentclass[reqno,english,11pt]{amsart}

\usepackage{tcolorbox}
\usepackage{multicol}

\usepackage{mathtools}

\usepackage[margin=1in]{geometry}

\usepackage{iftex}
\ifPDFTeX
  \usepackage[T1]{fontenc}
  \usepackage[utf8]{inputenc}
\fi
\usepackage{amsbsy}
\usepackage{amstext}
\usepackage{amsthm}
\usepackage{amssymb}
\usepackage{bm}
\usepackage{xcolor}
\usepackage{amsfonts,euscript,mathrsfs,color,amsmath,latexsym}

\numberwithin{equation}{section}
\numberwithin{figure}{section}
\theoremstyle{plain}
\newtheorem{thm}{\protect\theoremname}[section]
\theoremstyle{remark}
\newtheorem{rem}[thm]{\protect\remarkname}
\theoremstyle{definition}
\newtheorem{defn}{\protect\definitionname}[section]

\newtheorem{lemma}{Lemma}[section]
\newtheorem{proposition}{Proposition}[section]

\newtheorem{cor}{Corollary}[section]

\usepackage{hyperref}

\newcommand{\norm}[1]{\left\lVert#1\right\rVert}

\makeatother

\usepackage{babel}
\providecommand{\definitionname}{Definition}
\providecommand{\remarkname}{Remark}
\providecommand{\theoremname}{Theorem}

\usepackage{etoolbox}
\AtEndEnvironment{displaymath}{\noindent}

\begin{document}
	
\title{Global regularity for 2D gravity water waves with two retreating point vortices}	
	
\author{Lei Su}
\author{Qingtang Su}
\author{Siwei Wang}
\address[L.\ Su]{Morningside Center of Mathematics, Academy of Mathematics and Systems Sciences,
Chinese Academy of Sciences (CAS), Beijing, 100080, People's Republic of China}
\email{sulei@amss.ac.cn}
\address[Q.\ Su]{Morningside Center of Mathematics, Academy of Mathematics and Systems Sciences,
Chinese Academy of Sciences (CAS), Beijing, 100080, People's Republic of China}
\email{suqingtang@amss.ac.cn}
\address[Wang]{Department of Mathematics, University of Michigan, Ann Arbor, MI 48109，USA}
\email{wsiwei@umich.edu}

\begin{abstract}
We prove global well-posedness for the two-dimensional infinite-depth gravity water wave system coupled to a pair of point vortices of opposite strengths. The vortex dipole is initially placed deep below the free surface and oriented so that its leading motion is away from the interface. We show that the vortices retreat almost linearly, and that the velocity induced on the free boundary is therefore integrable in time. Combining this decay with Wu's cubic formulation, together with a transition-of-derivatives and localization argument, we obtain global bounds for small localized perturbations of the flat surface.
\end{abstract}

\maketitle

\section{Introduction}

\subsection{Background}
The motion of a two-dimensional incompressible inviscid fluid with a free
surface under gravity, normalized to $(0,-1)$, is governed by the Euler
equations in the fluid domain $\Omega(t)$:
\begin{equation}\label{vortex_model}
\begin{cases}
v_t+v\cdot \nabla v = -\nabla P - (0,1) \quad \text{in } \Omega(t),\\
\operatorname{div} v = 0,
\end{cases}
\end{equation}
with boundary conditions on the free surface $\Sigma(t)$
\begin{equation}\label{vortex_model_boundary}
P\equiv 0,\qquad (1,v) \text{ is tangent to } (t,\Sigma(t)).
\end{equation}
Here $v$ is the fluid velocity and $P$ is the pressure. The evolution of
incompressible inviscid free-boundary flows is a central problem in fluid
dynamics and nonlinear partial differential equations. When the flow is
irrotational, the motion is completely determined by the motion of the free boundary, and the system reduces to a quasilinear
dispersive equation on the free surface.

The local theory for irrotational water waves began with the small-data works
of Nalimov \cite{Nalimov}, Yosihara \cite{Yosihara}, and Craig \cite{Craig}.
For large Sobolev data, Wu \cite{Wu1997,Wu1999} proved the Taylor sign
condition and established the first general local existence results. Since
then many local well-posedness results have been obtained; see for instance
\cite{alazard2014cauchy, ambrose2005zero, christodoulou2000motion,
coutand2007well, iguchi2001well, lannes2005well, lindblad2005well,
ogawa2002free, shatah2006geometry, zhang2008free} and the references therein.
See also \cite{Wu1,Wu2,Wu3,wu2019wellposedness} for water waves with
non-smooth interfaces, and \cite{castro2012finite, castro2013finite,
coutand2014finite, coutand2016impossibility} for splash singularities.

The study of long-time dynamics requires a more precise understanding of the
nonlinear structure. It has been known since the work of Dyachenko and
Zakharov \cite{dyachenko1994free} that weakly nonlinear two-dimensional
infinite-depth gravity waves have no three-wave resonances and that the
four-wave interaction coefficient vanishes on the non-trivial resonant
manifold. However, a rigorous justification is missed for a long time. In physical-space coordinates,
Wu \cite{Wu2009} discovered an explicit cubic structure for the
two-dimensional gravity water wave equation, which led to almost global
existence for small localized data. Ionescu and Pusateri \cite{Ionescu2015}
combined Wu's cubic structure with modified scattering to prove global
existence. Around the same time, Alazard and Delort \cite{alazard2015global}
used paradifferential methods to obtain sharp decay estimates and global
solutions. More recently, Deng, Ionescu, and Pusateri \cite{deng2022wave}
established a
quintic-type
deterministic energy
inequality, consistent
with approximate
quartic integrability, and Ai, Ifrim, and Tataru \cite{ai2022two}
proved global results at low regularity. Further developments include
\cite{Wu2011, germain2012global, HunterTataruIfrim1, HunterTataruIfrim2,
wang2018global, wang2019global, wang2020global, zheng2022long}. A rigorous Birkhoff-normal-form justification was obtained by Berti,
Feola, and Pusateri \cite{berti2023birkhoff}. 

For the rotational case, i.e., $\omega:=\operatorname{curl} v \neq 0$, the problem is far less understood. For local well-posedness with
regular vorticity, see \cite{iguchi1999free, ogawa2002free,
ogawa2003incompressible, christodoulou2000motion, lindblad2005well,
zhang2008free}; for long-time results involving point vortices, see
\cite{su2020long}. Other long-time investigations include \cite{ifrim2015two,
bieri2017motion, ginsberg2018lifespan, ginsberg2024long, langella2026transfer}. Nevertheless,
global existence remains open for general vorticity configurations. A basic
obstruction is that vorticity is transported by the flow,
\begin{equation}\label{eq:vorticity}
  \omega_t+v\cdot\nabla\omega=0,
\end{equation}
and may form long filamentary structures whose interaction with the moving
free boundary is difficult to control. Even without a free boundary,
two-dimensional Euler flows with non-trivial vorticity may generate strong
small-scale structures, including growth of vorticity gradients and
filamentation; see for example \cite{kiselev2012small,Zlatos2015,
Zlatos2026HalfPlaneEuler}. The free surface introduces an additional
dispersive boundary dynamics which interacts with the interior vorticity on
long time scales.

\subsection{A first step: water waves with point vortices}
In this paper we consider the simplest non-trivial vorticity configuration: a
finite number of point vortices. Thus
\begin{equation}\label{vorticitydistribution}
    \omega(\cdot,t)=\sum_{j=1}^N\lambda_j\delta_{z_j(t)}(\cdot),
\end{equation}
where $z_j(t)\in\Omega(t)$ and $\lambda_j\in\mathbb R$ are the vortex
locations and strengths. Point-vortex models are classical reduced models for
coherent vortex interactions. In the free-boundary problem, a point vortex is
a Dirac mass of vorticity, and its trajectory is governed by the
desingularized velocity field at the vortex location. The system (\ref{vortex_model})-(\ref{vortex_model_boundary})-(\ref{vorticitydistribution}) is a model for the motion of submerged bodies (see e.g. \cite{Chang2001},\cite{DalrympleRogers}) and it is believed to give some  insight into the problem of turbulence (\cite{MarchioroPulvirenti2012}, chap 4, \S 4.6). 
By expressing the velocity field in terms of the free-surface parametrization
and the vortex locations, this system can be reduced to a coupled system
for the boundary and the vortex trajectories. Using Lagrangian coordinates, i.e., choosing $\alpha$ such that $z_t(\alpha,t)=v(z(\alpha,t),t)$, the system takes the form
\begin{equation}\label{vortex_boundary}
\begin{cases}
z_{tt}-iaz_{\alpha}=-i,\\[4pt]
\displaystyle\frac{d}{dt}z_j(t)=\Bigl(v-\frac{\lambda_j i}{2\pi(\overline{z-z_j})}\Bigr)\Big|_{z=z_j},\\[10pt]
\displaystyle(I-\mathfrak{H})\Bigl(\bar{z}_t+\sum_{j=1}^N \frac{\lambda_j i}{2\pi(z(\alpha,t)-z_j(t))}\Bigr)=0.
\end{cases}
\end{equation}
Here $a$ is the rescaled normal derivative of the pressure and
$\mathfrak H$ is the Hilbert transform associated with the fluid domain; see
\cite{su2020long} for the derivation. The quantity $a|z_{\alpha}|=-\frac{\partial P}{\partial\hat{n}}\Big|_{\Sigma(t)}$ is called the Taylor sign coefficient and $a|z_{\alpha}|\geq 0$ is called the Taylor-sign condition. If the Taylor sign condition fails, the system is, in general, unstable, see for example, \cite{beale1993growth},\cite{birkhoff1962helmholtz},\cite{taylor1950instability},\cite{ebin1987equations}. In the irrotational case and without a bottom the validity of the Taylor sign condition was shown by Wu \cite{Wu1997, Wu1999}, and was the key to obtaining the first local-in-time existence results for large data in Sobolev spaces. In the case of non-trivial vorticity or with a bottom the Taylor sign condition can fail and the sign condition has to be part of the assumptions for the initial data. In the case of point vortices analyzed in this paper, the second author showed that the Taylor sign condition can fail if the point vortices are close to the interface, see \cite{su2020long, su2023transition}. In \cite{su2020long}, Su studied two-dimensional infinite-depth gravity water waves coupled to a pair of point vortices. The analysis was carried out under a vertical symmetry assumption on the water wave, with initial data of size $O(\epsilon)$. More precisely, Su considered the case $N=2$, with vortex strengths $\lambda_1=\lambda$ and $\lambda_2=-\lambda$, and assumed that the initial distance from the vortices to the free boundary satisfies $H_0=O(1)$. The parameters $H_0, \lambda$, and $\epsilon$ were chosen so that the vortex pair initially moves downward, away from the interface. Under these assumptions, Su proved that the lifespan of the water-wave solution is at least $O(\epsilon^{-2})$, and that the vortices move away from the free boundary with an almost constant downward velocity.

In the present paper, we work within Su’s framework, but remove the symmetry assumption on the water wave. Moreover, our assumptions on $H_0$ and $\lambda$ are more general. Under suitable localized smallness assumptions on the initial data, we prove global well-posedness for this class of water waves with point vortices and investigate the long-time behavior of the coupled water-wave/vortex system.

\subsection{Main theorem}
We parametrize the free surface in Lagrangian coordinates and denote by
$z(\alpha,t)$ the interface, by $v=z_t$ the boundary velocity, and by
$z_1(t),z_2(t)$ the two point vortices with strengths $\lambda$ and
$-\lambda$. The initial data are small localized perturbations of the flat
equilibrium $z(\alpha,0)=\alpha$, $v(\alpha,0)=0$, and the vortices are
initially located at depth near $H_0$ with a downward leading velocity. Denote $z_0(\alpha):=z(\alpha,0)$, $\xi_0(\alpha):=z_0(\alpha)-\alpha$, and $v_0(\alpha):=z_t(\alpha,0)$. In the region $\Omega(t)$, the vorticity is caused solely by point vortices, i.e., $(I-\mathcal{H})(\bar z_t-\mathfrak{q})=0$ where $\mathfrak{q}$ is the velocity induced by point vortices and
$$
\mathfrak{q}=-\frac{\lambda i}{2\pi}
    \frac{z_1(t)-z_2(t)}
    {(z(\alpha,t)-z_1(t))(z(\alpha,t)-z_2(t))}.
$$
The notation $\mathcal{H}$ stands for the Hilbert transform associated with $z$ whose definition is in \eqref{H_z_def}. We now make some assumptions about the initial conditions.
\subsubsection{Initial data}
\begin{enumerate}
\item $z_0(\alpha)$ is a non‑self‑intersecting curve. We assume that\begin{equation}
    z_1(0) = (x_R, -H_0),\quad z_2(0) = (x_L, -H_0),\quad \norm{\operatorname{Im} z_0}_{L^\infty}\leq1
  \end{equation}
  where $x_L<0$ and $x_R>0$. The distance between two point vortices satisfies $|x_R-x_L|=1$. $H_0$ is sufficiently large to guarantee that the point vortices stay in $\Omega(0)$.
\item There exists a constant $c_T>0$ such that $a(\alpha,0)\geq c_T$. Therefore, we can guarantee that the Taylor-sign condition is satisfied and apply the local existence theorem.
\end{enumerate}
We state the main result under such notations:

\begin{thm}[Main theorem]\label{main}
  Let $s\in\mathbb{N}$, $s\geq 10$. There exist $\epsilon_0>0$ and $\lambda_0>0$ sufficiently small such that for all $0<\epsilon\leq\epsilon_0$ and $0<\lambda\leq\lambda_0$, if the initial data satisfy the condition above and
  \begin{equation}\label{ini_z}
    \norm{v_0}_{H^{s+1/2}}+\norm{z_0-\alpha}_{\dot{H}^{1/2}\cap \dot{H}^{s+1}} +\norm{\alpha\partial_{\alpha}v_0}_{H^{s-1/2}}+\norm{\alpha\partial_{\alpha}\partial_{\alpha}\xi_0}_{H^{s-1}}   \le \epsilon,
  \end{equation}
   the point vortex strengths and $H_0$ satisfy 
  \begin{equation}\label{lambda_H_0}
      \frac{\lambda}{c_sH_0^\frac{1}{2}}\leq\epsilon\leq c_s\lambda H_0^{\frac{1}{4}},\quad c_sH_0^{\frac{1}{2}}\geq1
  \end{equation}
  where $c_s$ is a small constant dependent on $s$, then the water wave system with point vortices is globally well-posed. Moreover,
  \begin{equation}\label{main_one}
    \sup_{t\ge 0}\bigl(\norm{z_t}_{H^{s+1/2}}+\norm{z-\alpha}_{\dot{H}^{1/2}}+ \norm{z_{\alpha}-1}_{H^s}+\norm{z_{tt}}_{H^s}\bigr)\leq C\epsilon,
\end{equation}
and for all $t\geq 1$,
\begin{equation}\label{main_two}
   \norm{z_{\alpha}(\cdot,t)-1}_{W^{s-2,\infty}}+ \norm{z_{tt}(\cdot,t)}_{W^{s-2,\infty}}+\norm{\partial_{\alpha}z_t}_{W^{s-3,\infty}}\leq C\epsilon \langle t\rangle^{-1/2}.
\end{equation}
  Furthermore, the distance between the point vortices and the free surface satisfies
  \begin{equation}\label{main_three}
    \operatorname{dist}(z_j(t), \Sigma(t)) \ge c\,(H_0+\lambda t),\qquad j=1,2,
  \end{equation}
  for some constants $c,C>0$ independent of $\epsilon$, $\lambda$, $H_0$ and $t$.
\end{thm}

\begin{rem}
    Indeed, (i) we need only $\dot{z}_1(0)\cdot (0,-1)\geq \lambda c_0>0$; (ii) $|x_R-x_L|=1$ is assumed for convenience, we can assume $|x_R-x_L|\sim 1$; (iii) More generally, we can assume $z_1(0)=(x_R, y_R)$, $z_2(0)=(x_L, y_L)$, with $\dot{z}_j(0)\cdot (0,-1)\geq \lambda c_0>0$, $j=1,2$.
\end{rem}

\subsection{Difficulties and main new ideas}\label{subsec:difficulties}
We highlight three points where the present problem differs from a direct
perturbative application of the irrotational global theory.

\smallskip
\noindent\emph{1. No low-frequency control of the velocity potential.}
Our initial assumption is formulated directly on the boundary velocity,
\[
      z_t(0)\in H^{s+\frac12},
\]
and we do not impose an additional low-frequency condition on a velocity
potential. This is weaker than the assumptions used in the global theories of
\cite{Wu2009,Ionescu2015,alazard2015global,HunterTataruIfrim1,HunterTataruIfrim2,ai2022two},
where one controls, in one form or another, a primitive of the boundary velocity,
for instance a velocity potential in $L^2$ or in a low homogeneous space
$\dot H^a$, $a<3/4$. Since $v=\nabla\phi$ and $z_t=v|_{\Sigma(t)}$, such a
condition imposes an extra assumption at low frequency which is not assumed here.

We illustrate the difficulty that would arise in doing so with a heuristic derivation. Assuming $Q_\alpha\sim z_t$, $Q$ is then a quantity comparable to a velocity potential. Let us introduce the vector fields
\[
\Omega_0=\alpha\partial_t+\frac{1}{2}ti,\ L_0=\frac{1}{2}t\partial_t+\alpha\partial_\alpha.
\]
By virtue of the identity
\[
\Omega_0Q_\alpha=L_0Q_t-\frac{t}{2}(\partial_t^2-i\partial_\alpha)Q,
\]
a direct bound for $\Omega_0z_t$ would require control of $L_0Q_t$, which involves a low-frequency primitive that is out of control. We avoid this difficulty by using and sharpening the
transition-of-derivatives method of \cite{su2025new}. Such a method allows us to transfer the vector-field derivative to a better
factor, to an ordinary derivative, or to the equation. Compared with
\cite{su2025new}, the present version is more systematic and includes fewer spatial
cutoffs. It is particularly suited to the non-convolution singular integrals
coming from the moving Hilbert transform and from the point vortices.

\smallskip
\noindent\emph{2. The vortex field is not small merely because the vortices retreat.}
It is tempting to regard the point vortices as a small time-integrable external
forcing. This is not immediate in the natural regime $H_0\sim1$ and
$\lambda\sim\epsilon$. Then the initial boundary field satisfies only
\[
     \|q(0)\|_{H^s}+\|D_tq(0)\|_{H^s}\lesssim\epsilon,
\]
and the dipole retreats with speed of size $\lambda/(4\pi)\sim\epsilon$. Hence
on the interval $0\le t\lesssim\epsilon^{-1}$ the depth has increased only by
an amount comparable to one, and a naive estimate gives
\[
     \int_0^{\epsilon^{-1}}
     \bigl(\|q(t)\|_{H^s}+\|D_tq(t)\|_{H^s}\bigr)\,dt\sim1.
\]
Such a bound is insufficient to preserve the small-data regime for long-time existence.

To overcome this difficulty, we exploit the cancellation structure inside the coupled system. Our key observation is that, after embedding the holomorphicity constraint into the good unknown, the equation is not forced by $q$ itself. Rather, it takes the form
\[
(D_t^2-iA\partial_\alpha)\theta=G_c+G_d,
\]
where $G_c$ is the cubic irrotational nonlinearity, and the leading vortex contribution in $G_d$ appears as $-4D_tq$. This reveals that the physically relevant quantity is not the dipole field itself, but its material variation. By a precise use of the point-vortex ODE, the dipole cancellation, and the retreat estimate, we can establish the integrable bound
\[
      \int_0^\infty \|D_tq(t)\|_{H^s}\,dt\lesssim\epsilon.
\]
It is precisely this estimate that allows us to treat the vortex contribution perturbatively.

\smallskip
\noindent\emph{3. The vector-field vortex term requires a holomorphic estimate.}
To obtain the $t^{-1/2}$ decay, we need Klainerman--Sobolev type bounds and
therefore vector-field energies such as $\|L_0D_t\theta\|_{H^{s-1}}$. The
corresponding vortex term is $L_0D_tq$. A direct boundary estimate is too weak;
in the same natural regime it gives an $O(1)$ cumulative contribution up to
time $\epsilon^{-1}$.

We overcome this by using the analytic structure inside the fluid. Since
\[
      \mathfrak{F}=D_t\bar\zeta-q
\]
is holomorphic in the fluid domain, we extend $L_0\mathfrak{F}$ into $\Omega(t)$ and
estimate the associated holomorphic velocity at the vortex locations. This
gives a sharper bound for the regular part of the vortex ODE and hence for
$L_0D_tq$. In the wave-packet decay estimate we also isolate the principal part
of $b$ by introducing an auxiliary coefficient $\hat b$ with better derivative
decay; and in the high-order commutators we decompose
\[
      b=\frac12\bigl[(I-\mathfrak{H})b+(I-\bar{\mathfrak{H}})b
        +(\mathfrak{H}+\bar{\mathfrak{H}})b\bigr].
\]
These refinements close the vector-field and pointwise decay estimates without
assuming that $H_0$ is artificially large.

\subsection{Overall strategy of the proof}\label{subsec:overall-strategy}
The proof follows the physical-space framework established in \cite{su2025new}, supplemented with novel estimates to handle the point-vortex contributions. According to the local well-posedness theorem in \cite{su2020long}, it suffices to control $\norm{(z_t,z_{tt})}_{H^{s+\frac{1}{2}}\times H^s}$, $\sup_{\alpha\neq\beta}\left|\frac{\alpha-\beta}{z(\alpha,t)-z(\beta,t)}\right|$, $\sup_{\alpha\neq\beta}\left|\frac{z(\alpha,t)-z(\beta,t)}{\alpha-\beta}\right|$, $\operatorname{dist}(\Gamma(t),z_1(t))^{-1}$, $\operatorname{dist}(\Gamma(t),z_2(t))^{-1}$, $\operatorname{dist}(z_1(t),z_2(t))^{-1}$ and guarantee that the Taylor-sign condition is satisfied. We begin by isolating the singular vortex velocity from the holomorphic remainder. Specifically, in the dipole case, we write
\[
    q(\alpha,t)=
    -\frac{\lambda i}{2\pi}
    \frac{z_1(t)-z_2(t)}
    {(\zeta(\alpha,t)-z_1(t))(\zeta(\alpha,t)-z_2(t))},
    \qquad
    \mathfrak{F}=D_t\bar\zeta-q,
\]
where $\mathfrak{F}$ represents the boundary trace of a function holomorphic in the fluid domain. This splitting is consistently utilized in both the boundary governing equation and the point-vortex ODE system.

Next, we introduce Wu's good unknowns
\[
      \tilde\theta=(I-\mathfrak H)(\zeta-\bar\zeta),
      \qquad \tilde\sigma=D_t\tilde\theta.
\]
By virtue of Wu's modified coordinates and the commutator identities for the Hilbert transform, we derive the structural equation
\[
      (D_t^2-iA\partial_\alpha)\tilde\theta=G_c+G_d.
\]
Here, the irrotational nonlinearity $G_c$ enjoys the same favorable cubic structure as in \cite{su2025new}, whereas the vortex contribution $G_d$ is managed by exploiting the time-integrability of $D_tq$.

The energy analysis consists of three interconnected components. First, the retreating dipole estimate ensures that the vortices remain at a distance $H(t)\gtrsim H_0+\lambda t$, which yields the crucial integrable bound for $D_tq$. Second, the high-order Sobolev and vector-field energies are successfully closed by utilizing the cubic structure and the transition-of-derivatives method, thereby avoiding direct control of the problematic terms. Third, the pointwise decay is established via a refined wave-packet energy estimate; the sole non-perturbative contribution arising from the top-order coefficients is handled through a delicate analysis of the advection-like terms $b$ and $\hat b$.

These bounds ultimately allow us to verify and strengthen the bootstrap assumptions concerning the chord-arc geometry, the Taylor coefficient, the high-order Sobolev and vector-field energies, the $t^{-1/2}$ decay norm, and the retreat rate of the vortices. A standard continuity argument then establishes the global-in-time existence of the solution and yields the uniform bounds stated in Theorem~\ref{main}.

\subsection{Outline of the paper}
Section 2 recalls the formulation of water waves with point vortices,
introduces Wu's modified coordinates, and records the Hilbert-transform and
analytic estimates used later. Section 3 sets up the bootstrap assumptions.
Section 4 derives the basic consequences of the bootstrap assumptions,
including the point-vortex estimates, the bounds for $q$, and the estimates
for the coefficients $b$ and $A$. Section 5 proves the high-order energy and
vector-field estimates, including the refined commutator and localization
arguments. Section 6 proves the pointwise decay estimates based on the
oscillatory energy and completes the proof of Theorem \ref{main}.

\subsection{Notations}
We follow the notation of \cite{Wu2009,su2020long}. The Hilbert transform
associated with the fluid domain is denoted by $\mathfrak H$, and
$D_t=\partial_t+b\partial_\alpha$ is the material derivative in Wu's modified
coordinates. The spaces $\dot H^s$ and $H^s$ are the homogeneous and
inhomogeneous Sobolev spaces. We write $f=O(A)_{L^2}$ to mean
$\|f\|_{L^2}\le CA$, and similarly for other norms. Throughout the article, we let $C$ denote a large positive constant independent of $t$, $\lambda$, $H$, $H_0$, $\epsilon$ and let $\tilde{C}$ denote a large positive constant independent of $t$, $\lambda$, $H$, $H_0$, $\epsilon$ and $M$. The constants $C$ and $\tilde{C}$ may take different values from line to line. The notation $A\lesssim B$ means that $A\leq CB$ for some $C$.

\section{Preliminaries and formulation}\label{sec_pre}

\subsection{Geometry, chord-arc curves, and Hilbert transform}

\begin{defn}[Hilbert transform]
Assume that $z(\alpha)$ satisfies
\begin{equation}\label{chordchordarcarc}
    \beta_0|\alpha-\beta|\leq |z(\alpha)-z(\beta)|\leq \beta_1|\alpha-\beta|, \quad \quad \forall \alpha,\beta\in \mathbb{R},
\end{equation}
where $0<\beta_0<\beta_1<\infty$ are two absolute constants.
We define the Hilbert transform associated to a curve $z(\alpha)$ as 
\begin{equation}\label{H_z_def}
\mathcal{H}f(\alpha):=\frac{1}{\pi i}p.v.\int_{-\infty}^{\infty}\frac{z_{\beta}(\beta)}{z(\alpha)-z(\beta)}f(\beta)d\beta.
\end{equation}
We use the notation $\mathfrak{H}$ to denote the Hilbert transform associated with $\zeta$:
\begin{equation}\label{H_def}\mathfrak{H}f(\alpha,t):=\frac{1}{\pi i}p.v.\int_{-\infty}^{\infty}\frac{\zeta_{\beta}}{\zeta(\alpha,t)-\zeta(\beta,t)}f(\beta,t)d\beta.
\end{equation}
The standard Hilbert transform is the Hilbert transform associated with $z(\alpha)=\alpha$, which is denoted by
\begin{equation}
\mathbb{H}f(\alpha):=\frac{1}{\pi i}p.v.\int_{-\infty}^{\infty}\frac{1}{\alpha-\beta}f(\beta)d\beta.
\end{equation}
\end{defn}
It's well-known (see \cite{david1984operateurs}, Theorem 6) that if $\zeta(\alpha)$ satisfies (\ref{chordchordarcarc}), then $\mathfrak{H}$ is bounded on $L^2$.
\begin{lemma}\label{boundednesshilbert}
Assume that $\zeta(\alpha)$ satisfies (\ref{chordchordarcarc}), then 
\begin{equation}
    \|\mathfrak{H}f\|_{L^2}\leq C\|f\|_{L^2},
\end{equation}
for some constant that depends on $\beta_0$ and $\beta_1$ only.
\end{lemma}
We can use the Hilbert transform to characterize the boundary value of holomorphic functions. 
\begin{lemma}\label{holomorphic}
Let $f\in L^2(\mathbb{R})$. Then $f$ is the boundary value of a holomorphic function in $\Omega(t)$ if and only if $(I-\mathfrak{H})f=0$. $f$ is the boundary value of a holomorphic function in $\Omega(t)^c$ if and only if $(I+\mathfrak{H})f=0$. 
\end{lemma}

\subsection{Wu's modified coordinates and basic unknowns}

We now introduce the modified Lagrangian coordinates of Wu. The purpose of
this change of variables is to remove the leading quadratic terms in the
water wave system and to reveal the cubic structure of the two-dimensional
gravity wave equation. In the present paper the same coordinate change is
used, while the point vortices generate additional lower-order perturbative
terms.

Let \(z=z(\alpha,t)\) be the Lagrangian parametrization of the free surface.
Following Wu, we introduce a time-dependent change of variables
\[
    \kappa=\kappa(\alpha,t)
\]
satisfying
$$
\kappa+\Phi\circ z=z+\bar z
$$
where $\Phi$ is a holomorphic function defined on the region $\Omega(t)$ 
and define
\[
    \zeta(\alpha,t):=z(\kappa^{-1}(\alpha,t),t).
\]
All quantities in the modified coordinates will be written as functions of
\((\alpha,t)\). The material derivative in the new coordinates is
\[
    D_t:=\partial_t+b\partial_\alpha,
\]
where
\[
    b(\alpha,t):=\kappa_t(\kappa^{-1}(\alpha,t),t).
\]
The rescaled Taylor coefficient is denoted by
\[
    A(\alpha,t):=(a\kappa_\alpha)(\kappa^{-1}(\alpha,t),t).
\]
With these definitions, the free boundary equation becomes
\begin{equation}\label{main_equation}
    D_t^2\zeta-iA\zeta_\alpha=-i.
\end{equation}
Here \(A\) is real-valued, and the Taylor sign condition corresponds to
\(A>0\). For small perturbations of the flat interface one has
\[
    A=1+O(\epsilon),\qquad b=O(\epsilon),
\]
in the norms used below.

The Hilbert transform associated with the curve \(\zeta(\cdot,t)\) is
\[
    \mathfrak{H}f(\alpha,t)
    =
    \frac{1}{\pi i}\,{\rm p.v.}
    \int_{\mathbb R}
    \frac{\zeta_\beta(\beta,t)}
    {\zeta(\alpha,t)-\zeta(\beta,t)}
    f(\beta,t)\,d\beta.
\]
In the modified coordinates the holomorphicity constraint takes the form
\[
    (I-\mathfrak{H})\mathfrak F=0,
\]
where
\[
    \mathfrak F(\alpha,t):=D_t\bar\zeta(\alpha,t)-q(\alpha,t).
\]
Here \(q\) denotes the boundary value of the singular velocity generated by
the point-vortex pair. In the dipole case
\[
    \lambda_1=\lambda,\qquad \lambda_2=-\lambda,
\]
we write
\[
    q(\alpha,t)
    :=
    -\frac{\lambda i}{2\pi}
    \frac{z_1(t)-z_2(t)}
    {(\zeta(\alpha,t)-z_1(t))(\zeta(\alpha,t)-z_2(t))}.
\]
Thus \(\mathfrak F\) is the boundary trace of a holomorphic function in the fluid
domain. We shall use the same letter \(\mathfrak U\) for its holomorphic extension:
\[
    \mathfrak U(\zeta(\alpha,t),t)=\mathfrak F(\alpha,t).
\]
Equivalently,
\[
    \mathfrak U(\zeta(\alpha,t),t)
    =
    D_t\bar\zeta(\alpha,t)-q(\alpha,t).
\]

The quantities \(b\) and \(A\) can be characterized by applying the projection
\(I-\mathfrak{H}\) to the equations above. First, since \(D_t\bar\zeta-q\) is
holomorphic, one has
\[
    (I-\mathfrak{H})(D_t\bar\zeta-q)=0.
\]
This identity may be used to determine \(b\), because
\[
    D_t\bar\zeta
    =
    \bar\zeta_t+b\bar\zeta_\alpha.
\]
Equivalently,
\[
    (I-\mathfrak{H})b
    =
    -(I-\mathfrak{H})\frac{\bar\zeta_t-q}{\bar\zeta_\alpha}
    + \text{lower order commutator terms},
\]
or, in the perturbative regime \(\zeta_\alpha\approx1\),
\[
    b
    =
    O(D_t\zeta(\zeta_\alpha-1))+O(q).
\]
The precise estimates for \(b\) will be derived later.

Similarly, applying \(I-\mathfrak{H}\) to the equation
\[
    D_t^2\zeta-iA\zeta_\alpha=-i
\]
and using the holomorphicity of the appropriate boundary traces gives an
equation for \(A-1\). In the small-data regime this yields schematically
\[
    A-1
    =
    O((D_t\zeta)^2)+O(D_tq)+O(qD_t\zeta),
\]
again in the norms used below. The exact bounds for \(A-1\) are postponed to
Section~4.

We next introduce the main good unknowns. Define
\[
    \tilde\theta:=(I-\mathfrak{H})(\zeta-\bar\zeta),
\]
and
\[
    \tilde\sigma:=D_t\tilde\theta.
\]
Since \((I-\mathfrak{H})(\bar\zeta-\alpha)=0\) and
\(\zeta-\alpha\) is small, the quantity \(\tilde\theta\) captures the
non-holomorphic part of the interface. More precisely, it is equivalent to
the perturbation of the interface in the sense that
\[
    \tilde\theta_\alpha
    =
    2(\zeta_\alpha-1)
    +
    \text{quadratic terms}.
\]
Likewise,
\[
    D_t\tilde\theta
    =
    2D_t\zeta
    +
    \text{quadratic terms}
    +
    O(q),
\]
and
\[
    D_t\tilde\sigma
    =
    2D_t^2\zeta
    +
    \text{quadratic terms}
    +
    O(D_tq).
\]
These equivalences will be made quantitative in Section~4.

The advantage of \(\tilde\theta\) is that it satisfies a quasilinear dispersive
equation whose leading nonlinear part is cubic. More precisely,
\[
    (D_t^2-iA\partial_\alpha)\tilde\theta
    =
    G_c+G_d,
\]
where \(G_c\) denotes the cubic term already present in the irrotational
gravity water wave equation, and \(G_d\) denotes the contribution of the point
vortices. In the present notation,
\[
\begin{aligned}
    G_c
    &=
    -2\left[\bar {\mathfrak {F}},
        \mathfrak{H}\frac{1}{\zeta_\alpha}
        +\bar {\mathfrak{H}}\frac{1}{\bar\zeta_\alpha}
      \right]\bar {\mathfrak{F}}_\alpha
   +\frac{1}{\pi i}
    \int_{\mathbb R}
    \left(
        \frac{D_t\zeta(\alpha,t)-D_t\zeta(\beta,t)}
        {\zeta(\alpha,t)-\zeta(\beta,t)}
    \right)^2
    (\zeta-\bar\zeta)_\beta(\beta,t)\,d\beta,
\end{aligned}
\]
while the point-vortex contribution is
\[
    G_d
    =
    -2[\bar q,\mathfrak{H}]\frac{\bar {\mathfrak{F}}_\alpha}{\zeta_\alpha}
    -2[\bar {\mathfrak{F}},\mathfrak{H}]\frac{\bar q_\alpha}{\zeta_\alpha}
    -2[\bar q,\mathfrak{H}]\frac{\bar q_\alpha}{\zeta_\alpha}
    -4D_tq.
\]
Here \([f,\mathfrak{H}]g=f\mathfrak{H}g-\mathfrak{H}(fg)\). The term \(G_c\) has the same cubic
structure as in the irrotational problem. The term \(G_d\) is new and will be
treated perturbatively using the decay of \(q\), which follows from the
retreat of the point vortices from the free boundary.

For the weighted energy estimates we also use the vector fields
\[
    L_0:=\frac12 t\partial_t+\alpha\partial_\alpha,
\]
and
\[
    \Omega_0:=\alpha\partial_t+\frac12 ti.
\]
They are adapted to the linearized gravity wave operator
\[
    \partial_t^2-i\partial_\alpha.
\]
In the nonlinear problem they do not commute exactly with
\(D_t^2-iA\partial_\alpha\), and the resulting commutator terms will be
controlled by the transition-of-derivatives argument. The basic energy
variables used later are
\[
    \tilde\theta,\quad D_t\tilde\theta,\quad
    \tilde\sigma,\quad D_t\tilde\sigma,
\]
together with their \(L_0\)- and \(\Omega_0\)-derivatives.

Finally, throughout the rest of the paper we use the shorthand
\[
    \xi:=\zeta-\alpha.
\]
Thus smallness of the interface will be measured by Sobolev norms of
\(\xi\), \(\zeta_\alpha-1\), \(D_t\zeta\), and \(D_t^2\zeta\). The equivalence
between these quantities and the good unknowns \(\tilde\theta\),
\(D_t\tilde\theta\), and \(D_t\tilde\sigma\) is one of the main elementary
steps in the energy argument.

\subsection{Commutator identities}

In this subsection we collect several commutator identities for the Hilbert
transform associated with the moving interface. These identities are used
throughout the paper to derive the equations for the good unknowns and to
estimate the coefficients \(b\) and \(A\).

Recall that
\[
    \mathfrak H f(\alpha,t)
    =
    \frac{1}{\pi i}\,{\rm p.v.}
    \int_{\mathbb R}
    \frac{\zeta_\beta(\beta,t)}
    {\zeta(\alpha,t)-\zeta(\beta,t)}
    f(\beta,t)\,d\beta.
\]
We also write
\[
    D_t:=\partial_t+b\partial_\alpha.
\]
The following identities are understood in the principal value sense. They
are first verified for smooth rapidly decaying functions and then extended
by density to the Sobolev spaces used below.

For any sufficiently regular function \(f\), one has
\[
    [\partial_t,\mathfrak H]f
    =
    [\zeta_t,\mathfrak H]\frac{f_\alpha}{\zeta_\alpha}.
\]
Equivalently,
\[
    \partial_t(\mathfrak H f)
    =
    \mathfrak H f_t
    +
    [\zeta_t,\mathfrak H]\frac{f_\alpha}{\zeta_\alpha}.
\]
Similarly, for the material derivative \(D_t\),
\[
    [D_t,\mathfrak H]f
    =
    [D_t\zeta,\mathfrak H]\frac{f_\alpha}{\zeta_\alpha}.
\]
That is,
\[
    D_t(\mathfrak H f)
    =
    \mathfrak H D_t f
    +
    [D_t\zeta,\mathfrak H]\frac{f_\alpha}{\zeta_\alpha}.
\]

We shall also need the second material derivative commutator. A direct
calculation gives
\[
\begin{aligned}
    [D_t^2,\mathfrak H]f
    &=
    [D_t^2\zeta,\mathfrak H]\frac{f_\alpha}{\zeta_\alpha}
    +
    2[D_t\zeta,\mathfrak H]\frac{\partial_\alpha D_t f}{\zeta_\alpha}
    \\
    &\quad
    -
    \frac{1}{\pi i}
    \int_{\mathbb R}
    \left(
        \frac{D_t\zeta(\alpha,t)-D_t\zeta(\beta,t)}
             {\zeta(\alpha,t)-\zeta(\beta,t)}
    \right)^2
    f_\beta(\beta,t)\,d\beta .
\end{aligned}
\]
Equivalently,
\[
\begin{aligned}
    D_t^2(\mathfrak H f)
    &=
    \mathfrak H D_t^2 f
    +
    [D_t^2\zeta,\mathfrak H]\frac{f_\alpha}{\zeta_\alpha}
    +
    2[D_t\zeta,\mathfrak H]\frac{\partial_\alpha D_t f}{\zeta_\alpha}
    \\
    &\quad
    -
    \frac{1}{\pi i}
    \int_{\mathbb R}
    \left(
        \frac{D_t\zeta(\alpha,t)-D_t\zeta(\beta,t)}
             {\zeta(\alpha,t)-\zeta(\beta,t)}
    \right)^2
    f_\beta(\beta,t)\,d\beta .
\end{aligned}
\]

The spatial derivative satisfies
\[
    [\partial_\alpha,\mathfrak H]f
    =
    [\zeta_\alpha,\mathfrak H]\frac{f_\alpha}{\zeta_\alpha}.
\]
More explicitly,
\[
    [\zeta_\alpha,\mathfrak H]\frac{f_\alpha}{\zeta_\alpha}
    =
    \frac{1}{\pi i}\,{\rm p.v.}
    \int_{\mathbb R}
    \left(
        \frac{\zeta_\alpha(\alpha,t)-\zeta_\beta(\beta,t)}
             {\zeta(\alpha,t)-\zeta(\beta,t)}
    \right)
    f_\beta(\beta,t)\,d\beta .
\]
Thus this commutator is at least quadratic in the perturbative regime
\(\zeta_\alpha\approx1\).

We shall frequently use the explicit formula
\[
    [g,\mathfrak H]f(\alpha)
    =
    \frac{1}{\pi i}\,{\rm p.v.}
    \int_{\mathbb R}
    \frac{g(\alpha)-g(\beta)}
         {\zeta(\alpha)-\zeta(\beta)}
    f(\beta)\zeta_\beta(\beta)\,d\beta .
\]
In particular,
\[
    [g,\mathfrak H]\frac{f_\alpha}{\zeta_\alpha}(\alpha)
    =
    \frac{1}{\pi i}\,{\rm p.v.}
    \int_{\mathbb R}
    \frac{g(\alpha)-g(\beta)}
         {\zeta(\alpha)-\zeta(\beta)}
    f_\beta(\beta)\,d\beta .
\]
This form is useful because the numerator \(g(\alpha)-g(\beta)\) provides
one cancellation.

We also record the corresponding identities for the conjugate Hilbert
transform \(\overline{\mathfrak H}\), defined by
\[
    \overline{\mathfrak H}f(\alpha,t)
    =
    -\frac{1}{\pi i}\,{\rm p.v.}
    \int_{\mathbb R}
    \frac{\bar\zeta_\beta(\beta,t)}
         {\bar\zeta(\alpha,t)-\bar\zeta(\beta,t)}
    f(\beta,t)\,d\beta .
\]
With this convention,
\[
    [D_t,\overline{\mathfrak H}]f
    =
    [D_t\bar\zeta,\overline{\mathfrak H}]
    \frac{f_\alpha}{\bar\zeta_\alpha}.
\]
Moreover,
\[
\begin{aligned}
    [D_t^2,\overline{\mathfrak H}]f
    &=
    [D_t^2\bar\zeta,\overline{\mathfrak H}]
    \frac{f_\alpha}{\bar\zeta_\alpha}
    +
    2[D_t\bar\zeta,\overline{\mathfrak H}]
    \frac{\partial_\alpha D_t f}{\bar\zeta_\alpha}
    \\
    &\quad
    +
    \frac{1}{\pi i}
    \int_{\mathbb R}
    \left(
        \frac{D_t\bar\zeta(\alpha,t)-D_t\bar\zeta(\beta,t)}
             {\bar\zeta(\alpha,t)-\bar\zeta(\beta,t)}
    \right)^2
    f_\beta(\beta,t)\,d\beta .
\end{aligned}
\]
The sign in the last integral is opposite to the one for \(\mathfrak H\),
because of the convention used in the definition of
\(\overline{\mathfrak H}\).

We next state the commutators with the scaling vector field
\[
    L_0:=\frac12 t\partial_t+\alpha\partial_\alpha .
\]
Since \(\mathfrak H\) depends on the curve \(\zeta\), \(L_0\) does not commute
with \(\mathfrak H\). One has
\[
    [L_0,\mathfrak H]f
    =
    [L_0\zeta,\mathfrak H]\frac{f_\alpha}{\zeta_\alpha}.
\]
Equivalently,
\[
    L_0(\mathfrak H f)
    =
    \mathfrak H L_0 f
    +
    [L_0\zeta,\mathfrak H]\frac{f_\alpha}{\zeta_\alpha}.
\]
Similarly,
\[
    [L_0,\overline{\mathfrak H}]f
    =
    [L_0\bar\zeta,\overline{\mathfrak H}]
    \frac{f_\alpha}{\bar\zeta_\alpha}.
\]
These identities will be combined with the transition-of-derivatives
argument in the vector-field estimates.

Finally, for any operator \(T\) we use the notation
\[
    [f,T]g:=fTg-T(fg).
\]
Therefore
\[
    T(fg)=fTg-[f,T]g.
\]
In particular,
\[
    (I-\mathfrak H)(fg)
    =
    f(I-\mathfrak H)g+[f,\mathfrak H]g.
\]
When \(g\) is the boundary value of a holomorphic function in the fluid
domain, \((I-\mathfrak H)g=0\), and hence
\[
    (I-\mathfrak H)(fg)=[f,\mathfrak H]g.
\]

\subsection{Analytic estimates away from the boundary}

In this subsection we record several elementary estimates for holomorphic
functions evaluated at points away from the free boundary. These estimates
will be used repeatedly to control the influence of the holomorphic
remainder of the velocity field on the point vortices.

Let \(\Omega(t)\) be the fluid domain bounded above by the chord-arc curve
\(\Sigma(t)=\zeta(\mathbb R,t)\). For \(z\in\Omega(t)\), we write
\[
    d(z,t):=\operatorname{dist}(z,\Sigma(t)).
\]
For the point vortices we use the notation
\[
    H_j(t):=\operatorname{dist}(z_j(t),\Sigma(t)),
    \qquad
    H(t):=\min_{z\in[z_1,z_2]}d(z,t)
\]
where $[z_1,z_2]$ stands for the line segment connecting $z_1$ and $z_2$. 
We also write
\[
    d_{12}(t):=|z_1(t)-z_2(t)|.
\]
In the bootstrap argument the relevant regime is
\[
    d_{12}(t)\approx d_{12}(0)\approx 1,
    \qquad
    H(t)\gtrsim H_0+\lambda t.
\]

We first state a Cauchy-type estimate for holomorphic functions in the fluid
domain.

\begin{lemma}[Cauchy estimates away from the boundary]
Let \(\mathfrak{U}(\cdot,t)\) be holomorphic in \(\Omega(t)\), and let
\[
    \mathfrak F(\alpha,t):=\mathfrak U(\zeta(\alpha,t),t)
\]
be its boundary trace. Assume that \(\mathfrak F(\cdot,t)\in L^2(\mathbb R)\). Then for
every integer \(k\ge 0\) and every \(z\in\Omega(t)\),
\[
    |\partial_z^k \mathfrak U(z,t)|
    \lesssim
    d(z,t)^{-k-\frac12}\|\mathfrak F(\cdot,t)\|_{L^2}.
\]
The implicit constant depends only on the chord-arc constants of the curve
\(\zeta(\cdot,t)\).
\end{lemma}

\begin{proof}
By the Cauchy representation formula for the lower chord-arc domain,
\[
    \partial_z^k \mathfrak U(z,t)
    =
    c_k
    \int_{\mathbb R}
    \frac{\mathfrak F(\beta,t)\zeta_\beta(\beta,t)}
         {(\zeta(\beta,t)-z)^{k+1}}\,d\beta .
\]
Hence, by Cauchy--Schwarz,
\[
    |\partial_z^k \mathfrak U(z,t)|
    \lesssim
    \|\mathfrak F(\cdot,t)\|_{L^2}
    \left\|
        \frac{1}{|\zeta(\cdot,t)-z|^{k+1}}
    \right\|_{L^2_\beta}.
\]
Let \(z_0=\zeta(\alpha_0,t)\) be a point on \(\Sigma(t)\) such that
\[
    |z-z_0|=d(z,t).
\]
Using the chord-arc condition, we have
\[
    |\zeta(\beta,t)-z|
    \gtrsim d(z,t)+|\beta-\alpha_0|.
\]
Therefore
\[
    \int_{\mathbb R}
    \frac{d\beta}
         {|\zeta(\beta,t)-z|^{2k+2}}
    \lesssim
    \int_{\mathbb R}
    \frac{d\beta}
         {(d(z,t)+|\beta-\alpha_0|)^{2k+2}}
    \lesssim
    d(z,t)^{-2k-1}.
\]
Taking the square root gives the desired bound.
\end{proof}

As an immediate consequence, if \(z\) belongs to the line segment joining the
two point vortices and the whole segment stays at distance at least \(H(t)\)
from the free boundary, then
\[
    |\partial_z^k \mathfrak U(z,t)|
    \lesssim
    H(t)^{-k-\frac12}\|\mathfrak F(t)\|_{L^2}.
\]
In particular,
\[
    |\mathfrak U_z(z,t)|
    \lesssim H(t)^{-\frac32}\|\mathfrak F(t)\|_{L^2},
    \qquad
    |\mathfrak U_{zz}(z,t)|
    \lesssim H(t)^{-\frac52}\|\mathfrak F(t)\|_{L^2}.
\]

We shall also use the same estimate for time derivatives of the holomorphic
remainder. If \(\mathfrak U_t(\cdot,t)\) is holomorphic in \(\Omega(t)\) and has
boundary trace \(\mathfrak F_t^{\rm hol}\), then
\[
    |\partial_z^k \mathfrak U_t(z,t)|
    \lesssim
    d(z,t)^{-k-\frac12}\|\mathfrak F_t^{\rm hol}(t)\|_{L^2}.
\]
Similarly,
\[
    |\partial_z^k \mathfrak U_{tt}(z,t)|
    \lesssim
    d(z,t)^{-k-\frac12}\|\mathfrak F_{tt}^{\rm hol}(t)\|_{L^2},
\]
provided the corresponding boundary traces are well defined.

We next record the elementary estimates for the vortex-induced boundary
velocity. Recall that in the dipole case
\[
    q(\alpha,t)
    =
    -\frac{\lambda i}{2\pi}
    \frac{z_1(t)-z_2(t)}
    {(\zeta(\alpha,t)-z_1(t))(\zeta(\alpha,t)-z_2(t))}.
\]
The following bounds express the fact that \(q\) gains inverse powers of the
distance from the point vortices to the free boundary.

\begin{lemma}[Estimates for the dipole field]
Assume that
\[
    d_{12}(t)=|z_1(t)-z_2(t)|\approx 1,
    \qquad
    H(t)=\min_{z\in[z_1,z_2]}\operatorname{dist}(z,\Sigma(t))\gg1.
\]
For every integer \(s\ge0\),
\[
    \|q(t)\|_{H^s}
    \lesssim
    \frac{\lambda}{H(t)^{\frac32}},
    \qquad
    \|q(t)\|_{W^{s,\infty}}
    \lesssim
    \frac{\lambda}{H(t)^2},
\]
where the constants may depend on \(s\) and on the chord-arc constants, but
not on \(H(t)\), \(\lambda\), or \(t\).
\end{lemma}

\begin{proof}
We only sketch the proof, since the same argument will be used several times
below. Differentiating \(q\) in \(\alpha\), every term is a linear combination
of expressions of the form
\[
    \lambda\,
    \frac{
        P(\zeta_\alpha,\ldots,\partial_\alpha^m\zeta)
    }
    {(\zeta-z_1)^{a}(\zeta-z_2)^{b}},
\]
where \(a+b\ge 2\), and \(P\) is a polynomial in derivatives of \(\zeta\).
Under the small chord-arc perturbation assumption, the derivatives of
\(\zeta\) appearing in \(P\) are bounded in the norms considered here.

For the \(L^\infty\) bound, since
\[
    |\zeta(\alpha,t)-z_j(t)|\ge H(t),
\]
we immediately get
\[
    |\partial_\alpha^m q(\alpha,t)|
    \lesssim
    \lambda H(t)^{-2}.
\]
For the \(L^2\) bound, choose a nearest point
\(\zeta(\alpha_j,t)\) to \(z_j(t)\). The chord-arc condition gives
\[
    |\zeta(\alpha,t)-z_j(t)|
    \gtrsim H(t)+|\alpha-\alpha_j|.
\]
Thus, for \(p>1\),
\[
    \int_{\mathbb R}
    \frac{d\alpha}{|\zeta(\alpha,t)-z_j(t)|^p}
    \lesssim
    H(t)^{1-p}.
\]
Applying this estimate to the differentiated expression for \(q\) gives
\[
    \|\partial_\alpha^m q(t)\|_{L^2}
    \lesssim
    \lambda H(t)^{-\frac32}.
\]
This proves the lemma.
\end{proof}

We shall frequently use the preceding lemma in the following form. If, under
the bootstrap assumptions,
\[
    H(t)\gtrsim H_0+\lambda t,
\]
then
\[
    \|q(t)\|_{L^\infty}
    \lesssim
    \frac{\lambda}{(H_0+\lambda t)^2},
    \qquad
    \|q(t)\|_{L^2}
    \lesssim
    \frac{\lambda}{(H_0+\lambda t)^{3/2}}.
\]
In particular,
\[
    \int_0^\infty
    \frac{\lambda}{(H_0+\lambda t)^2}\,dt
    \lesssim
    H_0^{-1}.
\]
This integrability is the basic mechanism by which the vortex contribution
is treated as a perturbative forcing term in the energy estimates.

\section{Bootstrap assumptions}

In this section we formulate the bootstrap assumptions used in the proof of
the main theorem. Let \(s\ge 10\) be fixed, and let \(0<\delta_0\ll1\). Let
\(T>1\). We assume that a smooth solution exists on \([0,T]\), that the point
vortices remain inside the fluid domain, and that the following estimates
hold on \([0,T]\). The case $0\leq t\leq 1$ is handled by the local well-posedness theory. Throughout the rest of this paper, we suppose that $t>1$ and write
\[
    \xi:=\zeta-\alpha,
    \qquad
    H_j(t):=\operatorname{dist}(z_j(t),\Sigma(t)),
    \qquad
    H(t):=\min_{z\in[z_1(t),z_2(t)]}\operatorname{dist}(z,\Sigma(t))
\]
where $[z_1(t),z_2(t)]$ represents the line segment connecting $z_1$ and $z_2$, 
and
\[
    d_{12}(t):=|z_1(t)-z_2(t)|.
\]
We also recall that
\[
    D_t=\partial_t+b\partial_\alpha.
\]
The energy quantities \(\mathcal E_s\), \(\mathcal E_{s,L}\), and \(\mathcal X_s\) are defined by the left-hand
sides of the estimates below.

\subsection{Bootstrap assumptions}

We fix a large constant \(M\gg1\), independent of
\[
    t,\quad \epsilon,\quad \lambda,\quad H_0,
\]
and assume that the following estimates hold for all \(t\in[0,T]\).

\paragraph{(B1) Chord-arc and Taylor sign bounds.}
The interface remains a small chord-arc perturbation of the flat interface:
\begin{equation}\label{ass_diff}
    \frac12|\alpha-\beta|
    \le
    |\zeta(\alpha,t)-\zeta(\beta,t)|
    \le
    2|\alpha-\beta|,
    \qquad \alpha,\beta\in\mathbb R.
\end{equation}
Moreover, the Taylor coefficient satisfies
\[
    \frac12\le A(\alpha,t)\le 2,
    \qquad \alpha\in\mathbb R.
\]

\paragraph{(B2) High-order Sobolev energy.}
The basic Sobolev energy satisfies
\[
\begin{aligned}
    \mathcal E_s^{1/2}(t):=
    &\,
    \|D_t\zeta(t)\|_{H^{s+\frac12}}
    +
    \|\zeta_\alpha(t)-1\|_{H^s}
    +
    \|D_t^2\zeta(t)\|_{H^s}
    \\
    &\le M\epsilon .
\end{aligned}
\]

\paragraph{(B3) Scaling vector-field energy.}
The \(L_0\)-weighted energy satisfies
\[
\begin{aligned}
    \mathcal E_{s,L}^{1/2}(t):=
    &\,
    \|L_0(\zeta_\alpha-1)(t)\|_{H^{s-1}}
    +
    \|L_0D_t\zeta(t)\|_{H^{s-1}}
    +
    \|L_0D_t^2\zeta(t)\|_{H^{s-1}}
    \\
    &\le M\epsilon(1+t)^{\delta_0}.
\end{aligned}
\]

\paragraph{(B4) Pointwise decay.}
The dispersive norm satisfies
\begin{equation}\label{boot_X_s}
\begin{aligned}
    \mathcal X_s(t):=
    &\,
    \|\zeta_\alpha(t)-1\|_{W^{s-2,\infty}}
    +
    \|D_t^2\zeta(t)\|_{W^{s-2,\infty}}
    +
    \|\partial_\alpha D_t\zeta(t)\|_{W^{s-3,\infty}}
    \\
    &\le M\epsilon(1+t)^{-1/2}.
\end{aligned}
\end{equation}

\paragraph{(B5) Separation of the point vortices.}
The two vortices remain separated from each other:
\[
    \frac12 d_{12}(0)
    \le
    d_{12}(t)
    \le
    2d_{12}(0),
\]
where
\[
    d_{12}(0)=|z_1(0)-z_2(0)|\approx1.
\]

\paragraph{(B6) Retreat of the point vortices from the interface.}
The point vortices stay away from the free surface:
\begin{equation}\label{boot_H}
    H(t)\ge c_0(H_0+\lambda t),
    \qquad t\in[0,T],
\end{equation}
for some fixed constant \(c_0\in(0,\frac{1}{8\pi})\).

The estimates for \(b\), \(A-1\), \(q\), \(D_tq\), \(L_0q\), and the point
vortex accelerations will be derived from (B1)--(B6) in the next section.

\subsection{Bootstrap improvement}

The goal of the rest of the paper is to improve the above bounds. More
precisely, assuming (B1)--(B6) on \([0,T]\), we shall prove that, provided
\(\epsilon>0\) is sufficiently small and \(H_0\) is sufficiently large in the
parameter regime of Theorem~1.1, there exists a constant \(c_1>c_0\) such
that
\begin{equation}\label{boot_B_2}
    \mathcal E_s^{1/2}(t)\le \frac{M}{2}\epsilon,
\end{equation}
\begin{equation}\label{boot_B_3}
    \mathcal E_{s,L}^{1/2}(t)
    \le \frac{M}{2}\epsilon(1+t)^{\delta_0},
\end{equation}
\begin{equation}\label{boot_B_4}
    \mathcal X_s(t)
    \le \frac{M}{2}\epsilon(1+t)^{-1/2},
\end{equation}
\begin{equation}\label{boot_B_5}
    \frac34 d_{12}(0)
    \le d_{12}(t)\le
    \frac32 d_{12}(0),
\end{equation}
and
\begin{equation}\label{boot_B_6}
    H(t)\ge c_1(H_0+\lambda t).
\end{equation}
By the standard continuity argument, these improved estimates close the
bootstrap and imply global existence.

\subsection{The closure of bootstrap assumptions}
Now we close the bootstrap assumptions under the results established in Section \ref{sec_b}, \ref{sec_energy} and \ref{sec_decay}
\begin{proof}
    We close the bootstrap assumptions step by step.\\
    (B1): It follows from (B4) that
    $$
    \left|\frac{\zeta(\alpha,t)-\zeta(\beta,t)}{\alpha-\beta}-1\right|=\left|\frac{1}{\alpha-\beta}\int_\alpha^\beta( \zeta_s(s,t)-1)ds\right|\leq\norm{\zeta_\alpha(\cdot,t)-1}_{L^\infty}\leq M\epsilon.
    $$
    From \eqref{boot_X_s} we deduce that if $M\epsilon<\frac{1}{4}$, then
    $$
    \frac{3}{4}|\alpha-\beta|
    \le
    |\zeta(\alpha,t)-\zeta(\beta,t)|
    \le
    \frac{5}{4}|\alpha-\beta|,
    \qquad \alpha,\beta\in\mathbb R.
    $$
    Moreover, from Lemma \ref{lem_A} we deduce that
    $$
    |A-1|\leq C\epsilon^2+C\frac{\lambda(\epsilon+\lambda)}{H_0^2}\leq\frac{1}{4}.
    $$
    (B2)-(B4): We refer to Lemma \ref{lem_d_t_zeta_E}, Corollary \ref{cor_E}, Lemma \ref{lem_L_0D_t_zeta_E}, Corollary \ref{cor_E^L_0}, Lemma \ref{lem_away_infty} with $\mu=\frac{1}{5}$, Lemma \ref{lem_tilde_E}, Corollary \ref{cor_tilde_E} and recall \eqref{ini_z} and \eqref{lambda_H_0}. Note that $\tilde{C}$ is a constant independent of $M$, and we can assume that $M\gg \tilde{C}$ to close the bootstrap assumptions. Thus if $\epsilon$ and $c_s$ are small enough, \eqref{boot_B_2}-\eqref{boot_B_4} hold.\\
    (B5): From \eqref{z_1-z_2_v} we have
    \begin{equation}\label{d_12}
    \begin{aligned}
    |d_{12}(t)-d_{12}(0)|&\leq|z_1(t)-z_2(t)-z_1(0)+z_2(0)|\\
    &\leq
    \left|\int_0^t C\frac{\epsilon+\lambda}{H(\tau)^{\frac{3}{2}}}d\tau\right|\leq\left|\int_0^t C\frac{\epsilon+\lambda}{c_0^{\frac{3}{2}}(H_0+\lambda \tau)^{\frac{3}{2}}}d\tau\right|\leq\frac{C(\epsilon+\lambda)}{c_0^{\frac{3}{2}}H_0^{\frac{1}{2}}\lambda}.
    \end{aligned}
    \end{equation}
    By \eqref{lambda_H_0} we derive that
    $$
    |d_{12}(t)-d_{12}(0)|\leq Cc_0^{-\frac{3}{2}}c_s.
    $$
    If $c_s$ is small enough such that
    $$
    Cc_0^{-\frac{3}{2}}c_s\leq\frac{1}{4}d_{12}(0),
    $$
    then we obtain \eqref{boot_B_5}.\\
    (B6): Set $\hat{z}=z_1(T)-z_2(T)$. From \eqref{z_1-z_2_v} and \eqref{boot_H} we obtain that 
    \begin{equation}\label{z_1-z_2-hat_z}
    |z_1-z_2-\hat{z}|\leq \left|\int_t^T C\frac{\epsilon+\lambda}{H(\tau)^{\frac{3}{2}}}d\tau\right|\leq\frac{C(\epsilon+\lambda)}{c_0^{\frac{3}{2}}\lambda(H_0+\lambda t)^{\frac{1}{2}}}.
    \end{equation}
    So from \eqref{dot_z_1_exp} and note the fact that $\mathfrak{U}$ is holomorphic:
    $$
    \left|\dot{z_1}+\frac{\lambda i}{2\pi\bar{\hat{z}}}\right|\leq\left|\frac{\lambda i}{2\pi}\frac{\overline{z_1-z_2-\hat{z}}}{\overline{\hat{z}(z_1-z_2)}}\right|+|\bar{\mathfrak{U}}(z_1,t)|\leq\frac{C(\epsilon+\lambda)}{c_0^{\frac{3}{2}}(H_0+\lambda t)^{\frac{1}{2}}}+\frac{C\epsilon}{H^{\frac{1}{2}}}.
    $$
    Integration with respect to $t$ yields that
    $$
    \left|z_1(t)-z_1(0)+\frac{\lambda i}{2\pi\bar{\hat{z}}}t\right|\leq\left|\int_0^t\frac{C(\epsilon+\lambda)}{c_0^{\frac{3}{2}}(H_0+\lambda\tau)^{\frac{1}{2}}}+\frac{C\epsilon}{H(\tau)^{\frac{1}{2}}}d\tau\right|\leq C\left(\frac{\epsilon+\lambda}{c_0^{\frac{3}{2}}\lambda^{\frac{1}{2}}}+\frac{\epsilon}{\lambda^{\frac{1}{2}}}\right)t^{\frac{1}{2}}.
    $$
    Then we consider the vector $\hat{z}$ and recall \eqref{z_1-z_2-hat_z} to find that 
    $$
    |z_1(0)-z_2(0)-\hat{z}|\leq\frac{C(\epsilon+\lambda)}{c_0^{\frac{3}{2}}\lambda H_0^{\frac{1}{2}}}\leq \frac{Cc_s}{c_0^{\frac{3}{2}}}.
    $$
    Suppose that $\epsilon$ is small enough such that 
    $$
    |z_1(0)-z_2(0)-\hat{z}|\leq\frac{1}{10}|z_1(0)-z_2(0)|,
    $$
    then 
    $$
    \operatorname{Re}\hat{z}\geq \operatorname{Re}(z_1(0)-z_2(0))-|z_1(0)-z_2(0)-\hat{z}|\geq\frac{9}{10}
    $$
    which means
    $$
    \operatorname{Im}\frac{\lambda i}{2\pi \bar{\hat{z}}}=\frac{\lambda}{2\pi|\hat{z}|^2}\operatorname{Re}\hat{z}\geq\frac{1}{3\pi}\lambda.
    $$
    Thus for $\forall \alpha\in\mathbb{R}$ we have
    $$
    \begin{aligned}
    |z_1(t)-\zeta(\alpha,t)|\geq\ &\left|-\frac{\lambda i}{2\pi\bar{\hat{z}}}t+z_1(0)-\zeta(\alpha,0)\right|-\left|z_1(t)-z_1(0)+\frac{\lambda i}{2\pi\bar{\hat{z}}}t\right|-|\zeta(\alpha,0)-\zeta(\alpha,t)|\\
    \geq\ &\left|\operatorname{Im}\left(-\frac{\lambda i}{2\pi\bar{\hat{z}}}t+z_1(0)-\zeta(\alpha,0)\right)\right|-C\left(\frac{\epsilon+\lambda}{c_0^{\frac{3}{2}}\lambda^{\frac{1}{2}}}+\frac{\epsilon}{\lambda^{\frac{1}{2}}}\right)t^{\frac{1}{2}}-\int_0^t|\zeta_t(\alpha,\tau)|d\tau\\
    \geq\ &\frac{1}{3\pi}\lambda t+H_0-1-C\left(\frac{\epsilon+\lambda}{c_0^{\frac{3}{2}}\lambda^{\frac{1}{2}}}+\frac{\epsilon}{\lambda^{\frac{1}{2}}}\right)t^{\frac{1}{2}}-C\epsilon t^{\frac{3}{4}}.
    \end{aligned}
    $$
    From \eqref{lambda_H_0} and Young's inequality we deduce that for $\epsilon$ small enough and $H_0$ large enough
    $$
    |z_1(t)-\zeta(\alpha,t)|\geq\frac{7}{24\pi}\lambda t+\frac{1}{2}H_0-\frac{C}{c_0^{\frac{3}{2}}}\lambda^{\frac{1}{2}}t^{\frac{1}{2}}-C\epsilon t^{\frac{3}{4}}\geq\frac{1}{4\pi}\lambda t+\left(\frac{1}{2}H_0-\frac{C^2}{c_0^3}\pi-Cc_s\lambda H_0\right)\geq\frac{1}{4\pi}(\lambda t+H_0)
    $$
    From (B5) we know that $d_{12}\leq 2$. Triangle inequality implies that for $z\in [z_1(t),z_2(t)]$
    $$
    |z-\zeta(\alpha,t)|\geq|z_1-\zeta(\alpha,t)|-|z-z_1|\geq\frac{3}{16\pi}(\lambda t+H_0).
    $$ 
    Since $\alpha$ is chosen arbitrarily, we have proved \eqref{boot_B_6} with $c_1=\frac{3}{16\pi}$.
\end{proof}
\begin{proof}[Proof of Theorem \ref{main}]
From Lemma 5.3 and Lemma 5.4 in \cite{Wu2009} we obtain \eqref{main_one} and \eqref{main_two}. \eqref{main_three} follows from the closure of (B6).
\end{proof}

\section{Consequences of the bootstrap assumptions}\label{sec_b}

The goal of this section is to derive estimates for the point-vortex
trajectories, the dipole field \(q\), the good unknowns
\(\tilde\theta,\tilde\sigma\), and the coefficients \(b,A\). These estimates
will be used in the energy and decay arguments in Sections~\ref{sec_energy} and~\ref{sec_decay}. We split this section into four parts. In subsection \ref{subs_point}, we study the behavior of the point vortices and show that their depth grows approximately linearly in time. Subsection \ref{subs_dipole} suggests that due to the symmetry of the point vortices, their effect on the water-wave surface decays rapidly over time. Next we demonstrate that the new unknowns $\tilde{\theta}$ and $\tilde{\sigma}$ differ little from the old ones in subsection \ref{subs_good}. Finally, subsection \ref{subs_b} contains the estimates for $b$ and $A-1$ which are vital in energy estimates.

We introduce three useful lemmas whose proofs can be found in our previous work \cite{su2025new}.
\begin{lemma}\label{lem_D_t_zeta_L^infty}
    Assume the bootstrap assumptions. There holds
    \begin{equation}
    \begin{aligned}
    \norm{D_t\zeta}_{L^\infty}\lesssim&\ \epsilon t^{-\frac{1}{4}},\\
        \norm{\partial_\alpha^{s-1}D_t\zeta}_{L^\infty}\lesssim&\ \epsilon t^{-\frac{1}{4}}\ln (2+t),\\
        \norm{\partial_\alpha ^{s-1}D^2_t\zeta}_{L^\infty}+\norm{\partial_\alpha^s \zeta}_{L^\infty}\lesssim&\ \epsilon t^{-\frac{1}{6}}\ln (2+t).
    \end{aligned}
    \end{equation}
\end{lemma}

\begin{lemma}\label{lem_Hf_infty}
    Assume the bootstrap assumptions. If $f\in H^1\cap L^\infty$, then
    \begin{equation}
        \norm{\mathfrak{H}f}_{L^\infty}\lesssim\frac{1}{t^2}\norm{f}_{H^1}+\norm{f}_{L^\infty}\ln (2+t).
    \end{equation}
\end{lemma}
\begin{lemma}\label{lem_real_proj}
    Let f be a real-valued function. Suppose that $(I-\mathfrak{H})f=g$ or $(I-\mathfrak{H})f\bar\zeta_\alpha=g$, then
    for $0\leq k\leq s$,
    $$
        \norm{f}_{H^k}\lesssim\norm{g}_{H^k}.
    $$
    Furthermore, for $0\leq l\leq s-2$,
    $$
    \norm{f}_{W^{l,\infty}}\lesssim\norm{g}_{W^{l,\infty}}.
    $$
\end{lemma}

Deriving $L^\infty$ bounds for certain quantities calls for more delicate analysis. Thus the following lemma will be of use in the latter part of this section:
\begin{lemma}\label{lem_S_1}
    Consider the integral$$
    S_1=\int\frac{g(\alpha,t)-g(\beta,t)}{(\zeta(\alpha,t)-\zeta(\beta,t))^2}f(\beta,t)d\beta
    $$
for $f(\cdot,t)\in H^1\cap L^\infty$ and $g_\alpha(\cdot,t)\in W^{1,\infty}$. We have
\begin{equation}\label{S_1}
    \norm{S_{1}}_{L^\infty}\lesssim\frac{1}{t^\frac{1}{2}}\norm{g_\alpha}_{W^{1,\infty}}\norm{f}_{H^1}+\ln (2+t)\norm{g_\alpha}_{L^\infty}\norm{f}_{L^\infty}.
    \end{equation}
\end{lemma}
\begin{proof}
    Notice that
$$
    \begin{aligned}
    S_1&=\int\left(\frac{1}{(\zeta(\alpha,t)-\zeta(\beta,t))^2}-\frac{1}{(\alpha-\beta)^2}\right)(g(\alpha,t)-g(\beta,t))f(\beta,t)d\beta\\
    &+\int\frac{g(\alpha,t)-g(\beta,t)}{(\alpha-\beta)^2}f(\beta,t)d\beta\\
    &\coloneqq S_{11}+S_{12}.
    \end{aligned}
    $$
Direct calculation shows that
$$
    \begin{aligned}
    \norm{S_{11}}_{L^\infty}\lesssim&\norm{S_{11}}_{H^1}\\
    \lesssim&\norm{\int\frac{(\alpha-\zeta(\alpha,t)-\beta+\zeta(\beta,t))(g(\alpha,t)-g(\beta,t))}{(\zeta(\alpha,t)-\zeta(\beta,t))^2(\alpha-\beta)}f(\beta,t)d\beta}_{H^1}\\
    &+\norm{\int\frac{(\alpha-\zeta(\alpha,t)-\beta+\zeta(\beta,t))(g(\alpha,t)-g(\beta,t))}{(\zeta(\alpha,t)-\zeta(\beta,t))(\alpha-\beta)^2}f(\beta,t)d\beta}_{H^1}\\
    \lesssim&\frac{1}{t^\frac{1}{2}}\norm{\partial_\alpha g}_{W^{1,\infty}}\norm{f}_{H^1}.
    \end{aligned}
    $$
    We split the integration region of $S_{12}$ into three parts:
    
    $$
    \begin{aligned}
    S_{12}=&\int_{|\alpha-\beta|\leq t^{-1}}\frac{g(\alpha,t)-g(\beta,t)}{(\alpha-\beta)^2}f(\beta,t)d\beta+\int_{t^{-1}<|\alpha-\beta|\leq t}\frac{g(\alpha,t)-g(\beta,t)}{(\alpha-\beta)^2}f(\beta,t)d\beta\\
    &+\int_{|\alpha-\beta|> t}\frac{g(\alpha,t)-g(\beta,t)}{(\alpha-\beta)^2}f(\beta,t)d\beta\\
    =&\int_{|\alpha-\beta|\leq t^{-1}}\frac{1}{\alpha-\beta}\left(\frac{g(\alpha,t)-g(\beta,t)}{\alpha-\beta}-g_\alpha(\alpha,t)\right)f(\beta,t)d\beta\\
    &+g_\alpha(\alpha,t)\int_{|\alpha-\beta|\leq t^{-1}}\frac{f(\beta,t)-f(\alpha,t)}{\alpha-\beta}d\beta+\int_{t^{-1}<|\alpha-\beta|\leq t}\frac{g(\alpha,t)-g(\beta,t)}{(\alpha-\beta)^2}f(\beta,t)d\beta\\
    &+\int_{|\alpha-\beta|> t}\frac{g(\alpha,t)-g(\beta,t)}{(\alpha-\beta)^2}f(\beta,t)d\beta.
    \end{aligned}
    $$
Then by Young's inequality and Hardy's inequality:
$$
    \begin{aligned}
    \norm{S_{12}}_{L^\infty}\lesssim& \frac{1}{t}\norm{\partial_\alpha^2g}_{L^\infty}\norm{f}_{L^\infty}+\frac{1}{t^\frac{1}{2}}\norm{g_\alpha}_{L^\infty}\left[\int\left(\frac{f(\beta,t)-f(\alpha,t)}{\alpha-\beta}\right)^2d\beta\right]^\frac{1}{2}\\
    &+\norm{g_\alpha}_{L^\infty}\norm{f}_{L^\infty}\int_{t^{-1}<|\alpha-\beta|\leq t}\frac{1}{|\alpha-\beta|}d\beta+\norm{\frac{1}{\alpha}1_{|\alpha|>t}}_{L^2}\norm{g_\alpha}_{L^\infty}\norm{f}_{L^2}\\
    \lesssim&\frac{1}{t}\norm{\partial_\alpha^2g}_{L^\infty}\norm{f}_{L^\infty}+\frac{1}{t^\frac{1}{2}}\norm{g_\alpha}_{L^\infty}\norm{f}_{H^1}+\ln (2+t)\norm{g_\alpha}_{L^\infty}\norm{f}_{L^\infty}
    \end{aligned}
    $$
where we obtain (\ref{S_1}). 
\end{proof}
Moreover, we introduce the following expressions with respect to $b$, $A$ and $\frac{a_t}{a}$ which appear in \cite{su2020long}.
\begin{lemma}\label{lem_b_exp}
    With the notation in Section \ref{sec_pre}, there holds
    \begin{equation}\label{b_exp}
       (I-\mathfrak{H})b= -[D_t\zeta,\mathfrak{H}]\frac{\bar\zeta_\alpha-1}{\zeta_\alpha}+2q,
    \end{equation}
    \begin{equation}\label{A_exp}
        (I-\mathfrak{H})(A-1)=i[ D_t\zeta,\mathfrak{H}]\frac{\partial_\alpha\mathfrak{F}}{\zeta_\alpha}
  + i[ D_t^2\zeta,\mathfrak{H}]\frac{\bar{\zeta}_\alpha-1}{\zeta_\alpha}
  - (I-\mathfrak{H})\frac{1}{2\pi}\sum_{j=1}^{2}\lambda_j
    \frac{D_t\zeta(\alpha,t)-\dot{z}_j(t)}{(\zeta(\alpha,t)-z_j(t))^2}
    \end{equation}
    and
    \begin{equation}\label{a_t/a}
\begin{aligned}
  (I-\mathfrak{H})\left(\frac{a_t}{a}\circ\kappa^{-1}A\bar{\zeta}_\alpha\right)
  = &2i[ D_t^2\zeta,\mathfrak{H}]\frac{\partial_\alpha D_t\bar{\zeta}}{\zeta_\alpha}
   + 2i[D_t\zeta,\mathfrak{H}]\frac{\partial_\alpha D_t^2\bar{\zeta}}{\zeta_\alpha}\\
  &- \frac{1}{\pi}\int
    \left(\frac{D_t\zeta(\alpha,t)-D_t\zeta(\beta,t)}{\zeta(\alpha,t)-\zeta(\beta,t)}\right)^{2}
    (D_t\bar{\zeta})_\beta\,\mathrm{d}\beta\\
  &- \frac{1}{\pi}\sum_{j=1}^{2}\lambda_j
    \left(
      \frac{2D_t^2\zeta + i - \partial_t^2 z_j}{\left(\zeta(\alpha,t)-z_j(t)\right)^2}
      - 2\frac{\left(D_t\zeta-\dot z_j(t)\right)^2}{\left(\zeta(\alpha,t)-z_j(t)\right)^3}
    \right).
\end{aligned}
\end{equation}
\end{lemma}

\subsection{Point-vortex trajectories}\label{subs_point}
When we study the pair of point vortices, we note that their velocities are almost parallel with their accelerations decreasing rapidly. This observation provides insight into the motion of the point vortices and the velocity field caused by them. 

\begin{lemma}\label{lem_z_1}
    For $t\in[0,T]$, we have
    \begin{equation}
    |\dot{z_1}|+|\dot{z_2}|\lesssim\lambda+\frac{\epsilon}{t^{\frac{1}{4}}},
    \end{equation}
    \begin{equation}\label{z_1-z_2_v}
    |\dot{z_1}-\dot{z_2}|\lesssim\frac{\epsilon+\lambda}{H^{\frac{3}{2}}},
    \end{equation}
    \begin{equation}\label{ddot_z_1}
    |\ddot{z_1}|+|\ddot{z_2}|\lesssim\frac{(\epsilon+\lambda)^2}{H^{\frac{3}{2}}}+\frac{\epsilon}{t^{\frac{1}{2}}},
    \end{equation}
    \begin{equation}\label{ddot_z_1-z_2}
    |\ddot{z_1}-\ddot{z_2}|\lesssim\frac{\epsilon+\lambda}{H^{\frac{3}{2}}},
    \end{equation}
    \begin{equation}
    |\dddot{z_1}|+|\dddot{z_2}|\lesssim\frac{(\epsilon+\lambda)^2}{H^{\frac{3}{2}}}+\frac{\epsilon}{t^{\frac{1}{2}}}
    \end{equation}
    and
    \begin{equation}
    |\dddot{z_1}-\dddot{z_2}|\lesssim\frac{\epsilon+\lambda}{H^{\frac{3}{2}}}.
    \end{equation}
\end{lemma}
\begin{proof}
    Note that
    \begin{equation}\label{dot_z_1_exp}
    \dot{z_1}=-\frac{\lambda i}{2\pi(\bar{z_1}-\bar{z_2})}+\bar{\mathfrak{U}}(z_1,t)
    \end{equation}
    and
    \begin{equation}\label{dot_z_2_exp}
    \dot{z_2}=-\frac{\lambda i}{2\pi(\bar{z_1}-\bar{z_2})}+\bar{\mathfrak{U}}(z_2,t),
    \end{equation}
    we take advantage of the holomorphicity of $\mathfrak{U}$ and claim that
    $$
    |\dot{z_1}|+|\dot{z_2}|\lesssim\frac{\lambda}{|z_1-z_2|}+\norm{\mathfrak{F}}_{L^\infty}\lesssim\lambda+\norm{D_t\zeta}_{L^\infty}+\norm{q}_{L^\infty}\lesssim\lambda+\frac{\epsilon}{t^{\frac{1}{4}}}
    $$
    and
    $$
    |\dot{z_1}-\dot{z_2}|\lesssim|z_1-z_2|\sup_{z\in[z_1,z_2]}|\mathfrak{U}_z(z,t)|,
    $$
    where $[z_1,z_2]$ represents the line segment connecting $z_1$ and $z_2$. For $z\in[z_1,z_2]$,
    $$
    \mathfrak{U}_z(z,t)=-\frac{1}{2\pi i}\int_{\Sigma(t)}\frac{\mathfrak{U}(\omega,t)}{(\omega-z)^2}d\omega=-\frac{1}{2\pi i}\int\frac{\mathfrak{F}(\alpha,t)}{(\zeta(\alpha,t)-z)^2}\zeta_\alpha d\alpha,
    $$
    which implies
    $$
    \sup_{z\in[z_1,z_2]}|\mathfrak{U}_z(z,t)|\lesssim\norm{\mathfrak{F}}_{L^2}\sup_{z\in[z_1,z_2]}\norm{\frac{1}{(\zeta(\alpha,t)-z)^2}}_{L^2}.
    $$
    Fix $z$ and we set $H_z=d(z,\Sigma(t))$. Denote by $\omega_0$ the nearest point $\zeta(\alpha_0,t)$ to $z$ in $\Sigma(t)$ where $\alpha_0\in\mathbb{R}$. We find that
    $$
    \begin{aligned}
    \int\frac{1}{|\zeta(\beta,t)-z|^4}d\beta
    =&\int_{|\zeta(\beta,t)-\omega_0|\leq 2H_z}\frac{1}{|\zeta(\beta,t)-z|^4}d\beta\\
    &+\int_{|\zeta(\beta,t)-\omega_0|> 2H_z}\frac{1}{|\zeta(\beta,t)-\omega_0+\omega_0-z|^4}d\beta.
    \end{aligned}
    $$
    If $|\zeta(\beta,t)-\omega_0|> 2H_z$, triangle inequality shows that
    $$
    |\zeta(\beta,t)-\omega_0+\omega_0-z|\geq\frac{1}{2}|\zeta(\beta,t)-\omega_0|\gtrsim|\beta-\alpha_0|
    $$
    which means
    $$
    \begin{aligned}
    \int\frac{1}{|\zeta(\beta,t)-z|^4}d\beta
    \lesssim\frac{1}{H_z^3}+\int_{|\beta-\alpha_0|\gtrsim H_z}\frac{1}{|\beta-\alpha_0|^4}d\beta
    \lesssim\frac{1}{H_z^3}\lesssim\frac{1}{H^3}.
    \end{aligned}
    $$
    We see that
    \begin{equation}\label{sup_U_z}
    \sup_{z\in[z_1,z_2]}|\mathfrak{U}_z(z,t)|\lesssim\frac{1}{H^{\frac{3}{2}}}(\norm{D_t\zeta}_{L^2}+\norm{q}_{L^2})\lesssim\frac{\epsilon+\lambda}{H^{\frac{3}{2}}}.
    \end{equation}
    Recalling the bootstrap assumptions, we conclude that
    $$
    |\dot{z_1}-\dot{z_2}|\lesssim\frac{\epsilon+\lambda}{H^{\frac{3}{2}}}.
    $$
    Moreover,
    $$
    \ddot{z_1}=\frac{\lambda i(\bar{\dot{z_2}}-\bar{\dot{z_1}})}{2\pi(\bar{z_1}-\bar{z_2})^2}+\overline{\mathfrak{U}_z}(z_1,t)\bar{\dot{z_1}}+\bar{\mathfrak{U}}_t(z_1,t)
    $$
    and
    $$
    \ddot{z_2}=\frac{\lambda i(\bar{\dot{z_2}}-\bar{\dot{z_1}})}{2\pi(\bar{z_1}-\bar{z_2})^2}+\overline{\mathfrak{U}_z}(z_2,t)\bar{\dot{z_2}}+\bar{\mathfrak{U}}_t(z_2,t)
    $$
    which implies that
    $$
    |\ddot{z_1}|+|\ddot{z_2}|\lesssim\lambda|\dot{z_1}-\dot{z_2}|+(|\overline{\mathfrak{U}_z}(z_1,t)|+|\overline{\mathfrak{U}_z}(z_2,t)|)(|\dot{z_1}|+|\dot{z_2}|)+\sup_{z\in\Sigma(t)}|\mathfrak{U}_t(z,t)|.
    $$
    Since $\mathfrak{F}(\alpha,t)=\mathfrak{U}(\zeta(\alpha,t),t)$, there holds
    \begin{equation}\label{U_t}
    \begin{aligned}
    \mathfrak{U}_t(\zeta(\alpha,t),t)=&\ \mathfrak{F}_t(\alpha,t)-\mathfrak{U}_z(\zeta(\alpha,t),t)\zeta_t(\alpha,t)=\partial_tD_t\bar\zeta-q_t-\mathfrak{U}_z(\zeta(\alpha,t),t)\zeta_t(\alpha,t)\\
    =&\ D_t^2\bar\zeta-b\partial_\alpha D_t\bar\zeta-q_t-\frac{\partial_\alpha D_t\bar\zeta-q_\alpha}{\zeta_\alpha}\zeta_t.
    \end{aligned}
    \end{equation}
    For $\norm{b}_{L^\infty}$ we refer to Lemma \ref{lem_real_proj} and Lemma \ref{lem_b_exp}. Reviewing the definition of $q$, we deduce that
    $$
    q_\alpha=\frac{\lambda i}{2\pi}(z_1-z_2)\zeta_\alpha\left(\frac{1}{(\zeta-z_1)^2(\zeta-z_2)}+\frac{1}{(\zeta-z_1)(\zeta-z_2)^2}\right)
    $$
    and
    $$
    q_t=D_tq-bq_\alpha=-\frac{\lambda i}{2\pi}\frac{\dot{z_1}-\dot{z_2}}{(\zeta-z_1)(\zeta-z_2)}+\frac{\lambda i}{2\pi}(z_1-z_2)\left(\frac{D_t\zeta-\dot{z_1}}{(\zeta-z_1)^2(\zeta-z_2)}+\frac{D_t\zeta-\dot{z_2}}{(\zeta-z_1)(\zeta-z_2)^2}\right)-bq_\alpha.
    $$
    Therefore,
    $$
    \norm{q_\alpha}_{L^\infty}\lesssim\frac{\lambda}{H^3}
    $$
    and
    $$
    \norm{q_t}_{L^\infty}\lesssim|\dot{z_1}-\dot{z_2}|\frac{\lambda}{H^2}+\frac{\lambda(\epsilon+\lambda)}{H^3}+\norm{bq_\alpha}_{L^\infty}\lesssim\frac{\lambda(\epsilon+\lambda)}{H^3}.
    $$
    Hence,
    $$
    \sup_{z\in\Sigma(t)}|\mathfrak{U}_t|\lesssim\frac{\epsilon}{t^{\frac{1}{2}}}+\frac{\lambda(\epsilon+\lambda)}{H^3}.
    $$
    Finally,
    $$
    |\ddot{z_1}|+|\ddot{z_2}|\lesssim\frac{\lambda(\epsilon+\lambda)}{H^{\frac{3}{2}}}+\frac{\epsilon(\epsilon+\lambda)}{H^{\frac{3}{2}}t^{\frac{1}{4}}}+\frac{\epsilon}{t^{\frac{1}{2}}}\lesssim\frac{(\epsilon+\lambda)^2}{H^{\frac{3}{2}}}+\frac{\epsilon}{t^{\frac{1}{2}}}.
    $$
    Meanwhile,
    $$
    \begin{aligned}
    |\ddot{z_1}-\ddot{z_2}|\leq&|\mathfrak{U}_z(z_1,t)-\mathfrak{U}_z(z_2,t)||\dot{z_1}|+|\mathfrak{U}_z(z_2,t)||\dot{z_1}-\dot{z_2}|+|\mathfrak{U}_t(z_1,t)-\mathfrak{U}_t(z_2,t)|\\
    \lesssim&\sup_{z\in[z_1,z_2]}|\mathfrak{U}_{zz}(z,t)|(\lambda+\frac{\epsilon}{t^\frac{1}{4}})+\frac{(\epsilon+\lambda)^2}{H^3}+\sup_{z\in[z_1,z_2]}|\mathfrak{U}_{zt}(z,t)|\\
    \lesssim&\frac{1}{H^{\frac{5}{2}}}\norm{\mathfrak{F}}_{L^2}(\lambda+\frac{\epsilon}{t^\frac{1}{4}})+\frac{(\epsilon+\lambda)^2}{H^3}+\frac{1}{H^{\frac{3}{2}}}\norm{\mathfrak{U}_t(\zeta(\alpha,t),t)}_{L^2}\\
    \lesssim&\frac{\epsilon+\lambda}{H^{\frac{3}{2}}},
    \end{aligned}
    $$
    where we treat $|\mathfrak{U}_{zz}(z,t)|$ and $|\mathfrak{U}_{zt}(z,t)|$ as the preceding argument. It follows from a direct computation that
    $$
    \dddot{z_1}=\partial_t\frac{\lambda i(\bar{\dot{z_2}}-\bar{\dot{z_1}})}{2\pi(\bar{z_1}-\bar{z_2})^2}+\overline{\mathfrak{U}_{zz}}(z_1,t)\bar{\dot{z_1}}^2+2\overline{\mathfrak{U}_{zt}}(z_1,t)\bar{\dot{z_1}}+\overline{\mathfrak{U}_z}(z_1,t)\bar{\ddot{z_1}}+\bar{\mathfrak{U}}_{tt}(z_1,t)
    $$
    and
    $$
    \dddot{z_2}=\partial_t\frac{\lambda i(\bar{\dot{z_2}}-\bar{\dot{z_1}})}{2\pi(\bar{z_1}-\bar{z_2})^2}+\overline{\mathfrak{U}_{zz}}(z_2,t)\bar{\dot{z_2}}^2+2\overline{\mathfrak{U}_{zt}}(z_2,t)\bar{\dot{z_2}}+\overline{\mathfrak{U}_z}(z_2,t)\bar{\ddot{z_2}}+\bar{\mathfrak{U}}_{tt}(z_2,t),
    $$
    which means
    $$
    |\dddot{z_1}|+|\dddot{z_2}|\lesssim\lambda|\ddot{z_1}-\ddot{z_2}|+\lambda|\dot{z_1}-\dot{z_2}|^2+\frac{(\epsilon+\lambda)^3}{H^{\frac{5}{2}}}+\frac{(\epsilon+\lambda)^2}{H^{\frac{3}{2}}}+\sup_{z\in\Sigma(t)}|\mathfrak{U}_{tt}(z,t)|.
    $$
    From (\ref{U_t}) we have
    $$
    \begin{aligned}
    \mathfrak{U}_{tt}(\zeta(\alpha,t),t)=&\ \partial_tD_t^2\bar\zeta-b_t\partial_\alpha D_t\bar\zeta-b\partial_\alpha(D_t-b\partial_\alpha)D_t\bar\zeta-q_{tt}-\partial_t\left(\frac{\partial_\alpha D_t\bar\zeta-q_\alpha}{\zeta_\alpha}\zeta_t\right)\\
    =&\ \partial_t(-iA\bar\zeta_\alpha)-b_t\partial_\alpha D_t\bar\zeta-b\partial_\alpha D_t^2\bar\zeta+bb_\alpha\partial_\alpha D_t\bar\zeta+b^2\partial_\alpha^2D_t\bar\zeta-q_{tt}-\frac{\partial_\alpha\partial_tD_t\bar\zeta-q_{t\alpha}}{\zeta_\alpha}\zeta_t\\
    &+\frac{\partial_\alpha D_t\bar\zeta-q_\alpha}{\zeta_\alpha^2}\zeta_{t\alpha}\zeta_t-\frac{\partial_\alpha D_t\bar\zeta-q_\alpha}{\zeta_\alpha}\zeta_{tt}.
    \end{aligned}
    $$
    For the bound of some functions with respect to $b$ and $A$, we refer to Lemma \ref{lem_real_proj} and Lemma \ref{lem_b_exp}. Since we have obtained the estimates of $\dot{z_1},\ \ddot{z_1},\ \dot{z_2}$ and $\ddot{z_2}$, we are ready to control the terms $b,\ b_\alpha,\ b_t,\ A,\  A_t,\ q_\alpha,\ q_{t\alpha}$ and $q_{tt}$. Thus we have
    $$
    \sup_{z\in\Sigma(t)}|\mathfrak{U}_{tt}(z,t)|\lesssim\frac{\lambda(\epsilon+\lambda)}{H^{\frac{7}{2}}}+\frac{\epsilon}{t^{\frac{1}{2}}}.
    $$
    With the preceding argument in mind, we derive that
    $$
    |\dddot{z_1}|+|\dddot{z_2}|\lesssim\frac{(\epsilon+\lambda)^2}{H^{\frac{3}{2}}}+\frac{\epsilon}{t^{\frac{1}{2}}}.
    $$
    Furthermore,
    $$
    \begin{aligned}
    \dddot{z_1}-\dddot{z_2}=&\ (\overline{\mathfrak{U}_{zz}}(z_1,t)-\overline{\mathfrak{U}_{zz}}(z_2,t))\bar{\dot{z_1}}^2+\overline{\mathfrak{U}_{zz}}(z_2,t)(\bar{\dot{z_1}}^2-\bar{\dot{z_2}}^2)\\
    &+2(\overline{\mathfrak{U}_{zt}}(z_1,t)-\overline{\mathfrak{U}_{zt}}(z_2,t))\bar{\dot{z_1}}+2\overline{\mathfrak{U}_{zt}}(z_2,t)(\bar{\dot{z_1}}-\bar{\dot{z_2}})\\
    &+(\overline{\mathfrak{U}_{z}}(z_1,t)-\overline{\mathfrak{U}_{z}}(z_2,t))\bar{\ddot{z_1}}+\overline{\mathfrak{U}_{z}}(z_2,t)(\bar{\ddot{z_1}}-\bar{\ddot{z_2}})+\bar{\mathfrak{U}}_{tt}(z_1,t)-\bar{\mathfrak{U}}_{tt}(z_2,t),
    \end{aligned}
    $$
    from which we show that
    $$
    \begin{aligned}
    |\dddot{z_1}-\dddot{z_2}|\lesssim&\frac{(\epsilon+\lambda)^2}{H^\frac{7}{2}}\norm{\mathfrak{F}}_{L^2}+\frac{\epsilon+\lambda}{H^{\frac{5}{2}}}\norm{\mathfrak{U}_t(\zeta(\alpha,t),t)}_{L^2}\\
    &+\frac{\epsilon+\lambda}{H^{3}}\norm{\mathfrak{F}}_{L^2}+\frac{1}{H^{\frac{3}{2}}}\norm{\mathfrak{U}_{tt}(\zeta(\alpha,t),t)}_{L^2}\\
    \lesssim&\frac{\epsilon+\lambda}{H^{\frac{3}{2}}}.
    \end{aligned}
    $$
\end{proof}

\subsection{Bounds for the dipole field}\label{subs_dipole}
Next we will investigate the vortex-induced velocity $q$. Thanks to the symmetry of the point vortices, we shall treat $q$ as a small perturbation of the velocity restricted on the surface, which can be easily handled during the proof.
\begin{proposition}[Bounds for the dipole field]\label{pro_q}
    Assume (B1)-(B6). For $t\in [0,T]$, we have
    $$\norm{q}_{H^{s+1}}\lesssim\frac{\lambda}{H^\frac{3}{2}},$$

    $$\norm{q}_{W^{s-1,\infty}}\lesssim\frac{\lambda}{H^2},$$

        $$\norm{D_tq}_{H^s}\lesssim\frac{\lambda(\epsilon+\lambda)}{H^{\frac{5}{2}}},$$

        $$\norm{D_tq}_{W^{s-1,\infty}}\lesssim\frac{\lambda(\epsilon+\lambda)}{H^3}$$

        and
        
        $$\norm{D_t^2q}_{H^s}\lesssim\min\left\{\frac{\lambda(\epsilon+\lambda)}{H^3},\frac{\lambda(\epsilon+\lambda)^2}{H^3}+\frac{\lambda\epsilon}{H^{\frac{3}{2}}t^{\frac{1}{2}}}\right\}.$$
\end{proposition}
\begin{proof}
    Set $H_j(t)=d(z_j,\Sigma(t))$ for $j=1,2$ and denote by $\omega_1$ the nearest point $\zeta(\alpha_1,t)$ to $z_1$ in $\Sigma(t)$ where $\alpha_1\in\mathbb{R}$. Let $g=(\zeta-z_1)(\zeta-z_2)$. Recall that
    $$
    q=-\frac{\lambda i}{2\pi}\frac{z_1-z_2}{g}.
    $$
    Direct calculation with chain rule shows that
$$
    \partial_\alpha^n q=-\frac{\lambda i}{2\pi}(z_1-z_2)\sum_{k=1}^n\sum_{\substack{\sum_{l=1}^nk_l=k,\\\sum_{l=1}^nlk_l=n}}\frac{n!}{(k_1)!\cdots(k_n)!}(-1)^kk!\left(\frac{1}{g}\right)^{k+1}\prod_{l=1}^n\left(\frac{\partial_\alpha^lg}{l!}\right)^{k_l}
    $$
for $1\leq n\leq s+1$. To estimate $\norm{\partial_\alpha^n q}_{L^2}$, we focus on $\norm{\frac{\Pi_{l=1}^n(\partial_\alpha^lg)^{k_l}}{g^{k+1}}}_{L^2}$ and examine the integral below when $k+l\geq 2$ via method similar to the one in Lemma \ref{lem_z_1}:
$$
    \begin{aligned}
    &\int\frac{1}{|\zeta(\beta,t)-z_1|^k|\zeta(\beta,t)-z_2|^l}d\beta\\
    =&\int_{|\zeta(\beta,t)-\omega_1|\leq 2H_1}\frac{1}{|\zeta(\beta,t)-z_1|^k|\zeta(\beta,t)-z_2|^l}d\beta\\
    &+\int_{|\zeta(\beta,t)-\omega_1|> 2H_1}\frac{1}{|\zeta(\beta,t)-\omega_1+\omega_1-z_1|^k|\zeta(\beta,t)-\omega_1+\omega_1-z_2|^l}d\beta\\
    \lesssim&\ \frac{1}{H_1^{k+l-1}}+\int_{|\beta|\gtrsim H_1}\frac{1}{|\beta|^{k+l}}d\beta\\
    \lesssim&\ \frac{1}{H^{k+l-1}}.
    \end{aligned}
    $$
The first inequality is deduced from $\frac{1}{2}\leq\frac{|\zeta(\alpha,t)-\zeta(\beta,t)|}{|\alpha-\beta|}\leq\frac{3}{2}$, which is a consequence of the bootstrap assumption about $\norm{\zeta_\alpha-1}_{L^\infty}$. Hence there holds
$$
    \norm{q}_{H^{s+1}}=\left(\sum_{n=0}^{s+1}\norm{\partial_\alpha^nq}^2_{L^2}\right)^{\frac{1}{2}}\lesssim\lambda|z_1-z_2|\left(\norm{\frac{1}{g}}_{L^2}+\sum_{n=1}^{s+1}\sum_{k=1}^n\norm{\frac{\Pi_{l=1}^n(\partial_\alpha^lg)^{k_l}}{g^{k+1}}}_{L^2}\right)\lesssim\frac{\lambda}{H^\frac{3}{2}}.
    $$
Meanwhile,
$$
    \norm{q}_{W^{s-1,\infty}}\lesssim\frac{\lambda}{H^2}.
    $$
For
\begin{equation}\label{D_t_q}
    D_tq=-\frac{\lambda i}{2\pi}\frac{\dot{z_1}-\dot{z_2}}{(\zeta-z_1)(\zeta-z_2)}+\frac{\lambda i}{2\pi}(z_1-z_2)\left(\frac{D_t\zeta-\dot{z_1}}{(\zeta-z_1)^2(\zeta-z_2)}+\frac{D_t\zeta-\dot{z_2}}{(\zeta-z_1)(\zeta-z_2)^2}\right),
\end{equation}
we utilize the chain rule and obtain
$$
    \norm{D_tq}_{H^s}\lesssim\frac{\lambda(\epsilon+\lambda)}{H^{\frac{5}{2}}},
    $$
$$
    \norm{D_tq}_{W^{s-1,\infty}}\lesssim\frac{\lambda(\epsilon+\lambda)}{H^3}.
    $$
Straightforward calculation yields that
\begin{equation}\label{D_t^2q}
\begin{aligned}
D_t^2q=&-\frac{\lambda i}{2\pi}\frac{\ddot{z_1}-\ddot{z_2}}{(\zeta-z_1)(\zeta-z_2)}+\frac{\lambda i}{\pi}(\dot{z_1}-\dot{z_2})\left(\frac{D_t\zeta-\dot{z_1}}{(\zeta-z_1)^2(\zeta-z_2)}+\frac{D_t\zeta-\dot{z_2}}{(\zeta-z_1)(\zeta-z_2)^2}\right)\\
&+\frac{\lambda i}{2\pi}(z_1-z_2)D_t\left(\frac{D_t\zeta-\dot{z_1}}{(\zeta-z_1)^2(\zeta-z_2)}+\frac{D_t\zeta-\dot{z_2}}{(\zeta-z_1)(\zeta-z_2)^2}\right).
\end{aligned}
\end{equation}
If we treat $\ddot{z_1}-\ddot{z_2}$ with (\ref{ddot_z_1-z_2}), there holds
$$
\norm{D_t^2q}_{H^s}\lesssim\frac{\lambda(\epsilon+\lambda)}{H^3}.
$$
(\ref{ddot_z_1}) also implies that $|\ddot{z_1}-\ddot{z_2}|\lesssim\frac{(\epsilon+\lambda)^2}{H^{\frac{3}{2}}}+\frac{\epsilon}{t^{\frac{1}{2}}}$. Therefore
$$
\norm{D_t^2q}_{H^s}\lesssim\frac{\lambda(\epsilon+\lambda)^2}{H^3}+\frac{\lambda\epsilon}{H^{\frac{3}{2}}t^{\frac{1}{2}}}.
$$

\end{proof}

\begin{proposition}[Vector-field bounds for $q$]\label{pro_L_0_q}
    Under (B1)-(B6),
    $$
    \norm{L_0q}_{L^2}\lesssim\frac{\lambda(\epsilon+\lambda)t}{H^\frac{5}{2}},
    $$
    $$
    \norm{\partial_\alpha L_0q}_{H^{s-1}}\lesssim\frac{\lambda\epsilon(\epsilon+\lambda)t^{\frac{1}{2}}}{H^{\frac{7}{2}}}+\frac{\lambda\epsilon}{H^{\frac{5}{2}}t^{\frac{1}{2}}},
    $$
    $$
    \norm{D_tL_0q}_{H^{s-1}}\lesssim\min\left\{\frac{\lambda(\epsilon+\lambda)t}{H^3
        }+\frac{\lambda(\epsilon+\lambda) t^{\frac{1}{2}}}{H^{\frac{5}{2}}},\frac{\lambda(\epsilon+\lambda)^2t}{H^3}+\frac{\lambda\epsilon t^{\frac{1}{2}}}{H^{\frac{3}{2}}}\right\}
    $$
    and
    $$
    \norm{L_0D_t^2q}_{H^{s-1}}\lesssim\frac{\lambda(\epsilon+\lambda) t}{H^3}+\frac{\lambda(\epsilon+\lambda) t^{\frac{1}{2}}}{H^{\frac{5}{2}}}.
    $$
\end{proposition}
\begin{proof}
    For $L_0q$, there holds
$$
    \begin{aligned}
    L_0q=&\ L_0\left(-\frac{\lambda i}{2\pi}\frac{z_1-z_2}{(\zeta-z_1)(\zeta-z_2)}\right)\\
    =&-\frac{\lambda it}{4\pi}\frac{\dot{z_1}-\dot{z_2}}{(\zeta-z_1)(\zeta-z_2)}+\frac{\lambda i}{2\pi}(z_1-z_2)\left(\frac{L_0\zeta-\frac{1}{2}t\dot{z_1}}{(\zeta-z_1)^2(\zeta-z_2)}+\frac{L_0\zeta-\frac{1}{2}t\dot{z_2}}{(\zeta-z_1)(\zeta-z_2)^2}\right).
    \end{aligned}
    $$
Under the bootstrap assumptions, we discover that
$$
    \norm{\frac{L_0\zeta}{\zeta-z_1}}_{L^\infty}\lesssim t\norm{\frac{D_t\zeta}{\zeta-z_1}}_{L^\infty}+\norm{\frac{\alpha\zeta_\alpha}{\zeta-z_1}}_{L^\infty}.
    $$
For each $t$, there exists $\alpha_0(t)\in \mathbb{R}$ such that $d(z_1(t),\Sigma(t))=d(z_1(t),\zeta(\alpha_0(t),t))$. We claim that $|\alpha_0(t)|\lesssim H$. Referring to \eqref{ass_diff} and \eqref{boot_H} leads that
$$\begin{aligned}
    |\alpha_0(t)|\lesssim&\ |\zeta(\alpha_0(t),t)-\zeta(0,t)|\\
    \leq&\ |\zeta(\alpha_0(t),t)-z_1(t)|+|z_1(t)-z_1(0)|\\
    &+|z_1(0)-\zeta(\alpha_0(0),0)|+|\zeta(0,t)-\zeta(\alpha_0(0),0)|\\
    \lesssim&\ H(t)+\left(\sup_{t\in[0,T]}|\dot{z_1}|\right)t+H_0+\alpha_0(0)\\
    \lesssim&\ H.
    \end{aligned}
    $$
    If $|\alpha-\alpha_0(t)|\leq 2H$, we see that
$$
    \left|\frac{\alpha\zeta_\alpha}{\zeta-z_1}\right|\lesssim\frac{|\alpha-\alpha_0|+|\alpha_0|}{H}\lesssim1.
    $$
Otherwise,
$$
    \left|\frac{\alpha\zeta_\alpha}{\zeta-z_1}\right|=\left|\frac{\alpha\zeta_\alpha}{\zeta-\zeta(\alpha_0,t)+\zeta(\alpha_0,t)-z_1}\right|\lesssim\frac{\alpha}{|\alpha-\alpha_0|-H}\lesssim 1
    $$
Therefore,
$$
    \norm{\frac{\alpha\zeta_\alpha}{\zeta-z_1}}_{L^\infty}\lesssim 1.
    $$
Hence,
$$
    \norm{\frac{L_0\zeta}{\zeta-z_1}}_{L^\infty}\lesssim\frac{\epsilon t^{\frac{1}{2}}}{H}+ 1\lesssim\frac{\epsilon t^{\frac{1}{2}}}{\sqrt{H_0\lambda t}}+1\lesssim 1.
    $$
With Lemma \ref{lem_z_1} in mind, we claim that
$$
    \norm{L_0q}_{L^2}\lesssim\frac{\lambda(\epsilon+\lambda) t}{H^{\frac{5}{2}}}.
    $$
For higher derivatives of $L_0q$, we proceed as above and deduce that
$$
    \norm{\partial_\alpha L_0q}_{H^{s-1}}\lesssim\frac{\lambda\epsilon(\epsilon+\lambda)t^{\frac{1}{2}}}{H^{\frac{7}{2}}}+\frac{\lambda\epsilon}{H^{\frac{5}{2}}t^{\frac{1}{2}}}.
    $$

    For $D_tL_0q$, straightforward calculation yields that
    \begin{equation}\label{D_t_L_0_q}
    \begin{aligned}
    D_tL_0q=&-\frac{\lambda i}{4\pi}\frac{\dot{z_1}-\dot{z_2}+t(\ddot{z_1}-\ddot{z_2})}{(\zeta-z_1)(\zeta-z_2)}+\frac{\lambda it}{4\pi}(\dot{z_1}-\dot{z_2})\left(\frac{D_t\zeta-\dot{z_1}}{(\zeta-z_1)^2(\zeta-z_2)}+\frac{D_t\zeta-\dot{z_2}}{(\zeta-z_1)(\zeta-z_2)^2}\right)\\
    &+\frac{\lambda i}{2\pi}(\dot{z_1}-\dot{z_2})\left(\frac{L_0\zeta-\frac{1}{2}t\dot{z_1}}{(\zeta-z_1)^2(\zeta-z_2)}+\frac{L_0\zeta-\frac{1}{2}t\dot{z_2}}{(\zeta-z_1)(\zeta-z_2)^2}\right)\\
    &+\frac{\lambda i}{2\pi}(z_1-z_2)D_t\left(\frac{L_0\zeta-\frac{1}{2}t\dot{z_1}}{(\zeta-z_1)^2(\zeta-z_2)}+\frac{L_0\zeta-\frac{1}{2}t\dot{z_2}}{(\zeta-z_1)(\zeta-z_2)^2}\right).
    \end{aligned}
    \end{equation}
    By adopting Lemma \ref{lem_z_1}, we derive that
    $$
    \norm{D_tL_0q}_{H^{s-1}}\lesssim\frac{\lambda(\epsilon+\lambda)t}{H^3
    }+\frac{\lambda(\epsilon+\lambda) t^{\frac{1}{2}}}{H^{\frac{5}{2}}}
    $$
    when we utilize (\ref{ddot_z_1-z_2}) to treat $\ddot{z_1}-\ddot{z_2}$. If we adopt (\ref{ddot_z_1}) instead, there holds
    $$
    \norm{D_tL_0q}_{H^{s-1}}\lesssim\frac{\lambda(\epsilon+\lambda)^2t}{H^3}+\frac{\lambda\epsilon t^{\frac{1}{2}}}{H^{\frac{3}{2}}}.
    $$
    Next we apply $L_0$ to both sides of (\ref{D_t^2q}):
    $$
    \begin{aligned}
    L_0D_t^2q=&-\frac{\lambda i}{4\pi}\frac{t(\dddot{z_1}-\dddot{z_2})}{(\zeta-z_1)(\zeta-z_2)}+\frac{\lambda i}{2\pi}(\ddot{z_1}-\ddot{z_2})\left(\frac{L_0\zeta-\frac{1}{2}t\dot{z_1}}{(\zeta-z_1)^2(\zeta-z_2)}+\frac{L_0\zeta-\frac{1}{2}t\dot{z_2}}{(\zeta-z_1)(\zeta-z_2)^2}\right)\\
    &+\frac{\lambda it}{2\pi}(\ddot{z_1}-\ddot{z_2})\left(\frac{D_t\zeta-\dot{z_1}}{(\zeta-z_1)^2(\zeta-z_2)}+\frac{D_t\zeta-\dot{z_2}}{(\zeta-z_1)(\zeta-z_2)^2}\right)\\
    &+\frac{\lambda i}{\pi}(\dot{z_1}-\dot{z_2})L_0\left(\frac{D_t\zeta-\dot{z_1}}{(\zeta-z_1)^2(\zeta-z_2)}+\frac{D_t\zeta-\dot{z_2}}{(\zeta-z_1)(\zeta-z_2)^2}\right)\\
    &+\frac{\lambda it}{4\pi}(\dot{z_1}-\dot{z_2})D_t\left(\frac{D_t\zeta-\dot{z_1}}{(\zeta-z_1)^2(\zeta-z_2)}+\frac{D_t\zeta-\dot{z_2}}{(\zeta-z_1)(\zeta-z_2)^2}\right)\\
    &+\frac{\lambda i}{2\pi}(z_1-z_2)L_0D_t\left(\frac{D_t\zeta-\dot{z_1}}{(\zeta-z_1)^2(\zeta-z_2)}+\frac{D_t\zeta-\dot{z_2}}{(\zeta-z_1)(\zeta-z_2)^2}\right),
    \end{aligned}
    $$
    which implies that
    $$
    \norm{L_0D_t^2q}_{H^{s-1}}\lesssim\frac{\lambda(\epsilon+\lambda) t}{H^3}+\frac{\lambda(\epsilon+\lambda) t^{\frac{1}{2}}}{H^{\frac{5}{2}}}.
    $$
\end{proof}

\subsection{Comparison estimates for the good unknowns}\label{subs_good}
In this section we compare the original variables with the new variables, such as $\tilde{\theta}$ and $\tilde{\sigma}$, and claim that they differ by an acceptable error. This observation allows us to close the bootstrap assumptions by investigating the new variables.
\begin{lemma}\label{lem_theta_t_zeta_t}
    We have for $t\in [0,T]$
    \begin{equation}\label{theta_t_zeta_t}
    \norm{D_t\tilde{\theta}-2D_t\zeta}_{H^{s+\frac{1}{2}}}\lesssim\frac{\epsilon^2}{t^\frac{1}{4}}+\frac{\lambda}{H^\frac{3}{2}},
    \end{equation}
\begin{equation}\label{sigma_t_zeta_tt}
    \norm{D_t\tilde{\sigma}-2D_t^2\zeta}_{H^s}\lesssim\frac{\epsilon^2}{t^\frac{1}{4}}+\frac{\lambda(\epsilon+\lambda)}{H^{\frac{5}{2}}},
    \end{equation}
    \begin{equation}
        \norm{D_t\tilde{\theta}-2D_t\zeta}_{W^{s-2,\infty}}\lesssim\frac{\epsilon^2}{t^{\frac{2}{3}}}+\frac{\lambda}{H^2}
    \end{equation}
    and
    \begin{equation}
    \norm{D_t\tilde{\sigma}-2D_t^2\zeta}_{W^{s-2,\infty}}\lesssim\frac{\epsilon^2}{t^{\frac{2}{3}}}+\frac{\lambda(\epsilon+\lambda)}{H^3}.
    \end{equation}
\end{lemma}
\begin{proof}
    First there holds
\begin{equation}\label{D_t_theta}
    \begin{aligned}
    D_t\tilde{\theta}=&\ 2D_t\zeta-(\mathfrak{H}+\bar{\mathfrak{H}})D_t\zeta-[D_t\zeta,\mathfrak{H}]\frac{\partial_\alpha}{\zeta_\alpha}(\zeta-\bar\zeta)-2(q+\bar q)\\
    =&\ 2D_t\zeta-\frac{1}{\pi i}\int\left(\frac{\zeta_\beta(\beta,t)}{\zeta(\alpha,t)-\zeta(\beta,t)}-\frac{1}{\alpha-\beta}\right)D_t\zeta d\beta\\
    &+\frac{1}{\pi i}\int\left(\frac{\bar\zeta_\beta(\beta,t)}{\bar\zeta(\alpha,t)-\bar\zeta(\beta,t)}-\frac{1}{\alpha-\beta}\right)D_t\zeta d\beta\\
    &-[D_t\zeta,\mathfrak{H}]\frac{\partial_\alpha}{\zeta_\alpha}(\zeta-\alpha+\alpha-\bar\zeta)-2(q+\bar q)\\
    =&\ 2D_t\zeta+\frac{1}{\pi i}\int\left(\frac{\zeta(\alpha,t)-\alpha-\zeta(\beta,t)+\beta}{(\zeta(\alpha,t)-\zeta(\beta,t))(\alpha-\beta)}-\frac{\zeta_\beta(\beta,t)-1}{\zeta(\alpha,t)-\zeta(\beta,t)}\right)D_t\zeta d\beta\\
    &-\frac{1}{\pi i}\int\left(\frac{\bar\zeta(\alpha,t)-\alpha-\bar\zeta(\beta,t)+\beta}{(\bar\zeta(\alpha,t)-\bar\zeta(\beta,t))(\alpha-\beta)}-\frac{\bar\zeta_\beta(\beta,t)-1}{\bar\zeta(\alpha,t)-\bar\zeta(\beta,t)}\right)D_t\zeta d\beta\\
    &-[D_t\zeta,\mathfrak{H}]\frac{\partial_\alpha}{\zeta_\alpha}(\zeta-\alpha+\alpha-\bar\zeta)-2(q+\bar q)
    \end{aligned}
    \end{equation}
and
\begin{equation}\label{D_t_sigma}
    \begin{aligned}
    D_t\tilde{\sigma}=&\ 2D_t^2\zeta-(\mathfrak{H}+\bar{\mathfrak{H}})D_t^2\zeta-[D_t\zeta,\mathfrak{H}]\frac{\partial_\alpha}{\zeta_\alpha}D_t\zeta-[D_t\bar\zeta,\bar{\mathfrak{H}}]\frac{\partial_\alpha}{\bar\zeta_\alpha}D_t\zeta\\
    &-D_t[D_t\zeta,\mathfrak{H}]\frac{\partial_\alpha}{\zeta_\alpha}(\zeta-\bar\zeta)-2D_t(q+\bar q).
    \end{aligned}
    \end{equation}
Hence by Lemma \ref{app_S1} and Proposition \ref{pro_q}:
$$
    \norm{D_t\tilde{\theta}-2D_t\zeta}_{H^{s}}\lesssim\norm{\zeta_\alpha-1}_{H^s}\norm{D_t\zeta}_{W^{s-2,\infty}}+\norm{\zeta_\alpha-1}_{W^{s-2,\infty}}\norm{D_t\zeta}_{H^s}+\norm{q}_{H^s}\lesssim\frac{\epsilon^2}{t^\frac{1}{4}}+\frac{\lambda}{H^\frac{3}{2}}
    $$
and
$$
  \begin{aligned}  \norm{D_t\tilde{\sigma}-2D_t^2\zeta}_{H^s}\lesssim&\norm{D_t^2\zeta}_{H^s}\norm{\zeta_\alpha-1}_{W^{s-2,\infty}}+\norm{D_t^2\zeta}_{W^{s-2,\infty}}\norm{\zeta_\alpha-1}_{H^s}\\
  &+\norm{D_t\zeta}_{H^s}\norm{D_t\zeta}_{W^{s-2,\infty}}+\norm{D_tq}_{H^s}\\
  \lesssim&\ \frac{\epsilon^2}{t^\frac{1}{4}}+\frac{\lambda(\epsilon+\lambda)}{H^{\frac{5}{2}}}.
    \end{aligned}
    $$
    
For $\norm{D_t\tilde{\theta}-2D_t\zeta}_{\dot{H}^{s+\frac{1}{2}}}$, it's essential to rewrite the expression of $D_t\tilde{\theta}$ by utilizing integration by parts:
$$
    \begin{aligned}
    D_t\tilde{\theta}=&\ 2D_t\zeta-\frac{1}{\pi i}\int \ln \left(\frac{\zeta(\alpha,t)-\zeta(\beta,t)}{\alpha-\beta}\right)\partial_\beta D_t\zeta d\beta+\frac{1}{\pi i}\int \ln \left(\frac{\bar\zeta(\alpha,t)-\bar\zeta(\beta,t)}{\alpha-\beta}\right)\partial_\beta D_t\zeta d\beta\\
    &-[D_t\zeta,\mathfrak{H}]\frac{\partial_\alpha}{\zeta_\alpha}(\zeta-\alpha+\alpha-\bar\zeta)-2(q+\bar q).
    \end{aligned}
    $$
We notice that
$$
    \begin{aligned}
    \partial_\alpha\int \ln \left(\frac{\zeta(\alpha,t)-\zeta(\beta,t)}{\alpha-\beta}\right)\partial_\beta D_t\zeta d\beta=&\int(\partial_\alpha+\partial_\beta-\partial_\beta)\ln \left(\frac{\zeta(\alpha,t)-\zeta(\beta,t)}{\alpha-\beta}\right)\partial_\beta D_t\zeta d\beta\\
    =&\int\frac{\zeta_\alpha(\alpha,t)-\zeta_\beta(\beta,t)}{\zeta(\alpha,t)-\zeta(\beta,t)}\partial_\beta D_t\zeta d\beta\\
    &+\int\left(\frac{\zeta_\beta(\beta,t)}{\zeta(\alpha,t)-\zeta(\beta,t)}-\frac{1}{\alpha-\beta}\right)\partial_\beta D_t\zeta d\beta\\
    =&\int\frac{\zeta_\alpha(\alpha,t)-\zeta_\beta(\beta,t)}{\zeta(\alpha,t)-\zeta(\beta,t)}\partial_\beta D_t\zeta d\beta\\
    &-\int\frac{\zeta(\alpha,t)-\alpha-\zeta(\beta,t)+\beta}{(\zeta(\alpha,t)-\zeta(\beta,t))(\alpha-\beta)}\partial_\beta D_t\zeta d\beta\\
    &+\int\frac{\zeta_\beta(\beta,t)-1}{\zeta(\alpha,t)-\zeta(\beta,t)}\partial_\beta D_t\zeta d\beta,
    \end{aligned}
    $$
which means that
$$
    \begin{aligned}
    \norm{\int \ln \left(\frac{\zeta(\alpha,t)-\zeta(\beta,t)}{\alpha-\beta}\right)\partial_\beta D_t\zeta d\beta}_{\dot{H}^{s+\frac{1}{2}}}\lesssim&\norm{\zeta_\alpha-1}_{W^{s-2,\infty}}\norm{D_t\zeta}_{H^{s+\frac{1}{2}}}+\norm{\zeta_\alpha-1}_{H^s}\norm{D_t\zeta}_{W^{s-2,\infty}}\\
    \lesssim&\frac{\epsilon^2}{t^\frac{1}{4}}.
    \end{aligned}
    $$
    Analogously,
$$
     \norm{\int \ln \left(\frac{\bar\zeta(\alpha,t)-\bar\zeta(\beta,t)}{\alpha-\beta}\right)\partial_\beta D_t\zeta d\beta}_{\dot{H}^{s+\frac{1}{2}}}\lesssim\frac{\epsilon^2}{t^\frac{1}{4}}.
    $$
    In summary, we conclude that
$$
     \norm{D_t\tilde{\theta}-2D_t\zeta}_{\dot H^{s+\frac{1}{2}}}\lesssim\frac{\epsilon^2}{t^\frac{1}{4}}+\frac{\lambda}{H^\frac{3}{2}}.
    $$
    For the $L^\infty$ norm, we use (\ref{D_t_theta}) again and recall Lemma
    \ref{lem_Hf_infty} and Lemma \ref{lem_S_1}:
    $$
    \begin{aligned}
    \norm{D_t\tilde{\theta}-2D_t\zeta}_{W^{s-2,\infty}}\lesssim&\ \frac{1}{t^{\frac{1}{2}}}\norm{\zeta_\alpha-1}_{W^{s-1,\infty}}\norm{D_t\zeta}_{H^{s-1}}+\ln (2+t)\norm{\zeta_\alpha-1}_{W^{s-2,\infty}}\norm{D_t\zeta}_{W^{s-2,\infty}}\\
    &+\norm{q}_{W^{s-2,\infty}}\\
    \lesssim&\ \frac{\epsilon^2}{t^{\frac{2}{3}}}+\frac{\lambda}{H^2}.
    \end{aligned}
    $$
    Similarly,
    $$
    \norm{D_t\tilde{\sigma}-2D_t^2\zeta}_{W^{s-2,\infty}}\lesssim\frac{\epsilon^2}{t^{\frac{2}{3}}}+\frac{\lambda(\epsilon+\lambda)}{H^3}.
    $$
    
\end{proof}

\begin{lemma}\label{lem_theta_alpha}
    For $t\in[0,T]$ we have
    \begin{equation}\label{theta_alpha_2}
    \norm{\tilde{\theta}_\alpha-2(\zeta_\alpha-1)}_{H^s}\lesssim\frac{\epsilon^2}{t^\frac{1}{2}}.
    \end{equation}
    and
    \begin{equation}\label{theta_alpha_infty}
    \norm{\tilde{\theta}_\alpha-2(\zeta_\alpha-1)}_{W^{s-1,\infty}}\lesssim\frac{\epsilon^2}{t^{\frac{2}{3}}}\ln (2+t).
    \end{equation}
\end{lemma}
\begin{proof}
    Recall the definition of $\tilde\theta$ and the fact that $(I-\mathfrak{H})(\bar\zeta-\alpha)=0$:
$$
    \begin{aligned}
    \tilde{\theta}_\alpha=\partial_\alpha(I-\mathfrak{H})(\zeta-\alpha)=-[\zeta_\alpha,\mathfrak{H}]\frac{\partial_\alpha}{\zeta_\alpha}(\zeta-\alpha)+(I-\mathfrak{H})(\zeta_\alpha-1).
    \end{aligned}
    $$
Since $[\frac{\partial_\alpha}{\zeta_\alpha},\mathfrak{H}]=0$, it's obvious that $(I-\bar{\mathfrak{H}})\frac{\zeta_\alpha-1}{\bar\zeta_\alpha}=0$, which implies
\begin{equation}\label{theta_alpha}
    \begin{aligned}
    \tilde{\theta}_\alpha=&\ (I-\bar{\mathfrak{H}})\left[(\zeta_\alpha-1)\left(\frac{1}{\bar{\zeta}_\alpha}-1\right)\right]+(I-\bar{\mathfrak{H}})(\zeta_\alpha-1)+(I-\mathfrak{H})(\zeta_\alpha-1)-[\zeta_\alpha,\mathfrak{H}]\frac{\partial_\alpha}{\zeta_\alpha}(\zeta-\alpha)\\
    =&\ 2(\zeta_\alpha-1)+(I-\bar{\mathfrak{H}})\left[(\zeta_\alpha-1)\left(\frac{1}{\bar{\zeta}_\alpha}-1\right)\right]-(\mathfrak{H}+\bar{\mathfrak{H}})(\zeta_\alpha-1)-[\zeta_\alpha,\mathfrak{H}]\frac{\partial_\alpha}{\zeta_\alpha}(\zeta-\alpha)
    \end{aligned}
    \end{equation}
By imitating the proof of the preceding theorem, we obtain
$$
    \norm{\tilde{\theta}_\alpha-2(\zeta_\alpha-1)}_{H^s}\lesssim\norm{\zeta_\alpha-1}_{W^{s-2,\infty}}\norm{\zeta_\alpha-1}_{H^s}\lesssim\frac{\epsilon^2}{t^\frac{1}{2}}.
    $$
    \eqref{theta_alpha_infty} follows from \eqref{theta_alpha}.
\end{proof}

The next two lemmas are about estimates of vector fields. However, as $\Omega_0$ is not a strict vector fields, we refer to the equation
$$
\Omega_0\partial_\alpha=L_0\partial_t-\frac{1}{2}t(\partial_t^2-i\partial_\alpha).
$$
Note that $(\partial_t^2-i\partial_\alpha)\tilde{\theta}$ and $(\partial_t^2-i\partial_\alpha)\tilde{\sigma}$ are cubic, which implies the slow growth of $\norm{\Omega_0\tilde{\theta}_\alpha}_{L^2}$ and $\norm{\Omega_0\tilde{\sigma}_\alpha}_{L^2}$

\begin{lemma}\label{lem_L_0D_t_theta}
    We have
$$
\norm{L_0D_t\tilde{\theta}-2L_0D_t\zeta}_{H^{s-1}}\lesssim\frac{\epsilon^2}{t^{\frac{1}{6}-\delta_0}}\ln (2+t)+\frac{\lambda(\epsilon+\lambda)t}{H^{\frac{5}{2}}},
$$
$$
    \norm{L_0\tilde{\theta}_\alpha-2L_0\zeta_\alpha}_{H^{s-1}}\lesssim\frac{\epsilon^2}{t^{\frac{1}{6}-\delta_0}}\ln (2+t)
    $$
and
$$
    \norm{L_0D_t\tilde{\sigma}-2L_0D_t^2\zeta}_{H^{s-1}}\lesssim\frac{\epsilon^2}{t^{\frac{1}{6}-\delta_0}}\ln (2+t)+\frac{\lambda(\epsilon+\lambda)t}{H^3
    }+\frac{\lambda(\epsilon+\lambda) t^{\frac{1}{2}}}{H^{\frac{5}{2}}}.
    $$
\end{lemma}
\begin{proof}
    Apply $L_0$ on (\ref{D_t_theta}) and we have
$$
    L_0D_t\tilde{\theta}=2L_0D_t\zeta-[L_0,\mathfrak{H}+\bar{\mathfrak{H}}]D_t\zeta-(\mathfrak{H}+\bar{\mathfrak{H}})L_0D_t\zeta-L_0[D_t\zeta,\mathfrak{H}]\frac{\partial_\alpha}{\zeta_\alpha}(\zeta-\bar\zeta)-2L_0(q+\bar q),
    $$
where we obtain that
$$
    \begin{aligned}
    \norm{L_0D_t\tilde{\theta}-2L_0D_t\zeta}_{H^{s-1}}\lesssim&\norm{\zeta_\alpha-1}_{W^{s-1,\infty}}\norm{L_0D_t\zeta}_{H^{s-1}}+\norm{D_t\zeta}_{W^{s-1,\infty}}\norm{L_0\zeta_\alpha}_{H^{s-1}}\\
    &+\norm{L_0 q}_{H^{s-1}}\\
    \lesssim&\ \frac{\epsilon^2}{t^{\frac{1}{6}-\delta_0}}\ln (2+t)+\frac{\lambda(\epsilon+\lambda)t}{H^{\frac{5}{2}}}.
    \end{aligned}
    $$
Similarly we notice that
$$
    \begin{aligned}
    L_0\tilde{\theta}_\alpha=&\ 2L_0\zeta_\alpha+L_0(I-\bar{\mathfrak{H}})\left[(\zeta_\alpha-1)\left(\frac{1}{\bar{\zeta}_\alpha}-1\right)\right]-[L_0,\mathfrak{H}+\bar{\mathfrak{H}}](\zeta_\alpha-1)\\
    &-(\mathfrak{H}+\bar{\mathfrak{H}})L_0\zeta_\alpha-L_0[\zeta_\alpha,\mathfrak{H}]\frac{\partial_\alpha}{\zeta_\alpha}(\zeta-\alpha),
    \end{aligned}
    $$
which implies that
$$
    \norm{L_0\tilde{\theta}_\alpha-2L_0\zeta_\alpha}_{H^{s-1}}\lesssim\norm{L_0\zeta_\alpha}_{H^{s-1}}\norm{\zeta_\alpha-1}_{W^{s-1,\infty}}\lesssim\frac{\epsilon^2}{t^{\frac{1}{6}-\delta_0}}\ln (2+t).
    $$
Utilizing (\ref{D_t_sigma}), we also prove that
$$
    \norm{L_0D_t\tilde{\sigma}-2L_0D_t^2\zeta}_{H^{s-1}}\lesssim\frac{\epsilon^2}{t^{\frac{1}{6}-\delta_0}}\ln (2+t)+\frac{\lambda(\epsilon+\lambda)t}{H^3
    }+\frac{\lambda(\epsilon+\lambda)^2t}{H^{\frac{7}{2}}}+\frac{\lambda(\epsilon+\lambda) t^{\frac{1}{2}}}{H^{\frac{5}{2}}}.
    $$
\end{proof}

\begin{lemma}\label{lem_Omega}
    Under the bootstrap assumptions, we have
    \begin{equation}
        \norm{\Omega_0\tilde{\theta}_\alpha}_{H^{s-1}}+\norm{\Omega_0\tilde{\theta}_\alpha}_{W^{s-2,\infty}}\lesssim\epsilon t^{\delta_0},
    \end{equation}
    \begin{equation}\label{Omega_zeta_alpha_2}
        \norm{\Omega_0(\zeta_\alpha-1)}_{H^{s-1}}\lesssim\epsilon^2t^{\frac{1}{2}}+\epsilon t^{\delta_0}
    \end{equation}
    and
    \begin{equation}\label{Omega_zeta_alpha_infty}
        \norm{\Omega_0(\zeta_\alpha-1)}_{W^{s-2,\infty}}\lesssim\epsilon^2t^{\frac{1}{4}}\ln (2+t)+\epsilon t^{\delta_0}.
    \end{equation}
\end{lemma}
\begin{proof}
    We only consider $\norm{\Omega_0\tilde{\theta}_\alpha}_{L^2}$ and $\norm{\Omega_0\tilde{\theta}_\alpha}_{L^\infty}$ as the most representative case. Since
    $$
    \Omega_0\partial_\alpha=L_0\partial_t-\frac{1}{2}t(\partial_t^2-i\partial_\alpha),
    $$
    there holds
    \begin{equation}\label{Omega_theta_alpha}
    \Omega_0\tilde{\theta}_\alpha=L_0D_t\tilde{\theta}-L_0b\tilde{\theta}_\alpha-bL_0\tilde{\theta}_\alpha-\frac{1}{2}t(D_t^2-iA\partial_\alpha)\tilde{\theta}+\frac{1}{2}t(D_tb\partial_\alpha+2b\partial_{\alpha t}+b^2\partial_\alpha^2-i(A-1)\partial_\alpha)\tilde{\theta}.
    \end{equation}
    By \ref{b_exp}, Proposition \ref{pro_q} and Lemma \ref{lem_real_proj}, we can prove
$$\begin{aligned}
        \norm{b}_{L^2}\lesssim\norm{D_t\zeta}_{L^2}\norm{\zeta_\alpha-1}_{L^\infty}+\norm{q}_{L^2}
        \lesssim\frac{\epsilon^{2}}{t^{\frac{1}{2}}}+\frac{\lambda}{H^{\frac{3}{2}}}.
    \end{aligned}$$
With the aid of Lemma \ref{lem_Hf_infty} we deduce that
$$\begin{aligned}
        \norm{b}_{L^\infty}\lesssim&\norm{D_t\zeta\frac{\bar{\zeta}_\alpha-1}{\zeta_\alpha}}_{L^\infty}+\frac{1}{t^2}\norm{D_t\zeta\frac{\bar{\zeta}_\alpha-1}{\zeta_\alpha}}_{H^1}+\norm{D_t\zeta\frac{\bar{\zeta}_\alpha-1}{\zeta_\alpha}}_{L^\infty}\ln (2+t)+\norm{q}_{L^\infty}\\
        \lesssim&\ \frac{\epsilon^2}{t^{\frac{3}{4}}}\ln (2+t)+\frac{\lambda}{H^2}.
    \end{aligned}$$
    For $D_tb$ we consider
    $$
    (I-\mathfrak{H})D_tb=[D_t\zeta,\mathfrak{H}]\frac{b_\alpha}{\zeta_\alpha}+D_t\left[-(I-\mathfrak{H})D_t\zeta\frac{\bar\zeta_\alpha-1}{\zeta_\alpha}+2q\right].
    $$
    Lemma \ref{app_S1} and Lemma \ref{lem_real_proj} imply
    $$
    \begin{aligned}
    \norm{D_tb}_{L^2}\lesssim&\norm{\partial_\alpha D_t\zeta}_{L^\infty}\norm{b}_{L^2}+\norm{D_t^2\zeta}_{L^\infty}\norm{\zeta_\alpha-1}_{L^2}+\norm{D_t\zeta}_{L^2}\norm{D_t\zeta_\alpha}_{L^\infty}+\norm{D_tq}_{L^2}\\
    \lesssim&\ \frac{\epsilon^2}{t^{\frac{1}{2}}}+\frac{\lambda(\epsilon+\lambda)}{H^{\frac{5}{2}}}.
    \end{aligned}
    $$
    Moreover, from Lemma \ref{lem_Hf_infty} and Lemma \ref{lem_S_1} we have
    $$
    \begin{aligned}
    &\norm{D_tb}_{L^\infty}\\
    \lesssim\ &\norm{\partial_\alpha D_t\zeta}_{L^\infty}\norm{b}_{L^\infty}\ln (2+t)+\norm{D_t^2\zeta}_{L^\infty}\norm{\zeta_\alpha-1}_{L^\infty}\ln (2+t)+\norm{D_t\zeta}_{L^\infty}\norm{D_t\zeta_\alpha}_{L^\infty}\ln (2+t)\\
    &+\norm{D_tq}_{L^\infty}+\frac{1}{t^{\frac{1}{2}}}\norm{\partial_\alpha D_t\zeta}_{W^{1,\infty}}\norm{b}_{H^1}+\frac{1}{t^2}\norm{D_t^2\zeta}_{W^{1,\infty}}\norm{\zeta_\alpha-1}_{H^1}+\frac{1}{t^2}\norm{D_t\zeta}_{H^1}\norm{D_t\zeta_\alpha}_{W^{1,\infty}}\\
    \lesssim\ &\frac{\epsilon^2}{t^{\frac{3}{4}}}\ln (2+t)+\frac{\lambda(\epsilon+\lambda)}{H^3}.
    \end{aligned}
    $$
    Similarly, by checking $(I-\mathfrak{H})L_0b$:
    $$
    (I-\mathfrak{H})L_0b=[L_0\zeta,\mathfrak{H}]\frac{b_\alpha}{\zeta_\alpha}+L_0\left[-(I-\mathfrak{H})D_t\zeta\frac{\bar\zeta_\alpha-1}{\zeta_\alpha}+2q\right]
    $$
    we obtain
    $$
    \norm{L_0b}_{L^2}+\norm{L_0b}_{L^\infty}\lesssim\frac{\epsilon^2}{t^{\frac{1}{4}-\delta_0}}\ln (2+t)+\frac{\lambda(\epsilon+\lambda)t}{H^{\frac{5}{2}}}.
    $$
    From (\ref{A_exp}), we obtain
    $$
    \norm{A-1}_{L^2}\lesssim\frac{\epsilon^2}{t^{\frac{1}{2}}}+\frac{\lambda(\epsilon+\lambda)}{H^\frac{3}{2}}
    $$
    and
    $$
    \norm{A-1}_{L^\infty}\lesssim\frac{\epsilon^2}{t^{\frac{3}{4}}}\ln (2+t)+\frac{\lambda(\epsilon+\lambda)}{H^2}.
    $$
    To estimate $\norm{\Omega_0\tilde{\theta}_\alpha}_{L^2}$ and $\norm{\Omega_0\tilde{\theta}_\alpha}_{L^\infty}$, it remains to treat the term $(D_t^2-iA\partial_\alpha)\tilde{\theta}$. Recall that
\begin{equation}
    (D^2_t-iA\partial_\alpha)\tilde{\theta}=G_c+G_d,
\end{equation}
where
\begin{equation}
    G_c=-2[\bar{\mathfrak{F}},\mathfrak{H}\frac{1}{\zeta_{\alpha}}+
    \bar{\mathfrak{H}}\frac{1}{\bar{\zeta}_{\alpha}}]\bar{\mathfrak{F}}_{\alpha}+\frac{1}{\pi i}\int\left(\frac{D_{t}\zeta(\alpha,t)-D_{t}\zeta(\beta,t)}{\zeta(\alpha,t)-\zeta(\beta,t)}\right)^2(\zeta-\bar{\zeta})_{\beta}d\beta
\end{equation}
and
\begin{equation}
    G_d=-2[\bar{q},\mathfrak{H}]\frac{\bar{\mathfrak{F}}_{\alpha}}{\zeta_{\alpha}}-2[\bar{\mathfrak{F}},\mathfrak{H}]\frac{\bar{q}_{\alpha}}{\zeta_{\alpha}}-2[\bar{q},\mathfrak{H}]\frac{\bar{q}_{\alpha}}{\zeta_{\alpha}}-4D_tq.
\end{equation}
Adopting Lemma \ref{app_S1} leads to
$$
\norm{G_c}_{L^2}\lesssim\norm{\zeta_\alpha-1}_{L^\infty}\norm{\partial_\alpha\mathfrak{F}}_{L^\infty}\norm{\mathfrak{F}}_{L^2}+\norm{\zeta_\alpha-1}_{L^\infty}\norm{\partial_\alpha D_t\zeta}_{L^\infty}\norm{D_t\zeta}_{L^2}\lesssim\frac{\epsilon^3}{t}\ln (2+t)
$$
and
$$
\norm{G_d}_{L^2}\lesssim\norm{\mathfrak{F}}_{L^2}\norm{q_\alpha}_{L^\infty}+\norm{q}_{L^2}\norm{q_\alpha}_{L^\infty}+\norm{D_tq}_{L^2}\lesssim\frac{\lambda(\epsilon+\lambda)}{H^{\frac{5}{2}}}.
$$
Similarly there holds
$$
\norm{G_c}_{L^\infty}\lesssim\frac{\epsilon^3}{t^{\frac{5}{4}}}\ln (2+t)
$$
and
$$
\norm{G_d}_{L^\infty}\lesssim\frac{\lambda(\epsilon+\lambda)}{H^3}.
$$
Return to (\ref{Omega_theta_alpha}), we deduce that
$$
\begin{aligned}
\norm{\Omega_0\tilde{\theta}_\alpha}_{L^2}\lesssim&\norm{L_0D_t\tilde{\theta}}_{L^2}+\norm{L_0b}_{L^2}\norm{\tilde{\theta}_\alpha}_{L^\infty}+\norm{b}_{L^\infty}\norm{L_0\tilde{\theta}_\alpha}_{L^2}+t\left(\norm{G_c}_{L^2}+\norm{G_d}_{L^2}\right)\\
&+t\left(\norm{D_tb}_{L^2}\norm{\tilde{\theta}_\alpha}_{L^\infty}+\norm{b}_{L^2}\norm{\tilde{\theta}_{\alpha t}}_{L^\infty}+\norm{b}^2_{L^2}\norm{\tilde{\theta}_{\alpha\alpha}}_{L^\infty}+\norm{A-1}_{L^2}\norm{\tilde{\theta}_\alpha}_{L^\infty}\right)\\
\lesssim&\ \epsilon t^{\delta_0}.
\end{aligned}
$$
and
$$
\begin{aligned}
\norm{\Omega_0\tilde{\theta}_\alpha}_{L^\infty}\lesssim&\norm{L_0D_t\tilde{\theta}}_{H^1}+\norm{L_0b}_{L^\infty}\norm{\tilde{\theta}_\alpha}_{L^\infty}+\norm{b}_{L^\infty}\norm{L_0\tilde{\theta}_\alpha}_{H^1}+t\left(\norm{G_c}_{L^\infty}+\norm{G_d}_{L^\infty}\right)\\
&+t\left(\norm{D_tb}_{L^\infty}\norm{\tilde{\theta}_\alpha}_{L^\infty}+\norm{b}_{L^\infty}\norm{\tilde{\theta}_{\alpha t}}_{L^\infty}+\norm{b}^2_{L^\infty}\norm{\tilde{\theta}_{\alpha\alpha}}_{L^\infty}+\norm{A-1}_{L^\infty}\norm{\tilde{\theta}_\alpha}_{L^\infty}\right)\\
\lesssim&\ \epsilon t^{\delta_0}.
\end{aligned}
$$
For (\ref{Omega_zeta_alpha_2}) and \eqref{Omega_zeta_alpha_infty}, note that
$$
(\partial_t^2-i\partial_\alpha)(\zeta-\alpha)=-(D_tb\partial_\alpha+2b\partial_{\alpha t}+b^2\partial_\alpha^2-i(A-1)\partial_\alpha)\zeta
$$
and we follow the steps as above.
\end{proof}
\begin{rem}
    In fact, if we take Lemma \ref{lem_b} and \ref{lem_A} into consideration, there holds
    $$\norm{\Omega_0(\zeta_\alpha-1)}_{H^{s-1}}+\norm{\Omega_0(\zeta_\alpha-1)}_{W^{s-2,\infty}}\lesssim\epsilon t^{\delta_0}.$$
    It can be achieved by adopting Lemma \ref{lem_b} and \ref{lem_A} to refine Lemma \ref{lem_Omega}, and then refining  \ref{lem_b} and \ref{lem_A}. After several iterations of this step, the result follows. 
\end{rem}

The upcoming lemma serves as a key lemma in this paper. It provides a method to transfer derivatives among distinct functions, which will enable us to derive the necessary time-decay estimates. We introduce two essential equations which follow directly from the definition of the vector field:
\begin{equation}\label{transfer_partial_t}
    f=-\frac{2\alpha}{it}f_t+\frac{2}{it}\Omega_0f
\end{equation}
and
\begin{equation}\label{transfer_partial_alpha}
    f=\frac{4\alpha^2}{it^2}f_\alpha-\frac{4\alpha}{it^2}L_0f+\frac{2}{it}\Omega_0f.
\end{equation}
The proof of the subsequent lemma requires only direct calculation.
\begin{lemma}\label{lem_transfer_d}
    For any functions $f(\cdot,t),g(\cdot,t)$, there holds
    \begin{equation}\label{transfer_d_1}
    f\bar g_\alpha+f_\alpha\bar g=\frac{2}{it}(\Omega_0f\bar g_\alpha+L_0f\bar g_t-f_\alpha\overline{\Omega_0g}-f_t\overline{L_0g}),
    \end{equation}
    \begin{equation}\label{transfer_d_2}
    f g_\alpha-f_\alpha g=\frac{2}{it}(\Omega_0f g_\alpha+L_0f g_t-f_\alpha{\Omega_0g}-f_t{L_0g}),
    \end{equation}
    \begin{equation}\label{transfer_d_3}
    f\bar g_t+f_t\bar g=\frac{2}{it}(\Omega_0f\bar g_t-f_t\overline{\Omega_0g}),
    \end{equation}
    \begin{equation}\label{transfer_d_4}
    f g_t-f_t g=\frac{2}{it}(\Omega_0f g_t-f_t{\Omega_0g}).
    \end{equation}
    and
    \begin{equation}\label{transfer_d_var_1}
    \begin{aligned}
    &(f_t(\alpha,t)-f_t(\beta,t))\bar g(\beta,t)+(f(\alpha,t)-f(\beta,t))\bar g_t(\beta,t)\\
    =&\ \frac{2}{it}(\Omega_0f(\alpha,t)-\Omega_0f(\beta,t)-(\alpha-\beta)f_t(\alpha,t))\bar g_t(\beta,t)-\frac{2}{it}(f_t(\alpha,t)-f_t(\beta,t))\overline{\Omega_0g(\beta,t)},
    \end{aligned}
    \end{equation}
    \begin{equation}\label{transfer_d_var_2}
    \begin{aligned}
    &(f_t(\alpha,t)-f_t(\beta,t)) g(\beta,t)-(f(\alpha,t)-f(\beta,t))g_t(\beta,t)\\
    =&\ \frac{2}{it}(\Omega_0f(\alpha,t)-\Omega_0f(\beta,t)-(\alpha-\beta)f_t(\alpha,t))g_t(\beta,t)-\frac{2}{it}(f_t(\alpha,t)-f_t(\beta,t)){\Omega_0g(\beta,t)}.
    \end{aligned}
    \end{equation}
\end{lemma}

The following lemma is a simple corollary of Lemma 5.3 from \cite{su2025new}:
\begin{lemma}\label{lem_int_alpha_away_t}
    Fix a positive number $\gamma<\frac{1}{3}$. Define
$$
    I(f_1,f_2,f_3)=\frac{1}{\pi i}\int\frac{(f_1(\alpha)-f_1(\beta))(f_2(\alpha)-f_2(\beta))}{(\alpha-\beta)^2}f_3(\beta)d\beta
    $$
then
$$
\begin{aligned}
    &\norm{I}_{L^2(|\alpha|\leq t^{1-\gamma})}\\
    &\qquad\lesssim\left(\frac{1}{t}(\norm{L_0\partial_\alpha f_1}_{H^1}+\norm{\Omega_0\partial_\alpha f_1}_{H^1})+\frac{1}{t^{2\gamma}}\norm{\partial_\alpha f_1}_{W^{1,\infty}}+\frac{1}{t^{1-\gamma}}\norm{f_1}_{L^\infty}\right)\norm{f_2}_{L^\infty}\norm{f_3}_{L^2}
    \end{aligned}
    $$
and
$$
    \norm{I}_{L^2(|\alpha|\geq t^{1+\gamma})}\lesssim\left(\frac{1}{t^{\frac{1}{2}+\frac{\gamma}{2}}}(\norm{L_0\partial_\alpha f_1}_{L^2}+\norm{\Omega_0\partial_\alpha f_1}_{L^2})+\frac{1}{t^{1+\gamma}}\norm{f_1}_{L^\infty}\right)\norm{f_2}_{L^\infty}\norm{f_3}_{L^2}.\footnote{We may replace $\norm{f_2}_{L^\infty}$ and $\norm{f_3}_{L^2}$ with $\norm{f_2}_{L^2}$ and $\norm{f_3}_{L^\infty}$ respectively in the above two inequalities.}
    $$
\end{lemma}

In the following series of lemmas, we consider integrals of the form when $\alpha$ is close to $t$:
$$
I=\frac{1}{\pi i}\int\frac{\Pi_{i\leq k}(f_i(\alpha,t)-f_i(\beta,t))}{(\alpha-\beta)^{k+1}}g(\beta,t)d\beta\quad and\quad I'=\frac{1}{\pi i}\int\frac{\Pi_{i\leq k}(f_i(\alpha,t)-f_i(\beta,t))}{(\alpha-\beta)^{k}}g(\beta,t)d\beta
$$
which will play a vital role in the energy estimates. We provide two versions of estimates for the reason that we lack the control of $\norm{\Omega_0D_t\zeta}_{L^2}$. Although some of the results may seem rather complex, one can treat $b$ as $0$ to see which terms dominate.
\begin{lemma}\label{lem_(I-H)fg_main}
    Suppose that $\mu\in(0,\frac{1}{2})$, denote $S(t)=\{\alpha\in\mathbb{R}|t^{1-\mu}\leq|\alpha|\leq t^{1+\mu}\}$ and $\widetilde{S}(t)=\{\alpha\in\mathbb{R}|\frac{1}{2}t^{1-\mu}\leq|\alpha|\leq 2t^{1+\mu}\}$. For $f(\cdot,t)\in W^{1,\infty}$ and $g(\cdot,t)\in L^2\cap L^\infty$, we consider the integral$$
        I(f,g)=\frac{1}{\pi i}\int\frac{f(\alpha)-f(\beta)}{(\alpha-\beta)^2}\bar{g}(\beta)d\beta
    $$
and
$$
        J(f,g)=\frac{1}{\pi i}\int\frac{f(\alpha)-f(\beta)}{(\alpha-\beta)^2}g(\beta)d\beta.
    $$
Assume $L_0f(\cdot,t),L_0g(\cdot,t),\Omega_0f(\cdot,t),\Omega_0g(\cdot,t)\in L^2(\widetilde{S}(t))$. For a general function $h$ we set
    
    $$
    W(h)=\norm{L_0h}_{L^2(\widetilde{S}(t))}+\norm{\Omega_0h}_{L^2(\widetilde{S}(t))},
    $$
    then
\begin{equation}
        \norm{I-f_\alpha\bar{g}}_{L^2(S(t))}+\norm{J+f_\alpha g}_{L^2(S(t))}\lesssim\frac{1}{t^{1-2\mu}}(\norm{f}_{L^\infty}+\norm{g}_{L^\infty})(\norm{g}_{L^2}+W(f)+W(g)).
    \end{equation}
\end{lemma}
\begin{proof}
    We only consider $I$. Define $F=e^{\frac{it^2}{4\alpha}}f$ and $G=e^{\frac{it^2}{4\alpha}}g$. We split the integral into two parts with respect to t:
$$
    I(f,g)=\left(\frac{1}{\pi i}\int_{|\alpha-\beta|>t^\gamma}+\frac{1}{\pi i}\int_{|\alpha-\beta|\leq t^\gamma}\right)\frac{f(\alpha)-f(\beta)}{(\alpha-\beta)^2}\bar{g}(\beta)d\beta=I_1+I_2
    $$
where $\gamma=1-2\mu$. According to Young's inequality, the error term $I_1$ satisfies the following:
\begin{equation}\label{I_1}
    \norm{I_1}_{L^2}\lesssim\norm{\frac{1}{\alpha^2}}_{L^1(|\alpha|>t^\gamma)}\norm{f}_{L^{\infty}}\norm{g}_{L^2}\lesssim\frac{1}{t^\gamma}\norm{f}_{L^{\infty}}\norm{g}_{L^2}.
    \end{equation}
For $I_2$, direct calculation shows that
$$
    \begin{aligned}
    I_2=&\ \frac{1}{\pi i}\int_{|\alpha-\beta|\leq t^\gamma}\frac{e^{-\frac{it^2}{4\alpha}}F(\alpha)-e^{-\frac{it^2}{4\beta}}F(\beta)}{(\alpha-\beta)^2}e^{\frac{it^2}{4\beta}}\bar{G}(\beta)d\beta\\
    =&\ \frac{1}{\pi i}F\bar{G}\int\frac{e^{-\frac{it^2}{4\alpha}}-e^{-\frac{it^2}{4\beta}}}{(\alpha-\beta)^2}e^{\frac{it^2}{4\beta}}d\beta-\frac{1}{\pi i}F\bar{G}\int_{|\alpha-\beta|>t^\gamma}\frac{e^{-\frac{it^2}{4\alpha}}-e^{-\frac{it^2}{4\beta}}}{(\alpha-\beta)^2}e^{\frac{it^2}{4\beta}}d\beta\\
    &+\frac{1}{\pi i}\int_{|\alpha-\beta|\leq t^\gamma}\frac{e^{-\frac{it^2}{4\alpha}}F(\alpha)-e^{-\frac{it^2}{4\beta}}F(\beta)}{(\alpha-\beta)^2}e^{\frac{it^2}{4\beta}}[\bar{G}(\beta)-\bar{G}(\alpha)]d\beta\\
    &+\frac{1}{\pi i}\bar{G}(\alpha)\int_{|\alpha-\beta|\leq t^\gamma}\frac{F(\alpha)-F(\beta)}{(\alpha-\beta)^2}d\beta.\\
    =&\ I_m+I_{e_1}+I_{e_2}+I_{e_3}.
    \end{aligned}
    $$
Setting $\eta=\frac{t^2}{4\beta}$, $\xi=\frac{t^2}{4\alpha}$ and applying the residue theorem to the function $\phi(z)=\frac{e^{iz}-1}{z^2}$, we have
$$
    \int\frac{e^{-\frac{it^2}{4\alpha}}-e^{-\frac{it^2}{4\beta}}}{(\alpha-\beta)^2}e^{\frac{it^2}{4\beta}}d\beta=\frac{4\xi^2}{t^2}\int\frac{e^{-i(\xi-\eta)}-1}{(\xi-\eta)^2}d\eta=-\frac{4\xi^2}{t^2}\pi=-\frac{t^2}{4\alpha^2}\pi.
    $$
Meanwhile, 
$$
    \norm{\int_{|\alpha-\beta|>t^\gamma}\frac{e^{-\frac{it^2}{4\alpha}}-e^{-\frac{it^2}{4\beta}}}{(\alpha-\beta)^2}e^{\frac{it^2}{4\beta}}d\beta}_{L^\infty}\lesssim\frac{1}{t^\gamma},
    $$
which implies that
$$
    \norm{I_{e_1}}_{L^2}\lesssim\frac{1}{t^\gamma}\norm{f}_{L^\infty}\norm{g}_{L^2}.
    $$
Since $t^{1-\mu}\leq|\alpha|\leq t^{1+\mu}$, $|\beta-\alpha|\leq t^\gamma$ implies that $\beta\in \widetilde{S}$. Note that $\partial_\alpha G=e^\frac{it^2}{4\alpha}\left(\frac{1}{\alpha}L_0g-\frac{t}{2\alpha^2}\Omega_0g\right)$, we refer to T1 Theorem and deduce that
$$
    \begin{aligned}
    \norm{I_{e_2}}_{L^2(S(t))}
    \lesssim&\norm{f}_{L^\infty}\norm{\partial_\alpha G}_{L^2(\widetilde{S}(t))}\\
    \lesssim&\ \frac{1}{t^{1-2\mu}}\norm{f}_{L^\infty}(\norm{L_0g}_{L^2(\widetilde{S}(t))}+\norm{\Omega_0g}_{L^2(\widetilde{S}(t))}).
    \end{aligned}
    $$
For the same reason,
$$
    \norm{I_{e_3}}_{L^2(S(t))}\lesssim\frac{1}{t^{1-2\mu}}\norm{g}_{L^\infty}(\norm{L_0f}_{L^2(\widetilde{S}(t))}+\norm{\Omega_0f}_{L^2(\widetilde{S}(t))}).
    $$
We recall (\ref{transfer_partial_alpha}) and obtain that
$$
    I_m=\frac{it^2}{4\alpha^2}f\bar{g}=\left(\partial_\alpha f-\frac{1}{\alpha}L_0f+\frac{t}{2\alpha^2}\Omega_0f\right)\bar{g}.
    $$
Hence
$$
    \norm{I_m-f_\alpha \bar{g}}_{L^2(S(t))}\lesssim\frac{1}{t^{1-2\mu}}\norm{g}_{L^\infty}(\norm{L_0f}_{L^2(\widetilde{S}(t))}+\norm{\Omega_0f}_{L^2(\widetilde{S}(t))}),
    $$
which completes the proof.
\end{proof}
\begin{lemma}\label{lem_(I-H)fg_main2}
    Using the notation in Lemma \ref{lem_(I-H)fg_main} of $\mu,S(t),\tilde{S}(t),I(f,g),J(f,g)$ and the assumptions for $f,g$, if we assume $L_0D_tf(\cdot,t), L_0g(\cdot,t),\Omega_0D_tf(\cdot,t),\Omega_0g(\cdot,t),\partial_\alpha\Omega_0f(\cdot,t)\in L^2(\tilde{S}(t))$ instead, then for a general function $h$ we set
    $$
    W(h)=\norm{L_0h}_{L^2(\widetilde{S}(t))}+\norm{\Omega_0h}_{L^2(\widetilde{S}(t))},
    $$
    If in addition $D_tf\in L^2\cap L^\infty$, then
    \begin{equation}\label{I_f_tg}
    \begin{aligned}
    &\norm{I+\frac{t}{2\alpha}f_t\bar g}_{L^2(S(t))}+\norm{J-\frac{t}{2\alpha}f_tg}_{L^2(S(t))}\\
    \lesssim&\ \frac{1}{t}(\norm{D_tf}_{L^\infty}+\norm{g}_{L^\infty}+\norm{bf_\alpha}_{L^\infty})\left[\norm{D_tf}_{L^2}+t^{2\mu}\norm{g}_{L^2}+t^\mu (W(D_tf)+W(g))+\norm{\partial_\alpha\Omega_0f}_{L^2}\right]\\
    &+\frac{1}{t^{1-\mu}}\norm{f}_{L^\infty}\norm{g}_{L^2}+t^{\mu}\norm{b}_{H^1}\norm{f_\alpha}_{W^{1,\infty}}\norm{g}_{L^\infty}.
    \end{aligned}
    \end{equation}
\end{lemma}
\begin{proof}
    We retain the notation from Lemma \ref{lem_(I-H)fg_main} and set $\gamma={1-\mu}$. Now we use (\ref{transfer_partial_t}) to rewrite $I_2$:
    $$\begin{aligned}
    I_2=&-\frac{2}{\pi t}\int_{|\alpha-\beta|\leq t^{\gamma}}\frac{\Omega_0f(\alpha,t)-\Omega_0f(\beta,t)}{(\alpha-\beta)^2}\bar g(\beta,t)d\beta\\
    &+\frac{2}{\pi t}\int_{|\alpha-\beta|\leq t^{\gamma}}\frac{\alpha f_t(\alpha,t)-\beta f_t(\beta,t)}{(\alpha-\beta)^2}\bar g(\beta,t)d\beta.
    \end{aligned}$$
    For the first term,
    $$
    \left|\frac{2}{\pi t}\int_{|\alpha-\beta|\leq t^{\gamma}}\frac{\Omega_0f(\alpha,t)-\Omega_0f(\beta,t)}{(\alpha-\beta)^2}\bar g(\beta,t)d\beta\right|\lesssim\frac{1}{t}\norm{\partial_\alpha\Omega_0f}_{L^2}\norm{g}_{L^\infty}.
    $$
    Now set $\hat{F}=e^{\frac{it^2}{4\alpha}}\alpha f_t$ and $\hat{G}=e^{\frac{it^2}{4\alpha}}\alpha g$. Then
    $$\begin{aligned}
    \int_{|\alpha-\beta|\leq t^{\gamma}}\frac{\alpha f_t(\alpha,t)-\beta f_t(\beta,t)}{(\alpha-\beta)^2}\bar g(\beta,t)d\beta=&\int_{|\alpha-\beta|\leq t^{\gamma}}\frac{e^{-\frac{it^2}{4\alpha}}\hat{F}(\alpha,t)-e^{-\frac{it^2}{4\beta}}\hat{F}(\beta,t)}{(\alpha-\beta)^2}e^{\frac{it^2}{4\beta}}\bar G(\beta,t)d\beta\\
    =&\ \hat{F}\int_{|\alpha-\beta|\leq t^{\gamma}}\frac{e^{-\frac{it^2}{4\alpha}}-e^{-\frac{it^2}{4\beta}}}{(\alpha-\beta)^2}\bar g(\beta,t)d\beta\\
    &+\int_{|\alpha-\beta|\leq t^{\gamma}}\frac{\hat{F}(\alpha,t)-\hat{F}(\beta,t)}{(\alpha-\beta)^2}\bar G(\beta,t)d\beta\\
    =&\ e^{\frac{it^2}{4\alpha}}f_t\int_{|\alpha-\beta|\leq t^{\gamma}}\frac{e^{-\frac{it^2}{4\alpha}}-e^{-\frac{it^2}{4\beta}}}{\alpha-\beta}\bar g(\beta,t)d\beta\\
    &+e^{\frac{it^2}{4\alpha}}f_t\int_{|\alpha-\beta|\leq t^{\gamma}}\frac{e^{-\frac{it^2}{4\alpha}}-e^{-\frac{it^2}{4\beta}}}{(\alpha-\beta)^2}e^{\frac{it^2}{4\beta}}(\bar {\hat{G}}(\beta,t)-\bar {\hat{G}}(\alpha,t))d\beta\\
    &+\alpha f_t\bar g\int_{|\alpha-\beta|\leq t^{\gamma}}\frac{e^{-\frac{it^2}{4\alpha}}-e^{-\frac{it^2}{4\beta}}}{(\alpha-\beta)^2}e^{\frac{it^2}{4\beta}}d\beta\\
    &+\int_{|\alpha-\beta|\leq t^{\gamma}}\frac{\hat{F}(\alpha,t)-\hat{F}(\beta,t)}{(\alpha-\beta)^2}\bar G(\beta,t)d\beta.\\
    \end{aligned}
    $$
    A brief calculation shows that
    $$\begin{aligned}
    \partial_\alpha\hat{F}=&\ \partial_\alpha \left(e^{\frac{it^2}{4\alpha}}\alpha D_tf\right)-\partial_\alpha\left(e^{\frac{it^2}{4\alpha}}\alpha bf_\alpha\right)\\
    =&\  e^{\frac{it^2}{4\alpha}}f_t+e^{\frac{it^2}{4\alpha}}\left(-\frac{t}{2\alpha}\Omega_0D_tf+L_0D_tf\right)-\alpha\partial_\alpha\left(e^{\frac{it^2}{4\alpha}}bf_\alpha\right)
    \end{aligned}
    $$
    which indicates that
    $$
    \norm{\partial_\alpha\hat{F}}_{L^2(\tilde{S}(t))}\lesssim\norm{D_tf}_{L^2}+t^{\mu}(\norm{\Omega_0D_tf}_{L^2(\tilde{S}(t))}+\norm{L_0D_tf}_{L^2(\tilde{S}(t))})+t^{1+\mu}(\norm{b}_{H^1}\norm{f_\alpha}_{W^{1,\infty}}).
    $$
    Similarly,
    $$
    \norm{\partial_\alpha\hat{G}}_{L^2(\tilde{S}(t))}\lesssim\norm{g}_{L^2}+t^{\mu}(\norm{\Omega_0g}_{L^2(\tilde{S}(t))}+\norm{L_0g}_{L^2(\tilde{S}(t))}).
    $$
    Imitating the arguments in Lemma \ref{lem_(I-H)fg_main}, we derive that
    $$\begin{aligned}
    &\norm{I_2+\frac{t}{2\alpha}f_t\bar g}_{L^2(S(t))}\\
    \lesssim&\ \frac{1}{t}(\norm{D_tf}_{L^\infty}+\norm{g}_{L^\infty})\left[\norm{D_tf}_{L^2}+t^{2\mu}\norm{g}_{L^2}+t^\mu (W(D_tf)+W(g))+\norm{\partial_\alpha\Omega_0f}_{L^2}\right]\\
    &+\frac{1}{t}\norm{b}_{L^\infty}\norm{f_\alpha}_{L^\infty}[t^{2\mu}\norm{g}_{L^2}+t^\mu (W(D_tf)+W(g))]+t^{\mu}\norm{b}_{H^1}\norm{f_\alpha}_{W^{1,\infty}}\norm{g}_{L^\infty}.
    \end{aligned}
    $$
    Combining (\ref{I_1}), we prove (\ref{I_f_tg}).
\end{proof}
\begin{rem}\label{rem_(I-H)fg_main}
    We check the proof and reveal an essential fact that Lemma \ref{lem_(I-H)fg_main} and Lemma \ref{lem_(I-H)fg_main2} remain valid if $I$ is replaced by $I_2$ and $f,g\in L^2_{loc}\cap L^{\infty}_{loc}$ instead of satisfying the condition given in these lemma. The same idea holds for the following sequence of theorems.
\end{rem}
\begin{lemma}\label{integral_three}
    We adopt the notation in Lemma \ref{lem_(I-H)fg_main}. For $f_1\in W^{1,\infty}$ and $f_2,f_3\in L^2\cap L^\infty$, we consider the integral$$
    I(f_1,f_2,f_3)=\frac{1}{\pi i}\int\frac{f_1(\alpha)-f_1(\beta)}{(\alpha-\beta)^2}f_2(\beta)\bar{f_3}(\beta)d\beta
    $$
and
$$
    J(f_1,f_2,f_3)=\frac{1}{\pi i}\int\frac{f_1(\alpha)-f_1(\beta)}{(\alpha-\beta)^2}f_2(\beta)f_3(\beta)d\beta.
    $$
If $L_0f_i,\Omega_0f_i\in L^2(\widetilde{S}(t))$ for $i=1,2,3$, then
$$
\begin{aligned}
    \norm{I+\partial_\alpha f_1f_2\bar{f_3}}_{L^2(S(t))}&+\norm{J+\partial_\alpha f_1f_2f_3}_{L^2(S(t))}\\
    \lesssim&\ \frac{1}{t^{1-2\mu}}\left(\sum_{i=1}^3\norm{f_i}_{L^\infty}\right)^2\sum_{i=1}^3(\norm{f_i}_{L^2}+\norm{L_0f_i}_{L^2(\widetilde{S}(t))}+\norm{\Omega_0f_i}_{L^2(\widetilde{S}(t))}).
    \end{aligned}
    $$
    If $D_tf_1\in L^2\cap L^\infty$ and $L_0f_i,\Omega_0f_i\in L^2(\widetilde{S}(t))$ for $i=2,3$, then
    \begin{equation}\label{I_f_1_tf_2f_3}
    \begin{aligned}
    &\norm{I-\frac{t}{2\alpha}\partial_tf_1f_2\bar f_3}_{L^2(S(t))}+\norm{J-\frac{t}{2\alpha}\partial_tf_1f_2 f_3}_{L^2(S(t))}\\
    \lesssim&\ \frac{1}{t}(\norm{D_tf_1}_{L^\infty}+\sum_{j=2}^3\norm{f_j}_{L^\infty})^2\left[\norm{D_tf_1}_{L^2}+t^{2\mu}\sum_{j=2}^3\norm{f_j}_{L^2}+t^\mu \left(W(D_tf)+\sum_{j=2}^3W(f_j)\right)\right]\\
    &+\frac{1}{t}\norm{b}_{L^\infty}\norm{\partial_\alpha f_1}_{L^\infty}\sum_{j=2}^3\norm{f_j}_{L^\infty}\left(t^{2\mu}\sum_{j=2}^3\norm{f_j}_{L^2}+t^\mu\sum_{j=2}^3W(f_j)\right)+\frac{1}{t^{1-\mu}}\norm{f_1}_{L^\infty}\norm{f_2f_3}_{L^2}\\
    &+t^{\mu}\norm{b}_{H^1}\norm{\partial_\alpha f_1}_{W^{1,\infty}}\norm{f_2f_3}_{L^\infty}+\frac{1}{t}\norm{\partial_\alpha\Omega_0f_1}_{L^2}\norm{f_2f_3}_{L^{\infty}}.
    \end{aligned}
    \end{equation}
\end{lemma}
\begin{proof}
    Define $F_1=e^{\frac{it^2}{4\alpha}}f_1$, $F_2=e^{\frac{it^2}{4\alpha}}f_2$ and $F_3=e^{\frac{it^2}{4\alpha}}f_3$. As in Lemma \ref{lem_(I-H)fg_main}, we split the integral
$$
    I(f_1,f_2,f_3)=\left(\frac{1}{\pi i}\int_{|\alpha-\beta|>t^\gamma}+\frac{1}{\pi i}\int_{|\alpha-\beta|\leq t^\gamma}\right)\frac{f_1(\alpha)-f_1(\beta)}{(\alpha-\beta)^2}f_2(\beta)\bar{f_3}(\beta)d\beta=I_1+I_2,
    $$
where $\gamma=1-2\mu$. $I_2$ satisfies
    
    $$
    \begin{aligned}
    I_2=&\ \frac{1}{\pi i}\int_{|\alpha-\beta|\leq t^\gamma}\frac{e^{-\frac{it^2}{4\alpha}}F_1(\alpha)-e^{-\frac{it^2}{4\beta}}F_1(\beta)}{(\alpha-\beta)^2}F_2(\beta)\bar{F_3}(\beta)d\beta\\
    =&\ \frac{1}{\pi i}F_1F_2\bar{F_3}\int\frac{e^{-\frac{it^2}{4\alpha}}-e^{-\frac{it^2}{4\beta}}}{(\alpha-\beta)^2}d\beta-\frac{1}{\pi i}F_1F_2\bar{F_3}\int_{|\alpha-\beta|>t^\gamma}\frac{e^{-\frac{it^2}{4\alpha}}-e^{-\frac{it^2}{4\beta}}}{(\alpha-\beta)^2}d\beta\\
    &+\frac{1}{\pi i}\int_{|\alpha-\beta|\leq t^\gamma}\frac{e^{-\frac{it^2}{4\alpha}}F_1(\alpha)-e^{-\frac{it^2}{4\beta}}F_1(\beta)}{(\alpha-\beta)^2}F_2(\beta)[\bar{F_3}(\beta)-\bar{F_3}(\alpha)]d\beta\\
    &+\frac{1}{\pi i}\bar{F_3}(\alpha)\int_{|\alpha-\beta|\leq t^\gamma}\frac{e^{-\frac{it^2}{4\alpha}}F_1(\alpha)-e^{-\frac{it^2}{4\beta}}F_1(\beta)}{(\alpha-\beta)^2}[F_2(\beta)-F_2(\alpha)]d\beta\\
    &+\frac{1}{\pi i}F_2(\alpha)\bar{F_3}(\alpha)\int_{|\alpha-\beta|\leq t^\gamma}\frac{F_1(\alpha)-F_1(\beta)}{(\alpha-\beta)^2}e^{-\frac{it^2}{4\beta}}d\beta.\\
    =&\ I_m+I_{e_1}+I_{e_2}+I_{e_3}+I_{e_4}.
    \end{aligned}
    $$
Analogous to the proof of Lemma \ref{lem_(I-H)fg_main}, the following holds:
$$
    \norm{I_m+\partial_\alpha f_1 f_2\bar{f_3}}_{L^2(S(t))}\lesssim\frac{1}{t^{1-2\mu}}\norm{f_2}_{L^\infty}\norm{f_3}_{L^\infty}(\norm{L_0f_1}_{L^2(\widetilde{S}(t))}+\norm{\Omega_0f_1}_{L^2(\widetilde{S}(t))})
    $$
and
$$
    \sum_{i=1}^4\norm{I_{e_i}}_{L^2(S(t))}\lesssim\frac{1}{t^{1-2\mu}}\left(\sum_{i=1}^3\norm{f_i}_{L^\infty}\right)^2\sum_{i=1}^3(\norm{f_i}_{L^2}+\norm{L_0f_i}_{L^2(\widetilde{S}(t))}+\norm{\Omega_0f_i}_{L^2(\widetilde{S}(t))}).
    $$
$J$ can be handled by similar arguments. (\ref{I_f_1_tf_2f_3}) can be obtained by following the analogous arguments in Lemma \ref{lem_(I-H)fg_main2}.
\end{proof}

\begin{cor}\label{cor_integral_three}
    Suppose $f_1,f_2,f_3$ satisfy the condition in Lemma \ref{integral_three}. Consider the integral
$$
    I_1(f_1,f_2,f_3)=\frac{1}{\pi i}\int\frac{(f_1(\alpha)-f_1(\beta))(f_2(\alpha)-f_2(\beta))}{(\alpha-\beta)^2}\bar{f_3}(\beta)d\beta,
    $$
$$
    I_2(f_1,f_2,f_3)=\frac{1}{\pi i}\int\frac{(f_1(\alpha)-f_1(\beta))(\bar{f_2}(\alpha)-\bar{f_2}(\beta))}{(\alpha-\beta)^2}f_3(\beta)d\beta
    $$
    and
$$
    I_3(f_1,f_2,f_3)=\frac{1}{\pi i}\int\frac{(f_1(\alpha)-f_1(\beta))(f_2(\alpha)-f_2(\beta))}{(\alpha-\beta)^2}f_3(\beta)d\beta.
    $$
Then we have
\begin{equation}
\begin{aligned}
        \norm{I_1-2\partial_\alpha f_1f_2\bar{f_3}}_{L^2(S(t))}+&\norm{I_2}_{L^2(S(t))}+\norm{I_3}_{L^2(S(t))}\\
        \lesssim&\ \frac{1}{t^{1-2\mu}}\left(\sum_{i=1}^3\norm{f_i}_{L^\infty}\right)^2\sum_{i=1}^3(\norm{f_i}_{L^2}+\norm{L_0f_i}_{L^2(\widetilde{S}(t))}+\norm{\Omega_0f_i}_{L^2(\widetilde{S}(t))}).
    \end{aligned}
    \end{equation}
    Moreover, if $D_tf_1\in L^2\cap L^\infty$, then
    $$
    \begin{aligned}
        &\norm{I_1+\frac{t}{\alpha}\partial_t f_1f_2\bar{f_3}}_{L^2(S(t))}+\norm{I_2}_{L^2(S(t))}+\norm{I_3}_{L^2(S(t))}\\
        \lesssim&\ \frac{1}{t}(\norm{D_tf_1}_{L^\infty}+\sum_{j=2}^3\norm{f_j}_{L^\infty})^2\left[\norm{D_tf_1}_{L^2}+t^{2\mu}\sum_{j=2}^3\norm{f_j}_{L^2}+t^\mu \left(W(D_tf)+\sum_{j=2}^3W(f_j)\right)\right]\\
    &+\frac{1}{t}\norm{b}_{L^\infty}\norm{\partial_\alpha f_1}_{L^\infty}\sum_{j=2}^3\norm{f_j}_{L^\infty}\left(t^{2\mu}\sum_{j=2}^3\norm{f_j}_{L^2}+t^\mu\sum_{j=2}^3W(f_j)\right)+\frac{1}{t^{1-\mu}}\norm{f_1}_{L^\infty}\norm{f_2f_3}_{L^2}\\
    &+t^{\mu}\norm{b}_{H^1}\norm{\partial_\alpha f_1}_{W^{1,\infty}}\norm{f_2f_3}_{L^\infty}+\frac{1}{t}\norm{\partial_\alpha\Omega_0f_1}_{L^2}\norm{f_2f_3}_{L^{\infty}}.
    \end{aligned}
    $$
\end{cor}
\begin{proof}
    The integral can be written as
$$
        I_1(f_1,f_2,f_3)=\frac{1}{\pi i}f_2(\alpha)\int\frac{f_1(\alpha)-f_1(\beta)}{(\alpha-\beta)^2}\bar{f_3}(\beta)d\beta-\frac{1}{\pi i}\int\frac{f_1(\alpha)-f_1(\beta)}{(\alpha-\beta)^2}f_2(\beta)\bar{f_3}(\beta)d\beta,
    $$
$$
        I_2(f_1,f_2,f_3)=\frac{1}{\pi i}\bar{f_2}(\alpha)\int\frac{f_1(\alpha)-f_1(\beta)}{(\alpha-\beta)^2}f_3(\beta)d\beta-\frac{1}{\pi i}\int\frac{f_1(\alpha)-f_1(\beta)}{(\alpha-\beta)^2}\bar{f_2}(\beta)f_3(\beta)d\beta
    $$
and
$$
        I_3(f_1,f_2,f_3)=\frac{1}{\pi i}f_2(\alpha)\int\frac{f_1(\alpha)-f_1(\beta)}{(\alpha-\beta)^2}f_3(\beta)d\beta-\frac{1}{\pi i}\int\frac{f_1(\alpha)-f_1(\beta)}{(\alpha-\beta)^2}f_2(\beta)f_3(\beta)d\beta.
    $$
Then we appeal to Lemma \ref{lem_(I-H)fg_main} and Lemma \ref{integral_three}.
\end{proof}
\begin{rem}
    The techniques employed in the preceding theorems can be extended to integrals of an analogous form. For instance, we consider the integral:
$$
    I(f_1,f_2,f_3)=\frac{1}{\pi i}\int\frac{(f_1(\alpha)-f_1(\beta))(f_2(\alpha)-f_2(\beta))}{(\alpha-\beta)^3}f_3(\beta)d\beta.
    $$
Similar calculation shows that
$$
    \norm{I+\partial_\alpha f_1\partial_\alpha f_2f_3}_{L^2}\lesssim\frac{1}{t^{1-2\mu}}\left(\sum_{i=1}^3\norm{f_i}_{L^\infty}\right)^2\sum_{i=1}^3(\norm{f_i}_{L^2}+\norm{L_0f_i}_{L^2(\widetilde{S}(t))}+\norm{\Omega_0f_i}_{L^2(\widetilde{S}(t))}).
    $$
\end{rem}
\begin{rem}\label{rem_S(t)}
    $S(t)$ can be replaced by $\hat{S}(t)=\{\alpha|C^{-1}t^{1-\mu}\leq|\alpha|\leq Ct^{1+\mu}\}$ for some constant $C>0$ in the above series of theorems, with $\tilde{S}(t)$ substituted by $\{\alpha|(2C)^{-1}t^{1-\mu}\leq|\alpha|\leq 2Ct^{1+\mu}\}$. The result won't be influenced if we split the integral carefully.
\end{rem}
\subsection{Bounds for $b$ and $A-1$}\label{subs_b}
Based on the bootstrap assumptions (B1)-(B6), we estimate some quantities with respect to $b$, $A-1$, $\frac{a_t}{a}\circ \kappa^{-1}$, etc. The expression can be found in Lemma \ref{lem_b_exp}.
\begin{lemma}\label{lem_b}
    We have
    \begin{equation}\label{b_L^2}
        \norm{b}_{L^2}\lesssim\frac{\epsilon^2}{t^{\frac{1}{2}}}+\frac{\lambda}{H^{\frac{3}{2}}},
    \end{equation}
    \begin{equation}\label{b_L^infty}
        \norm{b}_{L^\infty}\lesssim\frac{\epsilon^2}{t^{\frac{3}{4}}}\ln (2+t)+\frac{\lambda}{H^2},
    \end{equation}
    \begin{equation}
        \norm{\partial_\alpha b}_{H^{s-1}}\lesssim\frac{\epsilon^2}{t^{\frac{2}{3}-\delta_0}}\ln (2+t)+\frac{\lambda}{H^{\frac{3}{2}}},
    \end{equation}
    \begin{equation}\label{partial_alpha_b_infty}
        \norm{\partial_\alpha b}_{W^{s-3,\infty}}\lesssim\frac{\epsilon^2}{t^{\frac{3}{2}-\delta_0}}(\ln (2+t))^2+\frac{\lambda}{H^3},
    \end{equation}
    \begin{equation}
        \norm{\partial_\alpha b}_{\dot W^{s-2,\infty}}\lesssim\frac{\epsilon^2}{t^{\frac{7}{6}-\delta_0}}(\ln (2+t))^2+\frac{\lambda}{H^3},
    \end{equation}
\begin{equation}\label{D_t_b_L2}
        \norm{D_tb}_{H^{s}}\lesssim\frac{\epsilon^2}{t^{\frac{2}{3}-\delta_0}}\ln (2+t)+\frac{\lambda(\epsilon+\lambda)}{H^{\frac{5}{2}}},
    \end{equation}
\begin{equation}\label{partial_t_b_infty}
        \norm{D_t b}_{W^{s-2,\infty}}\lesssim\frac{\epsilon^2}{t^{\frac{5}{4}-\delta_0}}(\ln (2+t))^2+\frac{\lambda(\epsilon+\lambda)}{H^3}
    \end{equation}
and
\begin{equation}\label{D_t_b_W^s-1}
        \norm{D_t b}_{\dot W^{s-1,\infty}}\lesssim\frac{\epsilon^2}{t^{\frac{7}{6}-\delta_0}}(\ln (2+t))^2+\frac{\lambda(\epsilon+\lambda)}{H^3}.
    \end{equation}
\end{lemma}
\begin{proof}
    (\ref{b_L^2}) and (\ref{b_L^infty}) are proved in Lemma \ref{lem_Omega}.
In order to estimate $\norm{\partial_\alpha b}_{H^{s-1}}$, we derive the expression of $\partial_\alpha^j (I-\mathfrak{H})b$ for $1\leq j\leq s$:
\begin{equation}\label{mul_b_alpha}
    \begin{aligned}
        \partial_\alpha^j(I-\mathfrak{H})b=&\ \partial_\alpha^j\left\{-(I-\mathfrak{H})D_t\zeta\frac{\bar{\zeta}_\alpha-1}{\zeta_\alpha}-\frac{i\lambda}{\pi}\left[\frac{z_1-z_2}{(\zeta-z_1)(\zeta-z_2)}\right]\right\}\\
        =&-(I-\mathfrak{H})\partial_\alpha^j\left(D_t\zeta\frac{\bar{\zeta}_\alpha-1}{\zeta_\alpha}\right)+\sum_{i=0}^{j-1}\partial_\alpha^i[\partial_\alpha,\mathfrak{H}]\partial_\alpha^{j-i-1}\left(D_t\zeta\frac{\bar{\zeta}_\alpha-1}{\zeta_\alpha}\right)\\
        &-\frac{i\lambda}{\pi}\partial_\alpha^j\left[\frac{z_1-z_2}{(\zeta-z_1)(\zeta-z_2)}\right]\\
        =&-(I-\mathfrak{H})\partial_\alpha^j\left(D_t\zeta\frac{\bar{\zeta}_\alpha-1}{\zeta_\alpha}\right)+\sum_{i=0}^{j-1}\partial_\alpha^i[\zeta_\alpha,\mathfrak{H}]\frac{\partial_\alpha}{\zeta_\alpha}\partial_\alpha^{j-i-1}\left(D_t\zeta\frac{\bar{\zeta}_\alpha-1}{\zeta_\alpha}\right)\\
        &-\frac{i\lambda}{\pi}\partial_\alpha^j\left[\frac{z_1-z_2}{(\zeta-z_1)(\zeta-z_2)}\right].
    \end{aligned}
    \end{equation}
    Since
    $$
    \partial_\alpha^j(D_t\zeta\frac{\bar\zeta_\alpha-1}{\zeta_\alpha})=\partial_\alpha^j\left[D_t\zeta(\bar\zeta_\alpha-1)\left(\frac{1}{\zeta_\alpha}-1\right)\right]+\sum_{l=1}^{j-1}\partial_\alpha[\partial_\alpha^lD_t\zeta\partial_\alpha^{j-l-1}(\bar\zeta_\alpha-1)]+\partial_\alpha[D_t\zeta\partial_\alpha^{j-1}(\bar\zeta_\alpha-1)],
    $$
    from Lemma \ref{lem_transfer_d} and Lemma \ref{lem_Omega} we have
    $$
    \begin{aligned}
    &\norm{\sum_{l=1}^{j-1}\partial_\alpha[\partial_\alpha^lD_t\zeta\partial_\alpha^{j-l-1}(\bar\zeta_\alpha-1)]}_{L^2}\\
    \lesssim&\frac{1}{t}\left(\norm{L_0D_t\zeta}_{H^{s-1}}\norm{D_t\zeta}_{W^{s-1}}+\norm{L_0(\zeta_\alpha-1)}_{H^{s-1}}\norm{D_t^2\zeta}_{W^{s-1}}\right)\\
    &+\frac{1}{t}\left(\norm{\Omega_0\partial_\alpha D_t\zeta}_{H^{s-2}}\norm{\zeta_\alpha-1}_{W^{s-1}}+\norm{D_t\zeta}_{H^s}\norm{\Omega_0(\zeta_\alpha-1)}_{W^{s-2}}\right)\\
    \lesssim&\frac{\epsilon^2}{t^{\frac{2}{3}}}\ln (2+t).
    \end{aligned}
    $$
Thus by Lemma \ref{lem_real_proj},
$$\begin{aligned}
        \norm{\partial_\alpha b}_{H^{s-1}}\lesssim&\sum_{i=0}^{s-1}\norm{(I-\mathfrak{H})\partial_\alpha^{i+1} b}_{L^2}\\
        \lesssim&\ \frac{\epsilon^2}{t^{\frac{2}{3}-\delta_0}}\ln (2+t)+\norm{q}_{H^s}+\sum_{j=1}^s\norm{(I-\mathfrak{H})(\partial_\alpha(D_t\zeta\partial_\alpha^{j-1}(\bar\zeta_\alpha-1)))}_{L^2}\\
        \lesssim&\ \frac{\epsilon^2}{t^{\frac{2}{3}-\delta_0}}\ln (2+t)+\frac{\lambda}{H^{\frac{3}{2}}}+\sum_{j=1}^s\norm{(I-\mathfrak{H})(\partial_\alpha(D_t\zeta\partial_\alpha^{j-1}(\bar\zeta_\alpha-1)))}_{L^2}.
    \end{aligned}$$
    We claim that
    $$
    \sum_{j=1}^s\norm{(I-\mathfrak{H})(\partial_\alpha(D_t\zeta\partial_\alpha^{j-1}(\bar\zeta_\alpha-1)))}_{L^2}\lesssim\frac{\epsilon^2}{t^{\frac{2}{3}-\delta_0}}\ln (2+t).
    $$
    the proof of the inequality is contained in that of $\norm{\partial_\alpha^sD_tb}_{L^2}$ which we will investigate below. Using Lemma \ref{lem_transfer_d}, Lemma \ref{lem_Hf_infty} and Lemma \ref{lem_real_proj}, we show that
$$\norm{\partial_\alpha b}_{W^{s-3,\infty}}\lesssim\frac{\epsilon^2}{t^{\frac{3}{2}-\delta_0}}(\ln (2+t))^2+\norm{q}_{W^{s-2,\infty}}\lesssim\frac{\epsilon^2}{t^{\frac{3}{2}-\delta_0}}(\ln (2+t))^2+\frac{\lambda}{H^3}$$
and
$$
\norm{\partial_\alpha b}_{\dot W^{s-2,\infty}}\lesssim\frac{\epsilon^2}{t^{\frac{7}{6}-\delta_0}}(\ln (2+t))^2+\frac{\lambda}{H^3}.
$$
   In summary, (\ref{partial_alpha_b_infty}) is thus proven. (\ref{D_t_b_L2}) and (\ref{partial_t_b_infty}) can be obtained by following similar steps except for $\norm{\partial_\alpha^sD_t b}_{L^2}$. To handle this problem, we take $(I-\mathfrak{H})\partial_\alpha^sD_t b$ into consideration:
\begin{equation}\label{partial_alpha^sD_tb}
    \begin{aligned}
    (I-\mathfrak{H})\partial_\alpha^sD_t b=&\ [\partial_\alpha^sD_t,\mathfrak{H}]b+\partial_\alpha^sD_t\left\{-(I-\mathfrak{H})\left(D_t\zeta\frac{\bar{\zeta}_\alpha-1}{\zeta_\alpha}\right)-\frac{i\lambda}{\pi}\left[\frac{z_1-z_2}{(\zeta-z_1)(\zeta-z_2)}\right]\right\}\\
    =&\ \partial_\alpha^s[D_t\zeta,\mathfrak{H}]\frac{b_\alpha}{\zeta_\alpha}+[\partial_\alpha^s,\mathfrak{H}]D_tb\\
    &+\partial_\alpha^s[D_t\zeta,\mathfrak{H}]\frac{\partial_\alpha}{\zeta_\alpha}\left(D_t\zeta\frac{\bar{\zeta}_\alpha-1}{\zeta_\alpha}\right)+[\partial_\alpha^s,\mathfrak{H}]D_t\left(D_t\zeta\frac{\bar{\zeta}_\alpha-1}{\zeta_\alpha}\right)\\
    &-(I-\mathfrak{H})\partial_\alpha^sD_t\left(D_t\zeta\frac{\bar{\zeta}_\alpha-1}{\zeta_\alpha}\right)-\frac{i\lambda}{\pi}\partial_\alpha^sD_t\left[\frac{z_1-z_2}{(\zeta-z_1)(\zeta-z_2)}\right].
    \end{aligned}
    \end{equation}
For the estimate of all the commutators we can appeal to Lemma \ref{app_S1}. Lemma \ref{lem_transfer_d} cannot be directly applied to the rest since we can only control $\norm{D_t\zeta}_{H^s}$. To solve the problem, we note that
$$
    \begin{aligned}
    &\ (I-\mathfrak{H})\partial_\alpha^sD_t\left(D_t\zeta\frac{\bar{\zeta}_\alpha-1}{\zeta_\alpha}\right)\\
    =&\ (I-\mathfrak{H})\partial_\alpha^s\left(D_t^2\zeta\frac{\bar{\zeta}_\alpha-1}{\zeta_\alpha}+D_t\zeta D_t\frac{\bar{\zeta}_\alpha-1}{\zeta_\alpha}\right)\\=&\ (I-\mathfrak{H})\partial_\alpha^s\left(D_t^2\zeta\frac{\bar{\zeta}_\alpha-1}{\zeta_\alpha}\right)+\sum_{i=0}^{s-2}C_i(I-\mathfrak{H})\partial_\alpha\left(\partial_\alpha^{s-i-1}D_t\zeta\partial_\alpha^iD_t\frac{\bar{\zeta}_\alpha-1}{\zeta_\alpha}\right)\\
    &+(I-\mathfrak{H})\partial_\alpha\left(D_t\zeta\partial_\alpha^{s-1}D_t\frac{\bar{\zeta}_\alpha-1}{\zeta_\alpha}\right).
    \end{aligned}
    $$
Lemma \ref{lem_transfer_d} implies that\footnote{Since we lack the estimate of $\norm{\Omega_0D^2_t\zeta}_{L^2}$, one way to bypass the difficulty is that we can replace $D_t^2\zeta$ with $\zeta_\alpha-1$ via \eqref{main_equation} in the initial term, where the error term is easily handled. This strategy will be adopted several times if necessary.}
$$
    \norm{\partial_\alpha^s\left(D_t^2\zeta\frac{\bar{\zeta}_\alpha-1}{\zeta_\alpha}\right)}_{L^2}+\sum_{i=0}^{s-2}\norm{\partial_\alpha\left(\partial_\alpha^{s-i-1}D_t\zeta\partial_\alpha^iD_t\frac{\bar{\zeta}_\alpha-1}{\zeta_\alpha}\right)}_{L^2}\lesssim\frac{\epsilon^2}{t^{\frac{2}{3}-\delta_0}}\ln (2+t).
    $$
For the last term, we notice that $(I-\mathfrak{H})(\bar{\zeta}-\alpha)=0$ implies $(I-\mathfrak{H})\frac{(\bar{\zeta}_\alpha-1)}{\zeta_\alpha}=0$ and transform it into an integral:
\begin{equation}\label{D_tb_H^s}
    \begin{aligned}
    (I-\mathfrak{H})\partial_\alpha\left(D_t\zeta\partial_\alpha^{s-1}D_t\frac{\bar{\zeta}_\alpha-1}{\zeta_\alpha}\right)=&\ (I-\mathfrak{H})\left(\partial_\alpha D_t\zeta\partial_\alpha^{s-1}D_t\frac{\bar{\zeta}_\alpha-1}{\zeta_\alpha}\right)\\
    &+(I-\mathfrak{H})\left(D_t\zeta\partial_\alpha^{s}D_t\frac{\bar{\zeta}_\alpha-1}{\zeta_\alpha}\right)\\
    =&\ [\partial_\alpha D_t\zeta,\mathfrak{H}]\partial_\alpha^{s-1}D_t\frac{\bar{\zeta}_\alpha-1}{\zeta_\alpha}+[ D_t\zeta,\mathfrak{H}]\partial_\alpha^{s}D_t\frac{\bar{\zeta}_\alpha-1}{\zeta_\alpha}\\
    &+\partial_\alpha D_t\zeta[\partial_\alpha^{s-1}D_t,\mathfrak{H}]\frac{\bar{\zeta}_\alpha-1}{\zeta_\alpha}+D_t\zeta[\partial_\alpha^{s}D_t,\mathfrak{H}]\frac{\bar{\zeta}_\alpha-1}{\zeta_\alpha}\\
    =&\ \frac{1}{\pi i}\int\frac{\partial_\alpha D_t\zeta(\alpha)-\partial_\beta D_t\zeta(\beta)}{\zeta(\alpha)-\zeta(\beta)}\zeta_\beta\left(\partial_\beta^{s-1}D_t\frac{\bar{\zeta}_\beta-1}{\zeta_\beta}\right)(\beta)d\beta\\
    &+\frac{1}{\pi i}\int\frac{D_t\zeta(\alpha)-D_t\zeta(\beta)}{\zeta(\alpha)-\zeta(\beta)}\zeta_\beta\left(\partial_\beta^{s}D_t\frac{\bar{\zeta}_\beta-1}{\zeta_\beta}\right)(\beta)d\beta\\
    &+\partial_\alpha D_t\zeta[\partial_\alpha^{s-1}D_t,\mathfrak{H}]\frac{\bar{\zeta}_\alpha-1}{\zeta_\alpha}+D_t\zeta[\partial_\alpha^{s}D_t,\mathfrak{H}]\frac{\bar{\zeta}_\alpha-1}{\zeta_\alpha}.
    \end{aligned}
    \end{equation}
The last two are cubic and higher orders and can be handled via Lemma \ref{app_S1}.Denote the sum of two integrals by $K$. We proceed as below by using integration by parts:
$$
    \begin{aligned}
        &\int\frac{D_t\zeta(\alpha)-D_t\zeta(\beta)}{\zeta(\alpha)-\zeta(\beta)}\zeta_\beta\left(\partial_\beta^{s}D_t\frac{\bar{\zeta}_\beta-1}{\zeta_\beta}\right)(\beta)d\beta\\
        =&\int\frac{\partial_\beta D_t\zeta(\beta)}{\zeta(\alpha)-\zeta(\beta)}\zeta_\beta\left(\partial_\beta^{s-1}D_t\frac{\bar{\zeta}_\beta-1}{\zeta_\beta}\right)(\beta)d\beta\\
        &-\int\left(\frac{D_t\zeta(\alpha)-D_t\zeta(\beta)}{(\zeta(\alpha)-\zeta(\beta))^2}\zeta_{\beta\beta}(\beta)\right)\left(\partial_\beta^{s-1}D_t\frac{\bar{\zeta}_\beta-1}{\zeta_\beta}\right)(\beta)d\beta\\
        &-\int(D_t\zeta(\alpha)-D_t\zeta(\beta))\left(\frac{\zeta_\beta}{\zeta(\alpha)-\zeta(\beta)}\right)^2\left(\partial_\beta^{s-1}D_t\frac{\bar{\zeta}_\beta-1}{\zeta_\beta}\right)(\beta)d\beta.
    \end{aligned}
    $$
Therefore,
$$
    \begin{aligned}
    K=&\ \frac{1}{\pi i}\partial_\alpha D_t\zeta(\alpha)\int\frac{\zeta_\beta}{\zeta(\alpha)-\zeta(\beta)}\left(\partial_\beta^{s-1}D_t\frac{\bar{\zeta}_\beta-1}{\zeta_\beta}\right)(\beta)d\beta\\
    &-\frac{1}{\pi i}\int\left(\frac{D_t\zeta(\alpha)-D_t\zeta(\beta)}{(\zeta(\alpha)-\zeta(\beta))^2}\zeta_{\beta\beta}(\beta)\right)\left(\partial_\beta^{s-1}D_t\frac{\bar{\zeta}_\beta-1}{\zeta_\beta}\right)(\beta)d\beta\\
        &-\frac{1}{\pi i}\int(D_t\zeta(\alpha)-D_t\zeta(\beta))\left(\frac{\zeta_\beta}{\zeta(\alpha)-\zeta(\beta)}\right)^2\left(\partial_\beta^{s-1}D_t\frac{\bar{\zeta}_\beta-1}{\zeta_\beta}\right)(\beta)d\beta\\
    =&\ \partial_\alpha D_t\zeta[\mathfrak{H},\partial_\alpha^{s-1}D_t]\frac{\bar{\zeta}_\alpha-1}{\zeta_\alpha}+\partial_\alpha D_t\zeta\partial_\alpha^{s-1}D_t\frac{\bar{\zeta}_\alpha-1}{\zeta_\alpha}\\
    &-\frac{1}{\pi i}\int\frac{D_t\zeta(\alpha)-D_t\zeta(\beta)}{(\alpha-\beta)^2}\left(\partial_\beta^{s-1}D_t\frac{\bar{\zeta}_\beta-1}{\zeta_\beta}\right)(\beta)d\beta\\
    &-\frac{1}{\pi i}\int\left(\frac{\zeta_\beta^2(\beta)}{(\zeta(\alpha)-\zeta(\beta))^2}-\frac{1}{(\alpha-\beta)^2}\right)(D_t\zeta(\alpha)-D_t\zeta(\beta))\left(\partial_\beta^{s-1}D_t\frac{\bar{\zeta}_\beta-1}{\zeta_\beta}\right)(\beta)d\beta.\\
    &-\frac{1}{\pi i}\int\left(\frac{D_t\zeta(\alpha)-D_t\zeta(\beta)}{(\zeta(\alpha)-\zeta(\beta))^2}\zeta_{\beta\beta}(\beta)\right)\left(\partial_\beta^{s-1}D_t\frac{\bar{\zeta}_\beta-1}{\zeta_\beta}\right)(\beta)d\beta.
    \end{aligned}
    $$
All the terms on the right side are cubic and higher orders except for 
the second and the third. We denote the sum of all the cubic and higher order terms as $R$. Rewrite $K$ as
$$
\begin{aligned}
    K=&\ \partial_\alpha D_t\zeta\partial_\alpha^{s-1}D_t\frac{\bar{\zeta}_\alpha-1}{\zeta_\alpha}-\frac{1}{\pi i}\int\frac{D_t\zeta(\alpha)-D_t\zeta(\beta)}{(\alpha-\beta)^2}\left(\partial_\beta^{s-1}D_t\frac{\bar{\zeta}_\beta-1}{\zeta_\beta}\right)(\beta)d\beta+R\\
    =&\ \partial_\alpha D_t\zeta\partial_\alpha^{s-1}D_t\frac{\bar{\zeta}_\alpha-1}{\zeta_\alpha}-\frac{1}{\pi i}\int\frac{D_t\zeta(\alpha)-D_t\zeta(\beta)}{(\alpha-\beta)^2}\partial_\beta^{s}D_t\bar{\zeta} d\beta+\widetilde{R},
\end{aligned}
$$
where
$$
\begin{aligned}
    \widetilde{R}=&\ R-\frac{1}{\pi i}\int\frac{D_t\zeta(\alpha)-D_t\zeta(\beta)}{(\alpha-\beta)^2}\partial_\beta^{s-1}\left[D_t\left(\frac{1}{\zeta_\beta}\right)(\bar{\zeta}_\beta-1)\right]d\beta\\
    &-\sum_{i=0}^{s-2}\frac{1}{\pi i}\int\frac{D_t\zeta(\alpha)-D_t\zeta(\beta)}{(\alpha-\beta)^2}\partial_\beta^{s-i-1}\left(\frac{1}{\zeta_\beta}\right)\partial_\beta^{i}D_t\bar{\zeta}_\beta d\beta\\
    &-\frac{1}{\pi i}\int\frac{D_t\zeta(\alpha)-D_t\zeta(\beta)}{(\alpha-\beta)^2}\partial_\beta^{s-1}D_t\bar{\zeta}_\beta(\frac{1}{\zeta_\beta}-1) d\beta
\end{aligned}
$$
$\widetilde{R}$ also consists of cubic or higher order terms, which implies that
$$
\norm{\widetilde{R}}_{L^2}\lesssim\frac{\epsilon^3}{t}.
$$
If $|\alpha|\leq t^{\frac{5}{6}}$, we recall (\ref{transfer_partial_alpha}):
$$
\begin{aligned}
&\norm{\partial_\alpha D_t\zeta\partial_\alpha^{s-1}D_t\frac{\bar{\zeta}_\alpha-1}{\zeta_\alpha}}_{L^2(|\alpha|\leq t^\frac{5}{6})}\\
\lesssim&\left(\frac{1}{t^\frac{1}{3}}\norm{\partial_\alpha^2D_t\zeta}_{L^\infty}+\frac{1}{t}\norm{\Omega_0\partial_\alpha D_t\zeta}_{L^\infty}+\frac{1}{t^{\frac{7}{6}}}\norm{L_0\partial_\alpha D_t\zeta}_{H^1}\right)\norm{D_t\zeta}_{H^s}\\
\lesssim&\ \frac{\epsilon^2}{t^{\frac{5}{6}}}.
\end{aligned}
$$
Meanwhile,
$$
\begin{aligned}
\norm{\int\frac{D_t\zeta(\alpha)-D_t\zeta(\beta)}{(\alpha-\beta)^2}\partial_\beta^{s}D_t\bar{\zeta} d\beta}_{L^2(|\alpha|\leq t^\frac{5}{6})}\lesssim&\norm{\int_{|\alpha-\beta|<\frac{1}{2}t^\frac{5}{6}}\frac{D_t\zeta(\alpha)-D_t\zeta(\beta)}{(\alpha-\beta)^2}\partial_\beta^{s}D_t\bar{\zeta} d\beta}_{L^2(|\alpha|\leq t^\frac{5}{6})}\\
&+\norm{\int_{|\alpha-\beta|\geq \frac{1}{2}t^\frac{5}{6}}\frac{D_t\zeta(\alpha)-D_t\zeta(\beta)}{(\alpha-\beta)^2}\partial_\beta^{s}D_t\bar{\zeta} d\beta}_{L^2(|\alpha|\leq t^\frac{5}{6})}\\
\lesssim&\norm{\partial_\alpha D_t\zeta}_{L^{\infty}(|\alpha|\lesssim t^{\frac{5}{6}})}\norm{D_t\zeta}_{H^s}\\
&+\norm{\frac{1}{\alpha^2}}_{L^1(|\alpha|\geq \frac{1}{2}t^{\frac{5}{6}})}\norm{D_t\zeta}_{H^1}\norm{D_t\zeta}_{H^s}\\
\lesssim&\ \frac{\epsilon^2}{t^{\frac{5}{6}}}.
\end{aligned}
$$
The similar estimates hold for $|\alpha|\geq t^{\frac{7}{6}}$. For the rest region, we cannot appeal to Lemma \ref{lem_(I-H)fg_main} since the lack of control over $\norm{L_0D_t\zeta}_{H^s}$. To handle the problem, we use a cut-off set
\[
\mathcal{S}(t)=\{\alpha\in\mathbb{R}\ |\ t^\frac{5}{6}<|\alpha|<t^\frac{7}{6}\}
\]
and recall (\ref{transfer_partial_alpha}):
$$
\begin{aligned}
&\norm{\int_{|\alpha-\beta|\leq \frac{1}{2}t^\frac{5}{6}}\frac{D_t\zeta(\alpha)-D_t\zeta(\beta)}{(\alpha-\beta)^2}\partial_\beta^{s}D_t\bar{\zeta} d\beta-\int_{|\alpha-\beta|\leq \frac{1}{2}t^\frac{5}{6}}\frac{D_t\zeta(\alpha)-D_t\zeta(\beta)}{(\alpha-\beta)^2}\partial_\beta^{s-1}\left(\frac{it^2}{4\beta^2}D_t\bar{\zeta}\right) d\beta}_{L^2(\mathcal{S}(t))}\\
&\lesssim\ \frac{1}{t^\frac{5}{6}}\norm{\partial_\alpha D_t\zeta}_{L^{\infty}}\norm{L_0D_t\zeta}_{H^{s-1}}+\frac{1}{t^\frac{2}{3}}\norm{\partial_\alpha D_t\zeta}_{L^{\infty}}\norm{\partial_\alpha\Omega_0D_t\zeta}_{H^{s-2}}\lesssim\frac{\epsilon^2}{t^{\frac{7}{6}-\delta_0}}.
\end{aligned}
$$
Obviously,
$$
\norm{\int_{|\alpha-\beta|> \frac{1}{2}t^\frac{5}{6}}\frac{D_t\zeta(\alpha)-D_t\zeta(\beta)}{(\alpha-\beta)^2}\partial_\beta^{s}D_t\bar{\zeta} d\beta}_{L^2(\mathcal{S}(t))}\lesssim\norm{\frac{1}{\alpha^2}}_{L^1(|\alpha|\geq \frac{1}{2}t^{\frac{5}{6}})}\norm{D_t\zeta}_{H^1}\norm{D_t\zeta}_{H^s}\lesssim\frac{\epsilon^2}{t^{\frac{5}{6}}}.
$$
By Remark \ref{rem_(I-H)fg_main} and (\ref{transfer_partial_alpha}),
$$
\begin{aligned}
\norm{K}_{L^2}
\lesssim&\norm{\frac{1}{\pi i}\int_{|\alpha-\beta|\leq \frac{1}{2}t^\frac{5}{6}}\frac{D_t\zeta(\alpha)-D_t\zeta(\beta)}{(\alpha-\beta)^2}\partial_\beta^{s-1}\left(\frac{it^2}{4\beta^2}D_t\bar{\zeta}\right) d\beta+\frac{t}{2\alpha}\partial_t D_t\zeta\partial_\alpha^{s-1}\left(\frac{it^2}{4\alpha^2}D_t\bar{\zeta}\right)}_{L^2(\mathcal{S}(t))}\\
&+\norm{\frac{1}{\alpha}L_0D_t\zeta\partial_\alpha^{s-1}\left(\frac{it^2}{4\alpha^2}D_t\bar{\zeta}\right)}_{L^2(\mathcal{S}(t))}\\
&+\norm{\partial_\alpha D_t\zeta\partial_\alpha^{s-1}\left(\frac{it^2}{4\alpha^2}D_t\bar{\zeta}\right)-\partial_\alpha D_t\zeta\partial_\alpha^{s}D_t\bar{\zeta}}_{L^2(\mathcal{S}(t))}\\
&+\norm{\partial_\alpha D_t\zeta\partial_\alpha^{s}D_t\bar{\zeta}}_{L^2(|\alpha|\leq t^\frac{5}{6})}+\norm{\partial_\alpha D_t\zeta\partial_\alpha^{s}D_t\bar{\zeta}}_{L^2(|\alpha|\geq t^\frac{7}{6})}\\
&+\norm{\partial_\alpha D_t\zeta\partial_\alpha^{s-1}D_t\frac{\bar{\zeta}_\alpha-1}{\zeta_\alpha}-\partial_\alpha D_t\zeta\partial_\alpha^sD_t\bar{\zeta}}_{L^2}+\norm{\widetilde{R}}_{L^2}+\frac{\epsilon^2}{t^\frac{5}{6}}\\
\lesssim&\ \frac{\epsilon^2}{t^{\frac{2}{3}-\delta_0}}\ln (2+t).
\end{aligned}
$$
All the above complete the estimate of $\norm{D_tb}_{H^s}$. It remains to focus on $\norm{\partial_\alpha^{s-1}D_tb}_{L^\infty}$. Rewrite \eqref{partial_alpha^sD_tb} with $s$ replaced by $s-1$:
\begin{equation}\label{partial_alpha^s-1D_t_b}
(I-\mathfrak{H})\partial_\alpha^{s-1}D_t b=\ [\partial_\alpha^{s-1}D_t,\mathfrak{H}]b+\partial_\alpha^{s-1}D_t\left(-[D_t\zeta,\mathfrak{H}]\frac{\bar{\zeta}_\alpha-1}{\zeta_\alpha}\right)+2\partial_\alpha^{s-1}D_tq.
\end{equation}
To control the second term, We consider the most representative term 
$$
\mathcal{K}=D_t[D_t\zeta,\mathfrak{H}]\frac{\partial_\alpha^s\bar\zeta}{\zeta_\alpha}=\frac{1}{\pi i}D_t\int\frac{D_t\zeta(\alpha,t)-D_t\zeta(\beta,t)}{\zeta(\alpha,t)-\zeta(\beta,t)}\partial_\beta^s\bar\zeta d\beta.
$$
Integration by parts shows that
$$
\begin{aligned}
\mathcal{K}=&\ \frac{1}{\pi i}D_t\int\frac{\partial_\beta D_t\zeta(\beta,t)}{\zeta(\alpha,t)-\zeta(\beta,t)}\partial_\beta^{s-1}\bar\zeta d\beta-\frac{1}{\pi i}D_t\int\frac{D_t\zeta(\alpha,t)-D_t\zeta(\beta,t)}{(\zeta(\alpha,t)-\zeta(\beta,t))^2}\zeta_\beta\partial_\beta^{s-1}\bar\zeta d\beta\\
=&\ \frac{1}{\pi i}\partial_t\int\frac{\partial_\beta D_t\zeta(\beta,t)}{\zeta(\alpha,t)-\zeta(\beta,t)}\partial_\beta^{s-1}\bar\zeta d\beta-\frac{1}{\pi i}\partial_t\int\frac{D_t\zeta(\alpha,t)-D_t\zeta(\beta,t)}{(\zeta(\alpha,t)-\zeta(\beta,t))^2}\zeta_\beta\partial_\beta^{s-1}\bar\zeta d\beta+R_{\mathcal{K}}
\end{aligned}
$$
where $R_{\mathcal{K}}$ contains $b$ and behaves well. Denote the first and the second integrals by $\mathcal{K}_1$ and $\mathcal{K}_2$.
$$
\mathcal{K}_1=-\frac{1}{\pi i}\int\frac{\zeta_t(\alpha,t)-\zeta_t(\beta,t)}{(\zeta(\alpha,t)-\zeta(\beta,t))^2}\partial_\beta D_t\zeta\partial_\beta^{s-1}\bar\zeta d\beta+\frac{1}{\pi i}\int\frac{1}{\zeta(\alpha,t)-\zeta(\beta,t)}\partial_t(\partial_\beta D_t\zeta\partial_\beta^{s-1}\bar\zeta)d\beta.
$$
From Lemma \ref{lem_Hf_infty}, Lemma \ref{lem_S_1} and Lemma \ref{lem_transfer_d} we have
$$
\norm{\mathcal{K}_1}_{L^\infty}\lesssim\frac{\epsilon^2}{t^{\frac{5}{4}}}\ln (2+t).
$$
Consider $\mathcal{K}_2$,
$$
\begin{aligned}
\mathcal{K}_2=&-\frac{1}{\pi i}\int\partial_t\left(\frac{\zeta_\beta}{(\zeta(\alpha,t)-\zeta(\beta,t))^2}\right)(D_t\zeta(\alpha,t)-D_t\zeta(\beta,t))\partial_\beta^{s-1}\bar\zeta d\beta\\
&-\frac{1}{\pi i}\int\left(\frac{\zeta_\beta}{(\zeta(\alpha,t)-\zeta(\beta,t))^2}\right)\partial_t((D_t\zeta(\alpha,t)-D_t\zeta(\beta,t))\partial_\beta^{s-1}\bar\zeta) d\beta.
\end{aligned}
$$
By \eqref{transfer_d_var_1} and Lemma \ref{lem_S_1}, there holds
$$
\norm{\mathcal{K}_2}\lesssim\frac{\epsilon^2}{t^{\frac{5}{4}}}\ln (2+t).
$$
Combining the estimates of $\mathcal{K}_1$ and $\mathcal{K}_2$ implies that
$$
\norm{\mathcal{K}}\lesssim\frac{\epsilon^2}{t^{\frac{5}{4}}}\ln (2+t).
$$
Return to \eqref{partial_alpha^s-1D_t_b} and we obtain that
$$
\begin{aligned}
\norm{\partial_\alpha^{s-1}D_tb}_{L^\infty}\lesssim&\ \norm{D_tq}_{W^{s-1,\infty}}+\norm{\partial_\alpha D_t\zeta}_{W^{s-2,\infty}}\norm{b}_{W^{s-1,\infty}}\ln (2+t)+\frac{\epsilon^2}{t^{\frac{5}{4}}}\ln (2+t)\\
\lesssim&\ \frac{\epsilon^2}{t^{\frac{7}{6}-\delta_0}}(\ln (2+t))^2+\frac{\lambda(\epsilon+\lambda)}{H^3}.
\end{aligned}
$$
\end{proof}
\begin{lemma}\label{lem_A}
    Suppose the bootstrap assumptions hold, we have
    \begin{equation}
        \norm{A-1}_{H^s}\lesssim\frac{\epsilon^2}{t^{\frac{2}{3}-\delta_0}}\ln (2+t)+\frac{\lambda(\epsilon+\lambda)}{H^{\frac{3}{2}}},
    \end{equation}
    \begin{equation}
        \norm{A-1}_{W^{s-3,\infty}}\lesssim\frac{\epsilon^2}{t^{\frac{3}{2}-\delta_0}}(\ln (2+t))^2+\frac{\lambda(\epsilon+\lambda)}{H^2}
    \end{equation}
and
\begin{equation}
        \norm{A-1}_{W^{s-1,\infty}}\lesssim\frac{\epsilon^2}{t^{\frac{7}{6}-\delta_0}}(\ln (2+t))^2+\frac{\lambda(\epsilon+\lambda)}{H^2}.
    \end{equation}
\end{lemma}
\begin{proof}
    We rewrite the expression of $A-1$ as below:
    
    $$
    \begin{aligned}
    (I-\mathfrak{H})(A-1)=&\ i(I-\mathfrak{H})\left(D_t^2\zeta\frac{\bar{\zeta}_\alpha-1}{\zeta_\alpha}+D_t\zeta\frac{\partial_\alpha D_t\bar{\zeta}}{\zeta_\alpha}\right)\\
    &-\frac{\lambda}{2\pi}(I-\mathfrak{H})D_t\zeta\left[\frac{1}{(\zeta-z_2)^2}-\frac{1}{(\zeta-z_1)^2}\right]\\
    &-\frac{\lambda}{2\pi}(I-\mathfrak{H})\left[\frac{D_t\zeta-\dot{z_1}}{(\zeta-z_1)^2}-\frac{D_t\zeta-\dot{z_2}}{(\zeta-z_2)^2}\right]\\
    =&\ i(I-\mathfrak{H})\left(D_t^2\zeta\frac{\bar{\zeta}_\alpha-1}{\zeta_\alpha}+D_t\zeta\frac{\partial_\alpha D_t\bar{\zeta}}{\zeta_\alpha}\right)\\
     &-\frac{\lambda}{\pi}\left[\frac{\dot{z_2}}{(\zeta-z_2)^2}-\frac{\dot{z_1}}{(\zeta-z_1)^2}\right].
    \end{aligned}  
    $$
The first term on the right hand side can be handled by imitating the proof of (\ref{D_t_b_L2}). Adopting Lemma \ref{lem_transfer_d}, Lemma \ref{lem_real_proj} and Lemma \ref{lem_z_1}, we prove that
$$\norm{A-1}_{H^{s}}\lesssim\frac{\epsilon^2}{t^{\frac{2}{3}-\delta_0}}\ln (2+t)+\frac{\lambda(\epsilon+\lambda)}{H^{\frac{3}{2}}}.$$
    Likewise,
$$\norm{A-1}_{W^{s-3,\infty}}\lesssim\frac{\epsilon^2}{t^{\frac{3}{2}-\delta_0}}(\ln (2+t))^2+\frac{\lambda(\epsilon+\lambda)}{H^2}$$
and
$$
\norm{A-1}_{\dot W^{s-2,\infty}}+\norm{A-1}_{\dot W^{s-1,\infty}}\lesssim\frac{\epsilon^2}{t^{\frac{7}{6}-\delta_0}}(\ln (2+t))^2+\frac{\lambda(\epsilon+\lambda)}{H^2}
$$
where we refer to Lemma \ref{lem_transfer_d} and mimic the proof of (\ref{D_t_b_W^s-1}).
\end{proof}

\begin{lemma}\label{lem_a_t/a_est}
    We have
\begin{equation}
        \norm{\frac{a_t}{a}\circ\kappa^{-1}A}_{H^s}\lesssim\frac{\epsilon^2}{t^{1-\delta_0}}\ln (2+t)+\frac{\lambda\epsilon}{H^{\frac{3}{2}}t^{\frac{1}{2}}}+\frac{\lambda}{H^{\frac{5}{2}}},
    \end{equation}
    \begin{equation}
        \norm{\frac{a_t}{a}\circ\kappa^{-1}A}_{W^{s-3,\infty}}\lesssim\frac{\epsilon^2}{t^{\frac{3}{2}-\delta_0}}(\ln (2+t))^2+\frac{\lambda\epsilon}{H^2t^{\frac{1}{2}}}+\frac{\lambda}{H^3}
    \end{equation}
and
    
    \begin{equation}
        \norm{\frac{a_t}{a}\circ\kappa^{-1}A}_{W^{s-1,\infty}}\lesssim\frac{\epsilon^2}{t^{\frac{7}{6}-\delta_0}}(\ln (2+t))^2+\frac{\lambda\epsilon}{H^2t^{\frac{1}{2}}}+\frac{\lambda}{H^3}.
    \end{equation}
\end{lemma}

\begin{proof}
    Rewrite \eqref{a_t/a} and we derive
$$
    \begin{aligned}
    (I-\mathfrak{H})\frac{a_t}{a}\circ\kappa^{-1}A\bar{\zeta}_\alpha=&\ 2i(I-\mathfrak{H})\left(D_t^2\zeta\frac{\partial_\alpha D_t\bar{\zeta}}{\zeta_\alpha}\right)+2i(I-\mathfrak{H})\left(D_t\zeta\frac{\partial_\alpha D_t^2\bar{\zeta}}{\zeta_\alpha}\right)\\
    &+2iD_t\zeta\frac{\partial_\alpha}{\zeta_\alpha}[\mathfrak{H},D_t]D_t\bar{\zeta}-\frac{1}{\pi}\int\left(\frac{D_t\zeta(\alpha,t)-D_t\zeta(\beta,t)}{\zeta(\alpha,t)-\zeta(\beta,t)}\right)^2(D_t\bar{\zeta})_\beta d\beta\\
    &+\frac{\lambda}{\pi}\left[\frac{\frac{\partial_\alpha}{\zeta_\alpha}(D_t\zeta)^2+\partial_t^2z_{1}-i}{(\zeta-z_1)^2}-\frac{\frac{\partial_\alpha}{\zeta_\alpha}(D_t\zeta)^2+\partial_t^2z_{2}-i}{(\zeta-z_2)^2}\right]\\
    &+\frac{2\lambda}{\pi}\left[\frac{\dot{z}_1^2-(D_t\zeta)^2}{(\zeta-z_1)^3}-\frac{\dot{z}_2^2-(D_t\zeta)^2}{(\zeta-z_2)^3}\right].
    \end{aligned}
    $$
The proof is essentially the same as that of the preceding two lemmas and we omit it.
\end{proof}
\section{Energy estimates }\label{sec_energy}
In this section, we derive the energy estimates using the equations for the quantities introduced earlier. We will study two kinds of energy: one that does not involve vector fields and one that pertains to vector fields. In what follows, we work under the bootstrap assumption. Referring to Section 6 of \cite{su2020long}, we have
\begin{equation}\label{theta}
    (D^2_t-iA\partial_\alpha)\tilde{\theta}=G^{\tilde{\theta}}_0=G_c+G_d,
\end{equation}
where
\begin{equation}\label{G_c}
    G_c=-2\left[\bar{\mathfrak{F}},\mathfrak{H}\frac{1}{\zeta_{\alpha}}+
    \bar{\mathfrak{H}}\frac{1}{\bar{\zeta}_{\alpha}}\right]\bar{\mathfrak{F}}_{\alpha}+\frac{1}{\pi i}\int\left(\frac{D_{t}\zeta(\alpha,t)-D_{t}\zeta(\beta,t)}{\zeta(\alpha,t)-\zeta(\beta,t)}\right)^2(\zeta-\bar{\zeta})_{\beta}d\beta
\end{equation}
and
\begin{equation}
    G_d=-2[\bar{q},\mathfrak{H}]\frac{\bar{\mathfrak{F}}_{\alpha}}{\zeta_{\alpha}}-2[\bar{\mathfrak{F}},\mathfrak{H}]\frac{\bar{q}_{\alpha}}{\zeta_{\alpha}}-2[\bar{q},\mathfrak{H}]\frac{\bar{q}_{\alpha}}{\zeta_{\alpha}}-4D_tq.
\end{equation}
Considering $\tilde{\sigma}=D_t\tilde{\theta}$, we have
\begin{equation}\label{sigma}
    (D^2_t-iA\partial_{\alpha})\tilde{\sigma}=G_0 ^{\tilde{\sigma}},
\end{equation}
with
\begin{equation}\label{G_0^sigma}
\begin{aligned}
G^{\tilde{\sigma}}_0=[D^2_t-iA\partial_\alpha,D_t]\tilde{\theta}+D_t(D^2_t-iA\partial_\alpha)\tilde{\theta}=i\frac{a_t}{a}\circ\kappa^{-1} A\tilde{\theta}_{\alpha}+D_t(G_c+G_d).
\end{aligned}
\end{equation}
\subsection{Energy estimates without vector fields}
We first consider the energy estimates without vector fields. For higher derivatives, we denote $\tilde{\theta}_k:=\partial_\alpha^k\tilde{\theta}$ and $\tilde{\sigma}_k:=\partial_\alpha^k\tilde{\sigma}$. Applying $\partial_{\alpha}$ on both sides of (\ref{theta}) and (\ref{sigma}) $k$ times simultaneously for $1\leq k\leq s$, we obtain that
\begin{equation}
\begin{aligned}
    (D^2_t-iA\partial_\alpha)\tilde{\theta}_k&=G^{\tilde{\theta}}_k,\\
    (D^2_t-iA\partial_\alpha)\tilde{\sigma}_k&=G^{\tilde{\sigma}}_k,
\end{aligned}
\end{equation}
where
\begin{equation}
\begin{aligned}
    G^{\tilde{\theta}}_k=&\ \partial_\alpha^k G_0^{\tilde{\theta}}+[D_t^2-iA\partial_\alpha,\partial_\alpha^k]\tilde{\theta},\\
    G^{\tilde{\sigma}}_k=&\ \partial_\alpha^k G_0^{\tilde{\sigma}}+[D_t^2-iA\partial_\alpha,\partial_\alpha^k]\tilde{\sigma}.
\end{aligned}
\end{equation}
We introduce the energies for $0\leq k\leq s$ that
\begin{equation}
\begin{aligned}
    E^{\tilde{\theta}}_k&=\int\frac{|D_t\tilde{\theta}_k|^2}{A}+i\tilde{\theta}_k\partial_{\beta}\overline{\tilde{\theta}_k}d{\beta},\\
    E^{\tilde{\sigma}}_k&=\int\frac{|D_t\tilde{\sigma}_k|^2}{A}+i\tilde{\sigma}_k\partial_{\beta}\overline{\tilde{\sigma}_k}d{\beta},
\end{aligned}
\end{equation}
here $\tilde{\theta}_0=\tilde{\theta}$ and $\tilde{\sigma}_0=\tilde{\sigma}$. Next we differentiate the energies above with respect to t:
\begin{equation}\label{d_t_E}
\begin{aligned}
    \frac{d}{dt}E^{\tilde{\theta}}_k&=\operatorname{Re}\left\{\int\frac{2D_t\overline{\tilde{\theta}_k}}{A}G^{\tilde{\theta}}_kd{\beta}\right\}-\int\frac{|D_t\tilde{\theta}_k|^2}{A}\left(\frac{a_t}{a}\circ\kappa^{-1}\right)d\beta,\\
    \frac{d}{dt}E^{\tilde{\sigma}}_k&=\operatorname{Re}\left\{\int\frac{2D_t\overline{\tilde{\sigma}_k}}{A}G^{\tilde{\sigma}}_kd{\beta}\right\}-\int\frac{|D_t\tilde{\sigma}_k|^2}{A}\left(\frac{a_t}{a}\circ\kappa^{-1}\right)d\beta.
\end{aligned}
\end{equation}
Define
\begin{equation}
    \mathfrak{E}_s=\sum_{k=0}^s(E_k^{\tilde{\theta}}+E_k^{\tilde{\sigma}}).
\end{equation}

It's necessary to establish the relation between the norms of $\zeta$ and $\mathfrak{E}_s$ to close the bootstrap argument. We illustrate it by the following lemma:
\begin{lemma}\label{lem_d_t_zeta_E}
    Suppose $\epsilon$ is small enough, then for $t\in[0,T]$ we have
$$
    \norm{D_t\zeta}_{H^{s+\frac{1}{2}}}+\norm{D_t^2\zeta}_{H^s}+\norm{\zeta_\alpha-1}_{H^s}\lesssim\sqrt{\mathfrak{E}_s}+\frac{\epsilon^3}{t^\frac{1}{6}}\ln (2+t)+\frac{\epsilon^2}{t^\frac{1}{6}}+\frac{\lambda}{H^\frac{3}{2}}+\frac{\epsilon\lambda}{H^{\frac{3}{2}}}\ln (2+t).
    $$
\end{lemma}
\begin{proof}
    According to (\ref{theta_t_zeta_t}), (\ref{sigma_t_zeta_tt}) and the definition of $\mathfrak{E}_s$ we have
\begin{equation}\label{D_t_zeta_E_s}
    \norm{D_t\zeta}_{H^s}+\norm{D_t^2\zeta}_{H^s}\lesssim\sqrt{\mathfrak{E}_s}+\left(\frac{\epsilon^2}{t^\frac{1}{6}}+\frac{\lambda}{H^\frac{3}{2}}\right).
    \end{equation}
Since $\zeta_\alpha-1=-iD_t^2\zeta+(1-A)\zeta_\alpha$, we have
$$
    \norm{\zeta_\alpha-1}_{H^s}\lesssim\sqrt{\mathfrak{E}_s}+\left(\frac{\epsilon^2}{t^\frac{1}{6}}+\frac{\lambda}{H^\frac{3}{2}}\right).
    $$
    
It remains to check $\norm{\partial_\alpha^sD_t\zeta}_{\dot{H}^\frac{1}{2}}$. By Lemma \ref{lem_theta_t_zeta_t} we only need to consider $\norm{\tilde{\sigma}_s}_{\dot{H}^\frac{1}{2}}$. Denote
$$
    \phi_{\tilde{\sigma}}=\frac{I-\mathfrak{H}}{2}\tilde{\sigma}_s, \quad r_{\tilde{\sigma}}=\frac{I+\mathfrak{H}}{2}\tilde{\sigma}_s.
    $$
There exists a holomorphic function $\psi$ which is defined on $\Omega(t)^c$ and satisfies $\psi\circ \zeta(\alpha,t)=\phi_{\tilde{\sigma}}$. From Green's formula we have the following observation:
$$
    i\int\phi_{\tilde{\sigma}}\partial_\beta\overline{\phi_{\tilde{\sigma}}}d\beta=-i\int\partial_\beta\phi_{\tilde{\sigma}}\overline{\phi_{\tilde{\sigma}}}d\beta=\operatorname{Re}\left\{i\int\phi_{\tilde{\sigma}}\partial_\beta\overline{\phi_{\tilde{\sigma}}}d\beta\right\}=\int_{\partial\Omega(t)^c} \psi\cdot\frac{\partial\psi}{\partial\mathbf{n}}ds=\int_{\Omega(t)^c}|\nabla\psi|^2dxdy
    $$
Adopting the trace formula:
$$
    \norm{\phi_{\tilde{\sigma}}}_{\dot{H}^{\frac{1}{2}}}^2\lesssim i\int\phi_{\tilde{\sigma}}\partial_\beta\overline{\phi_{\tilde{\sigma}}}d\beta.
    $$
Since
$$
    \begin{aligned}
    i\int\tilde{\sigma}_k\partial_\beta\overline{\tilde{\sigma}_k}=&\ i\int\phi_{\tilde{\sigma}}\partial_\beta\overline{\phi_{\tilde{\sigma}}}d\beta+i\int r_{\tilde{\sigma}}\partial_\beta\overline{\phi_{\tilde{\sigma}}}d\beta+i\int\phi_{\tilde{\sigma}}\partial_\beta\overline{r_{\tilde{\sigma}}}d\beta+i\int r_{\tilde{\sigma}}\partial_\beta\overline{r_{\tilde{\sigma}}}d\beta\\
    =&\ i\int\phi_{\tilde{\sigma}}\partial_\beta\overline{\phi_{\tilde{\sigma}}}d\beta+R,
    \end{aligned}
    $$
we have
$$
    |R|\lesssim\norm{\Lambda r_{\tilde{\sigma}}}_{L^2}\norm{\Lambda \phi_{\tilde{\sigma}}}_{L^2}+\norm{\Lambda r_{\tilde{\sigma}}}_{L^2}^2
    $$
where $\Lambda=\sqrt{|\partial_\alpha|}$. Therefore the employment of Young's inequality explores that
$$
    \begin{aligned}
    \norm{\phi_{\tilde{\sigma}}}_{\dot{H}^{\frac{1}{2}}}^2\leq& C\left(i\int\tilde{\sigma}_k\partial_\beta\overline{\tilde{\sigma}_k}-R\right)\\
    \leq &C(\mathfrak{E}_s+\norm{\Lambda r_{\tilde{\sigma}}}_{L^2}\norm{\Lambda \phi_{\tilde{\sigma}}}_{L^2}+\norm{\Lambda r_{\tilde{\sigma}}}_{L^2}^2)\\
    \leq& \frac{1}{2} \norm{\phi_{\tilde{\sigma}}}_{\dot{H}^{\frac{1}{2}}}^2+C(\mathfrak{E}_s+\norm{\Lambda r_{\tilde{\sigma}}}_{L^2}^2)
    \end{aligned}
    $$
which implies
$$
    \norm{\phi_{\tilde{\sigma}}}_{\dot{H}^{\frac{1}{2}}}^2\lesssim\mathfrak{E}_s+\norm{\Lambda r_{\tilde{\sigma}}}_{L^2}^2.
    $$
Recall the definition of $r_{\tilde{\sigma}}$:
\begin{equation}\label{r_sigma}
    \begin{aligned}
    r_{\tilde{\sigma}}=\frac{I+\mathfrak{H}}{2}\partial_\alpha^sD_t\tilde\theta=\frac{1}{2}[\mathfrak{H},\partial_\alpha^s]D_t\tilde{\theta}+\frac{1}{2}\partial_\alpha^s[\mathfrak{H},D_t]\tilde{\theta}=\frac{1}{2}\sum_{j=0}^{s-1}\partial_\alpha^j[\mathfrak{H},\partial_\alpha]\partial_\alpha^{s-j-1}D_t\tilde{\theta}+\frac{1}{2}\partial_\alpha^s[\mathfrak{H},D_t]\tilde{\theta}.
    \end{aligned}
    \end{equation}
Utilizing Lemma \ref{app_S1}, Lemma \ref{lem_theta_t_zeta_t} and Lemma \ref{lem_Hf_infty}, we deduce that
$$
    \begin{aligned}
    \norm{\Lambda\partial_\alpha^j[\mathfrak{H},\partial_\alpha]\partial_\alpha^{s-j-1}D_t\tilde{\theta}}_{L^2}\lesssim&\norm{\zeta_\alpha-1}_{W^{s-2,\infty}}\norm{D_t\tilde{\theta}}_{H^{s+\frac{1}{2}}}\\
    &+\norm{\zeta_\alpha-1}_{H^s}\left(\frac{1}{t^2}\norm{D_t\tilde{\theta}}_{H^s}+\norm{\partial_\alpha D_t\tilde{\theta}}_{W^{s-3,\infty}}\ln (2+t)\right)\\
    \lesssim&\ \frac{\epsilon^2}{t^\frac{1}{2}}+\epsilon\left(\norm{D_t\tilde{\theta}-2D_t\zeta}_{W^{s-2,\infty}}+\norm{2\partial_\alpha D_t\zeta}_{W^{s-2,\infty}}\right)\ln (2+t)\\
    \lesssim&\ \frac{\epsilon^2}{t^\frac{1}{2}}\ln (2+t)+\epsilon\norm{D_t\tilde{\theta}-2D_t\zeta}_{H^{s-1}}\ln (2+t)\\
    \lesssim&\ \frac{\epsilon^2}{t^\frac{1}{2}}\ln (2+t)+\frac{\epsilon^3}{t^\frac{1}{6}}\ln (2+t)+\frac{\epsilon\lambda}{H^{\frac{3}{2}}}\ln (2+t).
    \end{aligned}
    $$
For the last term in (\ref{r_sigma}), we consider the most extreme cases:
$$
    \begin{aligned}
    r_1=&\ \Lambda[\mathfrak{H},D_t\zeta]\frac{1}{\zeta_\alpha}\partial_\alpha^{s+1}\tilde{\theta},\quad r_2=\Lambda[\mathfrak{H},\partial_\alpha^sD_t\zeta]\frac{\tilde{\theta}_\alpha}{\zeta_\alpha}.
    \end{aligned}
    $$
Due to Lemma \ref{app_S1} and Lemma \ref{lem_theta_alpha}:
$$
    \norm{r_1}_{L^2}\leq\norm{[\mathfrak{H},D_t\zeta]\frac{\partial_\alpha^{s+1}\tilde{\theta}}{\zeta_\alpha}}_{H^1}\lesssim\norm{\partial_\alpha D_t\zeta}_{W^{2,\infty}}\norm{\tilde{\theta}_\alpha}_{H^s}\lesssim\frac{\epsilon^2}{t^\frac{1}{2}}.
    $$
Moreover,
$$
    r_2=\Lambda\mathfrak{H}\left(\partial_\alpha^sD_t\zeta\frac{\tilde{\theta}_\alpha}{\zeta_\alpha}\right)+\Lambda\left(\partial_\alpha^sD_t\zeta\mathfrak{H}\frac{\tilde{\theta}_\alpha}{\zeta_\alpha}\right)=\Lambda\mathfrak{H}\left(\partial_\alpha^sD_t\zeta\frac{\tilde{\theta}_\alpha}{\zeta_\alpha}\right)+[\Lambda,\mathfrak{H}\frac{\tilde{\theta}_\alpha}{\zeta_\alpha}]\partial_\alpha^sD_t\zeta+\mathfrak{H}\frac{\tilde{\theta}_\alpha}{\zeta_\alpha}\Lambda\partial_\alpha^sD_t\zeta.
    $$
According to Lemma 2.14 in \cite{Yosihara}, the estimate we have just proved with respect to $\norm{D_t\zeta}_{H^s}$ and fractional Leibniz rule,
$$
    \begin{aligned}
    \norm{r_2}_{L^2}\lesssim&\norm{D_t\zeta}_{H^{s+\frac{1}{2}}}\norm{\tilde{\theta}_\alpha}_{L^\infty}+\norm{D_t\zeta}_{H^s}\norm{\tilde{\theta}_\alpha}_{W^{\frac{1}{2},\infty}}\\
    &+\norm{\mathfrak{H}\frac{\tilde{\theta}_\alpha}{\zeta_\alpha}}_{H^2}\norm{\partial_\alpha^sD_t\zeta}_{H^{-\frac{1}{2}}}+\norm{\mathfrak{H}\frac{\tilde{\theta}_\alpha}{\zeta_\alpha}}_{L^2}\norm{D_t\zeta}_{H^{s+\frac{1}{2}}}\\
    \lesssim&\ \frac{\epsilon^2}{t^\frac{1}{2}}+\epsilon\left(\sqrt{\mathfrak{E_s}}+\frac{\epsilon^2}{t^\frac{1}{6}}+\frac{\lambda}{H^\frac{3}{2}}\right).
    \end{aligned}
    $$
Summing up the above, we deduce that
$$
    \norm{\Lambda r_{\tilde{\sigma}}}_{L^2}\lesssim\epsilon\sqrt{\mathfrak{E_s}}+\frac{\epsilon^2}{t^\frac{1}{2}}\ln (2+t)+\frac{\epsilon^3}{t^\frac{1}{6}}\ln (2+t)+\frac{\epsilon\lambda}{H^{\frac{3}{2}}}\ln (2+t).
    $$
Consequently,
$$
    \begin{aligned}
    \norm{\tilde{\sigma}_s}_{\dot{H}^\frac{1}{2}}^2\lesssim&\norm{\phi_{\tilde{\sigma}}}_{\dot{H}^{\frac{1}{2}}}^2+\norm{r_{\tilde{\sigma}}}_{\dot{H}^{\frac{1}{2}}}^2\lesssim \mathfrak{E}_s+\left(\frac{\epsilon^2}{t^\frac{1}{2}}\ln (2+t)+\frac{\epsilon^3}{t^\frac{1}{6}}\ln (2+t)+\frac{\epsilon\lambda}{H^{\frac{3}{2}}}\ln (2+t)\right)^2.
    \end{aligned}
    $$

\end{proof}

In the following two lemmas we estimate the terms on the right‑hand side of \eqref{d_t_E} one by one.
\begin{lemma}\label{lem_[P,partial_alpha]}
    We have
    \begin{equation}
        \norm{[D_t^2-iA\partial_\alpha,\partial_\alpha^k]\tilde{\theta}}_{L^2}\lesssim\frac{\epsilon^3}{t^{1+\delta}}+\frac{\lambda\epsilon}{H^{\frac{3}{2}}t^{\frac{1}{2}}}+\left(\frac{\epsilon^2}{t^{1+\delta}}+\frac{\lambda}{H^3}\right)\sqrt{\mathfrak{E}_s}
    \end{equation}
and
\begin{equation}\label{commute_sigma}
        \norm{[D_t^2-iA\partial_\alpha,\partial_\alpha^k]\tilde{\sigma}}_{L^2}\lesssim\frac{\epsilon^3}{t^{1+\delta}}+\frac{\lambda\epsilon}{H^{\frac{3}{2}}t^\frac{1}{2}}+\left(\frac{\epsilon^2}{t^{1+\delta}}+\frac{\lambda}{H^3}\right)\sqrt{\mathfrak{E}_s}
    \end{equation}
for $1\leq k\leq s$.
\end{lemma}
\begin{proof}
    Straight calculation yields
$$\begin{aligned}
        [D^2_t,\partial_\alpha^k]=&\ D_t[b\partial_\alpha,\partial_\alpha^k]+[b\partial_\alpha,\partial_\alpha^k]D_t\\
        =&-D_t\sum_{j=1}^k\frac{k!}{j!(k-j)!}\partial^{j}_\alpha b\partial_\alpha^{k-j+1}-\sum_{j=1}^k\frac{k!}{j!(k-j)!}\partial^{j}_\alpha b\partial_\alpha^{k-j+1}D_t.
    \end{aligned}$$
Using Lemma \ref{lem_b} and Lemma \ref{lem_theta_alpha}, we claim that
$$\begin{aligned}
        \norm{D_t\partial_\alpha^jb\partial_\alpha^{k-j+1}\tilde{\theta}}_{L^2}\lesssim&\ (\norm{[D_t,\partial_\alpha^j]b}_{L^2}+\norm{D_tb}_{H^s})\norm{\partial_\alpha^{k-j+1}\tilde{\theta}}_{L^\infty}\\
        \lesssim&\ \frac{\epsilon^3}{t^{1+\delta}}+\frac{\lambda\epsilon(\epsilon+\lambda)}{H^{\frac{5}{2}}t^{\frac{1}{2}}},
    \end{aligned}$$
when $k-j+1\leq s-2$. Otherwise we get
$$\begin{aligned}
        \norm{D_t\partial_\alpha^jb\partial_\alpha^{k-j+1}\tilde{\theta}}_{L^2}\lesssim&\ (\norm{[D_t,\partial_\alpha^j]b}_{L^\infty}+\norm{D_tb}_{W^{s-2,\infty}})\norm{\partial_\alpha^{k-j+1}\tilde{\theta}}_{L^2}\\
        \lesssim&\ \frac{\epsilon^3}{t^{1+\delta}}+\frac{\lambda\epsilon(\epsilon+\lambda)}{H^3}.
    \end{aligned}$$
Similarly,
$$\begin{aligned}
        \norm{\partial_\alpha^jbD_t\partial_\alpha^{k-j+1}\tilde{\theta}}_{L^2}\leq&\norm{\partial_\alpha^jb}_{L^\infty}\norm{D_t\partial_\alpha^{k-j+1}\tilde{\theta}}_{L^2}\\
        \lesssim&\left(\frac{\epsilon^2}{t^{1+\delta}}+\frac{\lambda}{H^3}\right)\sqrt{\mathfrak{E}_s}
    \end{aligned}$$
when $j\leq s-2$ and for $j>s-2$:
$$\begin{aligned}
        \norm{\partial_\alpha^jbD_t\partial_\alpha^{k-j+1}\tilde{\theta}}_{L^2}\lesssim&\norm{\partial_\alpha^jb}_{L^2}\norm{D_t\partial_\alpha^{k-j+1}\tilde{\theta}}_{L^\infty}\\
        \lesssim&\ \frac{\epsilon^3}{t^{1+\delta}}+\frac{\lambda\epsilon}{H^{\frac{3}{2}}t^{\frac{1}{2}}}.
    \end{aligned}$$
The term $[A\partial_\alpha,\partial_\alpha^k]\tilde{\theta}$ can be treated similarly. For the same reason, (\ref{commute_sigma}) holds.
\end{proof}
\begin{lemma}\label{lem_G_d}
    We have
    \begin{equation}\label{G_d}
    \norm{G_d}_{H^s}\lesssim\frac{\lambda}{H^{\frac{5}{2}}}\sqrt{\mathfrak{E}_s}+\frac{\lambda^2}{H^{\frac{3}{2}}}+\frac{\lambda\epsilon}{H^{\frac{3}{2}}t^{\frac{1}{4}}}
    \end{equation}
and
\begin{equation}
    \norm{D_tG_d}_{H^s}\lesssim\frac{\lambda}{H^{\frac{5}{2}}}\sqrt{\mathfrak{E}_s}+\frac{\lambda^2}{H^{\frac{3}{2}}}+\frac{\lambda\epsilon}{H^{\frac{3}{2}}t^{\frac{1}{4}}}.
    \end{equation}
\end{lemma}
\begin{proof}
    Using Lemma \ref{lem_Hf_infty} and (\ref{D_t_zeta_E_s}), we obtain that
    $$
    \begin{aligned}
    \norm{[\bar{q},\mathfrak{H}]\frac{\bar{\mathfrak{F}}_\alpha}{\zeta_\alpha}}_{H^s}\lesssim&\ \norm{q_\alpha}_{W^{s-2,\infty}}\norm{D_t\zeta-\bar{q}}_{H^s}\\
    &+\norm{q}_{H^s}\left(\norm{\partial_\alpha D_t\zeta-\bar q_\alpha}_{W^{s-2,\infty}}\ln (2+t)+\frac{1}{t^2}\norm{\partial_\alpha D_t\zeta-\bar q_\alpha}_{H^{s-1}}\right)\\
    \lesssim&\ \frac{\lambda}{H^{\frac{5}{2}}}\left(\sqrt{\mathfrak{E}_s}+\frac{\epsilon^2}{t^{\frac{1}{6}}}+\frac{\lambda}{H^{\frac{3}{2}}}\right).
    \end{aligned}
    $$
Analogously,
$$
    \norm{[\bar{\mathfrak{F}},\mathfrak{H}]\frac{\bar{q}_\alpha}{\zeta_\alpha}+[\bar{q},\mathfrak{H}]\frac{\bar{q}_\alpha}{\zeta_\alpha}}_{H^s}=\norm{[D_t\zeta,\mathfrak{H}]\frac{\bar{q}_\alpha}{\zeta_\alpha}}_{H^s}\lesssim\frac{\lambda}{H^{\frac{5}{2}}}\left(\sqrt{\mathfrak{E}_s}+\frac{\epsilon^2}{t^{\frac{1}{6}}}+\frac{\lambda}{H^{\frac{3}{2}}}\right).
    $$
    Recall (\ref{D_t_q}):
    $$
    D_tq=-\frac{\lambda i}{2\pi}\frac{\dot{z_1}-\dot{z_2}}{(\zeta-z_1)(\zeta-z_2)}+\frac{\lambda i}{2\pi}(z_1-z_2)\left(\frac{D_t\zeta-\dot{z_1}}{(\zeta-z_1)^2(\zeta-z_2)}+\frac{D_t\zeta-\dot{z_2}}{(\zeta-z_1)(\zeta-z_2)^2}\right).
    $$
    From Lemma \ref{lem_z_1} we deduce that:
    $$
    \norm{D_tq}_{H^s}\lesssim\frac{\lambda^2}{H^{\frac{3}{2}}}+\frac{\lambda\epsilon}{H^{\frac{3}{2}}t^{\frac{1}{4}}}.
    $$
    Hence we prove (\ref{G_d}). The similar estimates hold for $\norm{D_tG_d}_{H^s}$.
\end{proof}

We return to $\mathfrak{E}_s$ and claim that $\mathfrak{E}_s$ is globally bounded. In fact, the main term of $\partial_\alpha^kG_c$ is $i\varphi D_t\tilde{\theta_k}$ where $\varphi$ is a real function. Therefore, \eqref{d_t_E} indicates that $\frac{d\mathfrak{E}_s}{dt}$ is integrable with respect to $t$. See the following theorem for more details. 
\begin{thm}\label{thm_d_t_E_s}
    There exists $\delta>0$ such that
    \begin{equation}\label{d_t_E_s}
        \left|\frac{d\mathfrak{E}_s}{dt}\right|\lesssim\left(\frac{\epsilon^3}{t^{1+\delta}}+\frac{\lambda\epsilon}{H^{\frac{3}{2}}t^{\frac{1}{4}}}+\frac{\lambda^2}{H^{\frac{3}{2}}}\right)\sqrt{\mathfrak{E}_s}+\left(\frac{\epsilon^2}{t^{1+\delta}}+\frac{\lambda\epsilon}{H^2t^{\frac{1}{2}}}+\frac{\lambda}{H^{\frac{5}{2}}}\right)\mathfrak{E_s}.
    \end{equation}
\end{thm}
\begin{proof}
    From (\ref{d_t_E}), Lemma \ref{lem_a_t/a_est}, Lemma \ref{lem_[P,partial_alpha]} and Lemma \ref{lem_G_d}, we verify that
$$
\begin{aligned}
    &\left|\frac{d}{dt}\mathfrak{E}_s-\sum _{k=0}^s\operatorname{Re}\left\{\int\frac{2D_t\overline{\tilde{\theta}_k}}{A}\partial_\alpha^kG_c+\int\frac{2D_t\overline{\tilde{\sigma}_k}}{A}\partial_\alpha^kD_tG_c\right\}\right|\\
    \lesssim&\left(\frac{\epsilon^3}{t^{1+\delta}}+\frac{\lambda\epsilon}{H^{\frac{3}{2}}t^{\frac{1}{4}}}+\frac{\lambda^2}{H^{\frac{3}{2}}}\right)\sqrt{\mathfrak{E}_s}+\left(\frac{\epsilon^2}{t^{1+\delta}}+\frac{\lambda\epsilon}{H^2t^{\frac{1}{2}}}+\frac{\lambda}{H^{\frac{5}{2}}}\right)\mathfrak{E_s}.
    \end{aligned}
    $$
It remains to deal with the integral. We perform integration by parts on $G_c$:
$$
    \begin{aligned}
    G_c=&\ \frac{1}{\pi i}\int\frac{\partial_\beta(\bar{\mathfrak{F}}(\alpha,t)-\bar{\mathfrak{F}}(\beta,t))^2}{\zeta(\alpha,t)-\zeta(\beta,t)}-\frac{\partial_\beta(\bar{\mathfrak{F}}(\alpha,t)-\bar{\mathfrak{F}}(\beta,t))^2}{\bar{\zeta}(\alpha,t)-\bar{\zeta}(\beta,t)}d\beta\\
    &+\frac{1}{\pi i}\int\left(\frac{D_{t}\zeta(\alpha,t)-D_{t}\zeta(\beta,t)}{\zeta(\alpha,t)-\zeta(\beta,t)}\right)^2(\zeta-\bar{\zeta})_{\beta}d\beta\\
    =&-\frac{1}{\pi i}\int\left(\frac{\bar{\mathfrak{F}}(\alpha,t)-\bar{\mathfrak{F}}(\beta,t)}{\zeta(\alpha,t)-\zeta(\beta,t)}\right)^2\zeta_{\beta}-\left(\frac{\bar{\mathfrak{F}}(\alpha,t)-\bar{\mathfrak{F}}(\beta,t)}{\bar{\zeta}(\alpha,t)-\bar{\zeta}(\beta,t)}\right)^2\bar{\zeta}_{\beta}d\beta
    \\
    &+\frac{1}{\pi i}\int\left(\frac{D_{t}\zeta(\alpha,t)-D_{t}\zeta(\beta,t)}{\zeta(\alpha,t)-\zeta(\beta,t)}\right)^2(\zeta-\bar{\zeta})_{\beta}d\beta\\
    =&\ \frac{1}{\pi i}\int\left[\left(\frac{D_{t}\zeta(\alpha,t)-D_{t}\zeta(\beta,t)}{\zeta(\alpha,t)-\zeta(\beta,t)}\right)^2-\left(\frac{\bar{\mathfrak{F}}(\alpha,t)-\bar{\mathfrak{F}}(\beta,t)}{\zeta(\alpha,t)-\zeta(\beta,t)}\right)^2\right]\zeta_{\beta}d\beta\\
    &-\frac{1}{\pi i}\int\left[\left(\frac{D_{t}\zeta(\alpha,t)-D_{t}\zeta(\beta,t)}{\bar{\zeta}(\alpha,t)-\bar{\zeta}(\beta,t)}\right)^2-\left(\frac{\bar{\mathfrak{F}}(\alpha,t)-\bar{\mathfrak{F}}(\beta,t)}{\bar{\zeta}(\alpha,t)-\bar{\zeta}(\beta,t)}\right)^2\right]\bar{\zeta}_{\beta}d\beta\\
    &+\frac{1}{\pi i}\int\left[\left(\frac{D_{t}\zeta(\alpha,t)-D_{t}\zeta(\beta,t)}{\bar{\zeta}(\alpha,t)-\bar{\zeta}(\beta,t)}\right)^2-\left(\frac{D_t\zeta(\alpha,t)-D_t\zeta(\beta,t)}{\zeta(\alpha,t)-\zeta(\beta,t)}\right)^2\right]\bar{\zeta}_{\beta}d\beta.
    \end{aligned}
    $$
The first term satisfies
$$
    \begin{aligned}
    &\norm{\frac{1}{\pi i}\int\left[\left(\frac{D_{t}\zeta(\alpha,t)-D_{t}\zeta(\beta,t)}{\zeta(\alpha,t)-\zeta(\beta,t)}\right)^2-\left(\frac{\bar{\mathfrak{F}}(\alpha,t)-\bar{\mathfrak{F}}(\beta,t)}{\zeta(\alpha,t)-\zeta(\beta,t)}\right)^2\right]\zeta_{\beta}d\beta}_{H^s}\\
    =&\norm{\frac{1}{\pi i}\int\frac{(\bar{q}(\alpha,t)-\bar{q}(\beta,t))(D_t\zeta(\alpha,t)+\bar{\mathfrak{F}}(\alpha,t)-D_t\zeta(\beta,t)-\bar{\mathfrak{F}}(\beta,t))}{(\zeta(\alpha,t)-\zeta(\beta,t))^2}\zeta_\beta d\beta}_{H^s}\\
    \lesssim&\norm{D_t\zeta+\bar{\mathfrak{F}}}_{H^s}\norm{q_\alpha}_{W^{s-2,\infty}}\\
    \lesssim&\ \frac{\lambda}{H^{3}}\left(\sqrt{\mathfrak{E}_s}+\frac{\epsilon^2}{t^{\frac{1}{6}}}+\frac{\lambda}{H^{\frac{3}{2}}}\right).
    \end{aligned}
    $$
The same holds for the second term. We denote the last integral as $\mathfrak{G}$ and see that
$$
    \begin{aligned}
    \mathfrak{G}=&\ \frac{1}{\pi i}\int\frac{[(\zeta-\bar{\zeta})(\alpha,t)-(\zeta-\bar{\zeta})(\beta,t)](D_t\zeta(\alpha,t)-D_t\zeta(\beta,t))^2}{(\zeta(\alpha,t)-\zeta(\beta,t))^2(\bar{\zeta}(\alpha,t)-\bar{\zeta}(\beta,t))}\bar{\zeta}_\beta d\beta\\
    &+\frac{1}{\pi i}\int\frac{[(\zeta-\bar{\zeta})(\alpha,t)-(\zeta-\bar{\zeta})(\beta,t)](D_t\zeta(\alpha,t)-D_t\zeta(\beta,t))^2}{(\zeta(\alpha,t)-\zeta(\beta,t))(\bar{\zeta}(\alpha,t)-\bar{\zeta}(\beta,t))^2}\bar{\zeta}_\beta d\beta\\
    =&\ \frac{2}{\pi i}\int\frac{[(\zeta(\alpha,t)-\alpha)-(\zeta(\beta,t)-\beta)](\zeta_t(\alpha,t)-\zeta_t(\beta,t))^2}{(\alpha-\beta)^3}d\beta\\
    &-\frac{2}{\pi i}\int\frac{[(\bar{\zeta}(\alpha,t)-\alpha)-(\bar{\zeta}(\beta,t)-\beta)](\zeta_t(\alpha,t)-\zeta_t(\beta,t))^2}{(\alpha-\beta)^3}d\beta+R_\mathfrak{G},
    \end{aligned}
    $$
where $R_\mathfrak{G}$ is quartic and easily handled. To analyze $\partial_\alpha^k\mathfrak{G}$, we choose the most extreme scenario:
$$
    \mathfrak{G}_1=\frac{2}{\pi i}\int\frac{[(\zeta(\alpha,t)-\alpha)-(\zeta(\beta,t)-\beta)](\partial_\alpha^k \zeta_t(\alpha,t)-\partial_\beta^k \zeta_t(\beta,t))(\zeta_t(\alpha,t)-\zeta_t(\beta,t))}{(\alpha-\beta)^3}d\beta,
    $$
$$
    \mathfrak{G}_2=\frac{2}{\pi i}\int\frac{[(\bar\zeta(\alpha,t)-\alpha)-(\bar\zeta(\beta,t)-\beta)](\partial_\alpha^k \zeta_t(\alpha,t)-\partial_\beta^k \zeta_t(\beta,t))(\zeta_t(\alpha,t)-\zeta_t(\beta,t))}{(\alpha-\beta)^3}d\beta
    $$
and we aim to prove that
\begin{equation}\label{mathfrak_G_1}
    \norm{\mathfrak{G}_1}_{L^2}\lesssim\frac{\epsilon^3}{t^{1+\delta}},
    \end{equation}
\begin{equation}\label{mathfrak_G_2}
        \norm{\mathfrak{G}_2-1_{t^\frac{3}{4}\leq|\alpha|\leq t^\frac{5}{4}}\frac{it}{2\alpha}D_t\tilde{\theta}_k|\zeta_\alpha-1|^2}_{L^2}\lesssim\frac{\epsilon^3}{t^{1+\delta}}.
    \end{equation}
The strategy we will adopt can be used to treat other general situations. Apply integration by parts:
$$
    \begin{aligned}
    &\ \mathfrak{G}_1\\
    =&\ \frac{2}{\pi i}\int[(\zeta(\alpha,t)-\alpha)-(\zeta(\beta,t)-\beta)](\partial_\alpha^k \zeta_t(\alpha,t)-\partial_\beta^k \zeta_t(\beta,t))
    (\zeta_t(\alpha,t)-\zeta_t(\beta,t))\partial_\beta\left(\frac{1}{2(\alpha-\beta)^2}\right)d\beta\\
    =&\ \frac{1}{\pi i}\int\frac{(\partial_\alpha^k \zeta_t(\alpha,t)-\partial_\beta^k \zeta_t(\beta,t))(\zeta_t(\alpha,t)-\zeta_t(\beta,t))}{(\alpha-\beta)^2}(\zeta_\beta-1)d\beta\\
    &+\frac{1}{\pi i}\int\frac{[(\zeta(\alpha,t)-\alpha)-(\zeta(\beta,t)-\beta)](\zeta_t(\alpha,t)-\zeta_t(\beta,t))}{(\alpha-\beta)^2}\partial_\beta^{k+1}\zeta_td\beta\\
    &+\frac{1}{\pi i}\int\frac{[(\zeta(\alpha,t)-\alpha)-(\zeta(\beta,t)-\beta)](\partial_\alpha^k\zeta_t(\alpha,t)-\partial_\beta^k\zeta_t(\beta,t))}{(\alpha-\beta)^2}\partial_\beta\zeta_td\beta\\
    =&\ \mathfrak{G}_{11}+\mathfrak{G}_{12}+\mathfrak{G}_{13}.
    \end{aligned}
    $$
In the following operations, we want to transfer the derivative with respect to $t$ as needed. For instance, if $k=s$:
$$
    \begin{aligned}
    \mathfrak{G}_{11}=\ &\frac{1}{\pi i}\int\frac{\zeta_t(\alpha,t)-\zeta_t(\beta,t)}{(\alpha-\beta)^2}[\partial_t(\partial_\alpha^{s}\zeta(\alpha,t)-\partial_\beta^s\zeta(\beta,t))(\zeta_\beta-1)-(\partial_\alpha^{s}\zeta(\alpha,t)-\partial_\beta^s\zeta(\beta,t))\partial_t\zeta_\beta]\\
    &+\frac{1}{\pi i}\int\frac{(\partial_\alpha^s \zeta(\alpha,t)-\partial_\beta^s \zeta(\beta,t))(\zeta_t(\alpha,t)-\zeta_t(\beta,t))}{(\alpha-\beta)^2}\partial_t\zeta_\beta d\beta.
    \end{aligned}
    $$
From \eqref{transfer_d_var_2} we have
$$
    \norm{\mathfrak{G}_{11}-\frac{1}{\pi i}\int\frac{(\partial_\alpha^s \zeta(\alpha,t)-\partial_\beta^s \zeta(\beta,t))(\zeta_t(\alpha,t)-\zeta_t(\beta,t))}{(\alpha-\beta)^2}\partial_t\zeta_\beta d\beta}_{L^2}\lesssim\frac{\epsilon^3}{t^{\frac{3}{2}-\delta_0}}
    $$
which means, within the permissible range, we transfer the derivative of $\partial_\alpha^s\zeta_t$ with respect to $t$ onto $\zeta_\alpha-1$. Applying Corollary \ref{cor_integral_three} and Lemma \ref{lem_int_alpha_away_t} with $f_1=D_t\zeta,\  f_2=\partial_\alpha^s\zeta,\ f_3=\partial_\alpha D_t\zeta$ yields
\begin{equation}\label{mathfrak_G_11}
    \norm{\mathfrak{G}_{11}}_{L^2}\lesssim\frac{\epsilon^3}{t^{1+\delta}}.
    \end{equation}
Using the strategy mentioned above, we claim that
\begin{equation}\label{mathfrak_G_13}
    \norm{\mathfrak{G}_{13}}_{L^2}\lesssim\frac{\epsilon^3}{t^{1+\delta}}.
    \end{equation}
    
For $\mathfrak{G}_{12}$, more precise manipulations are required to be performed on the integral. Note that
$$
    \begin{aligned}
    \mathfrak{G}_{12}=&\ \frac{1}{\pi i}\int\frac{\zeta_t(\alpha,t)-\zeta_t(\beta,t)}{(\alpha-\beta)^2}[-\partial_t((\zeta(\alpha,t)-\alpha)-(\zeta(\beta,t)-\beta))\partial_\beta^{s+1}\zeta\\
    &+((\zeta(\alpha,t)-\alpha)-(\zeta(\beta,t)-\beta))\partial_\beta^{s+1}\zeta_t] d\beta\\
    &+\frac{1}{\pi i}\int\frac{(\zeta_t(\alpha,t)-\zeta_t(\beta,t))^2}{(\alpha-\beta)^2}\partial_\beta^{s+1}\zeta d\beta.
    \end{aligned}
    $$
For the first term we refer to \eqref{transfer_d_var_2}. For the second term, if $t^{\frac{4}{5}}\leq|\alpha|\leq t^{\frac{6}{5}}$:
$$
\begin{aligned}
&\frac{1}{\pi i}\int\frac{(\zeta_t(\alpha,t)-\zeta_t(\beta,t))^2}{(\alpha-\beta)^2}\partial_\beta^{s+1}\zeta d\beta\\
=\ &\frac{1}{\pi i}\int_{|\alpha-\beta|\geq\frac{1}{2}t^{\frac{4}{5}}}\frac{(\zeta_t(\alpha,t)-\zeta_t(\beta,t))^2}{(\alpha-\beta)^2}\partial_\beta^{s+1}\zeta d\beta\\
&+\frac{2}{\pi it}\int_{|\alpha-\beta|<\frac{1}{2}t^{\frac{4}{5}}}\frac{(\zeta_t(\alpha,t)-\zeta_t(\beta,t))(L_0(\zeta(\alpha,t)-\alpha)-L_0(\zeta(\beta,t)-\beta))}{(\alpha-\beta)^2}\partial_\beta^{s+1}\zeta d\beta\\
&-\frac{2}{\pi it}\int_{|\alpha-\beta|<\frac{1}{2}t^{\frac{4}{5}}}\frac{(\zeta_t(\alpha,t)-\zeta_t(\beta,t))(\alpha (\zeta_\alpha(\alpha,t)-1)-\beta (\zeta_\beta(\beta,t)-1))}{(\alpha-\beta)^2}\partial_\beta^{s+1}\zeta d\beta.
\end{aligned}
$$
Appealing to \eqref{transfer_d_var_2} and Remark \ref{rem_(I-H)fg_main} yields that
\begin{equation}\label{mathfrak_G_12}
    \norm{\mathfrak{G}_{12}}_{L^2}\lesssim\frac{\epsilon^3}{t^{1+\delta}}.
    \end{equation}
From \eqref{mathfrak_G_11}, \eqref{mathfrak_G_12} and \eqref{mathfrak_G_13}, we obtain (\ref{mathfrak_G_1}). Similarly,
$$
    \norm{\mathfrak{G}_2-1_{t^\frac{3}{4}\leq|\alpha|\leq t^\frac{5}{4}}2\partial_\alpha^s\zeta|\partial_t\zeta_\alpha|^2}_{L^2}\lesssim\frac{\epsilon^3}{t^{1+\delta}}.
    $$
Utilize (\ref{transfer_partial_t}) and Lemma \ref{lem_theta_t_zeta_t}:
$$
    \norm{\mathfrak{G}_2-1_{t^\frac{3}{4}\leq|\alpha|\leq t^\frac{5}{4}}\frac{it}{2\alpha}D_t\tilde{\theta}_s|\zeta_\alpha-1|^2}_{L^2}\lesssim\frac{\epsilon^3}{t^{1+\delta}}.
    $$
For $k< s$ the proof is easier and we omit it. Till now we have proved that
    
    \begin{equation}\label{G_c_main}
        \norm{\partial_\alpha^kG_c-C_k1_{t^\frac{3}{4}\leq|\alpha|\leq t^\frac{5}{4}}\frac{it}{2\alpha}D_t\tilde{\theta}_k|\zeta_\alpha-1|^2}_{L^2}\lesssim\frac{\epsilon^3}{t^{1+\delta}}
    \end{equation}
where $C_k$ is a positive number with respect to k. Recall Lemma \ref{lem_int_alpha_away_t}:
$$
    \left|\operatorname{Re}\int\frac{2D_t\overline{\tilde{\theta}_k}}{A}\partial_\alpha^kG_c\right|=\left|\operatorname{Re}\int\frac{2D_t\overline{\tilde{\theta}_k}}{A}\left(\partial_\alpha^kG_c-C_k1_{t^\frac{3}{4}\leq|\alpha|\leq t^\frac{5}{4}}\frac{it}{2\alpha}D_t\tilde{\theta}_k|\zeta_\alpha-1|^2\right)\right|\lesssim\frac{\epsilon^3}{t^{1+\delta}}\sqrt{\mathfrak{E_s}}.
    $$
    Analogously, there holds
$$
    \left|\operatorname{Re}\int\frac{2D_t\overline{\tilde{\sigma}_k}}{A}\partial_\alpha^kD_tG_c\right|\lesssim\frac{\epsilon^3}{t^{1+\delta}}\sqrt{\mathfrak{E_s}}.
    $$
All the above imply (\ref{d_t_E_s}).
\end{proof}
Now we are ready to prove the global boundedness of the energy.
\begin{cor}\label{cor_E}
    For $t\in [0,T]$, we claim that
$$
        \sqrt{\mathfrak{E}_s}\leq e^{C\epsilon^2+C\lambda^{\frac{1}{5}}\epsilon+Cc_s^3}(\sqrt{\mathfrak{E}_s(0)}+C\epsilon^3+C\lambda^{\frac{1}{7}}\epsilon+Cc_s\epsilon).
    $$
\end{cor}
\begin{proof}
    According to Theorem \ref{thm_d_t_E_s}, we deduce that
$$
\begin{aligned}
    \frac{d}{dt}\sqrt{\mathfrak{E}_s}\leq& C\left(\frac{\epsilon^3}{(1+t)^{1+\delta}}+\frac{\lambda\epsilon}{H^{\frac{3}{2}}(1+t)^{\frac{1}{4}}}+\frac{\lambda^2}{H^{\frac{3}{2}}}\right)+C\left(\frac{\epsilon^2}{(1+t)^{1+\delta}}+\frac{\lambda\epsilon}{H^{2}(1+t)^{\frac{1}{2}}}+\frac{\lambda}{H^{\frac{5}{2}}}\right)\sqrt{\mathfrak{E_s}}\\
    \leq&C\left(\frac{\epsilon^3}{(1+t)^{1+\delta}}+\frac{\lambda\epsilon}{(H_0+\lambda t)^{\frac{3}{2}}(1+t)^{\frac{1}{4}}}+\frac{\lambda^2}{(H_0+\lambda t)^{\frac{3}{2}}}\right)\\
    &+C\left(\frac{\epsilon^2}{(1+t)^{1+\delta}}+\frac{\lambda\epsilon}{(H_0+\lambda t)^2(1+t)^{\frac{1}{2}}}+\frac{\lambda}{(H_0+\lambda t)^{\frac{5}{2}}}\right)\sqrt{\mathfrak{E_s}}.
    \end{aligned}
    $$
Denote
$$
\Theta(t)=C\frac{\epsilon^2}{(1+t)^{1+\delta}}+C\frac{\lambda\epsilon}{(H_0+\lambda t)^2(1+t)^{\frac{1}{2}}}+C\frac{\lambda}{(H_0+\lambda t)^{\frac{5}{2}}},
$$
then we have
$$
    \frac{d}{dt}(e^{\int_0^t-\Theta(\tau)d\tau}\sqrt{\mathfrak{E}_s})\leq Ce^{\int_0^t-\Theta(\tau)d\tau}\left(\frac{\epsilon^3}{(1+t)^{1+\delta}}+\frac{\lambda\epsilon}{(H_0+\lambda t)^{\frac{3}{2}}(1+t)^{\frac{1}{4}}}+\frac{\lambda^2}{(H_0+\lambda t)^{\frac{3}{2}}}\right)
    $$
which implies that
$$
    \sqrt{\mathfrak{E}_s}\leq e^{\int_0^t\Theta(\tau)d\tau}\left[\sqrt{\mathfrak{E}_s(0)}+C\int_0^t\left(\frac{\epsilon^3}{(1+\tau)^{1+\delta}}+\frac{\lambda\epsilon}{(H_0+\lambda \tau)^{\frac{3}{2}}(1+\tau)^{\frac{1}{4}}}+\frac{\lambda^2}{(H_0+\lambda \tau)^{\frac{3}{2}}}\right)d\tau\right].
    $$
    Next we consider the integral
    $$
    \int_0^\infty\frac{\lambda}{(1+\lambda t)^p(1+t)^q}dt
    $$
    for $p+q>1$ and $p,q>0$. Obviously,
    $$
    \begin{aligned}
    \int_0^\infty\frac{\lambda}{(1+\lambda t)^p(1+t)^q}dt=&\int_{t\leq\tilde{\lambda}}\frac{\lambda}{(1+\lambda t)^p(1+t)^q}dt+\int_{t>\tilde{\lambda}}\frac{\lambda}{(1+\lambda t)^p(1+t)^q}dt\\
    \leq&\ \lambda\tilde{\lambda}+\lambda^{1-p}\int_{t>\tilde{\lambda}}\frac{1}{t^{p+q}}dt\\
    =&\ \lambda\tilde{\lambda}+\frac{1}{p+q-1}\lambda^{1-p}\tilde{\lambda}^{1-p-q}.
    \end{aligned}
    $$
    Let $\lambda^{-p}\tilde{\lambda}^{-p-q}=1$, i.e., $\tilde{\lambda}=\lambda^{\frac{-p}{p+q}}$. Then there holds
    $$
    \int_0^\infty\frac{\lambda}{(1+\lambda t)^p(1+t)^q}dt\leq\frac{p+q}{p+q-1}\lambda^{\frac{q}{p+q}}.
    $$
    From the inequality we see that
    $$
    \int_0^t\Theta(\tau)d\tau\leq C\epsilon^2+C\lambda^{\frac{1}{5}}\epsilon+C\frac{1}{H_0^{\frac{3}{2}}}.
    $$
    Meanwhile,
    $$
    \int_0^t\left(\frac{\epsilon^3}{(1+\tau)^{1+\delta}}+\frac{\lambda\epsilon}{(H_0+\lambda \tau)^{\frac{3}{2}}(1+\tau)^{\frac{1}{4}}}+\frac{\lambda^2}{(H_0+\lambda \tau)^{\frac{3}{2}}}\right)d\tau\leq\frac{\epsilon^3}{\delta}+\frac{7}{3}\lambda^{\frac{1}{7}}\epsilon+\frac{2\lambda}{H_0^{\frac{1}{2}}}.
    $$
    Since $H_0^{-\frac{1}{2}}\leq c_s$ and $\lambda H_0^{-\frac{1}{2}}\leq c_s\epsilon$, we deduce that
$$
\sqrt{\mathfrak{E}_s}\leq e^{C\epsilon^2+C\lambda^{\frac{1}{5}}\epsilon+Cc_s^3}(\sqrt{\mathfrak{E}_s(0)}+C\epsilon^3+C\lambda^{\frac{1}{7}}\epsilon+Cc_s\epsilon).
$$
\end{proof}
\subsection{Energy estimates for vector fields} In this subsection, we aim to close the bootstrap assumptions about $\norm{L_0\zeta_\alpha}_{H^{s-1}}$ and $\norm{L_0D_t\zeta}_{H^{s-1}}$. To begin with, we derive some essential equations with respect to vector fields. Applying $L_0$ to both sides of (\ref{theta}) and (\ref{sigma}) yields
$$
    (D^2_t-iA\partial_\alpha)L_0\tilde{\theta}=[D^2_t-iA\partial_\alpha,L_0]\tilde{\theta}+L_0G^{\tilde{\theta}}_0
$$
and
$$
    (D^2_t-iA\partial_\alpha)L_0\tilde{\sigma}=[D^2_t-iA\partial_\alpha,L_0]\tilde{\sigma}+L_0G^{\tilde{\sigma}}_0.
$$
Regarding higher-order derivatives, we deduce that
$$
\begin{aligned}
    (D^2_t-iA\partial_\alpha)\partial_\alpha^kL_0\tilde{\theta}&=[D^2_t-iA\partial_\alpha,\partial_\alpha^k]L_0\tilde{\theta}+\partial_\alpha^k[D^2_t-iA\partial_\alpha,L_0]\tilde{\theta}+\partial_\alpha^kL_0G^{\tilde{\theta}}_0=:G_k^{L_0\tilde{\theta}},\\
    (D^2_t-iA\partial_\alpha)\partial_\alpha^kL_0\tilde{\sigma}&=[D^2_t-iA\partial_\alpha,\partial_\alpha^k]L_0\tilde{\sigma}+\partial_\alpha^k[D^2_t-iA\partial_\alpha,L_0]\tilde{\sigma}+\partial_\alpha^kL_0G^{\tilde{\sigma}}_0=:G_k^{L_0\tilde{\sigma}}
\end{aligned}
$$
for $0\leq k\leq s-1$. Define the energy
$$
\begin{aligned}
    E^{L_0\tilde{\theta}}_k=&\int\frac{|D_t\partial_\alpha^k L_0\tilde{\theta}|^2}{A}+i\partial_\alpha^k L_0\tilde{\theta}\partial_{\beta}\overline{\partial_\alpha^k L_0\tilde{\theta}}d{\beta},\\
    E^{L_0\tilde{\sigma}}_k=&\int\frac{|D_t\partial_\alpha^k L_0\tilde{\sigma}|^2}{A}+i\partial_\alpha^k L_0\tilde{\sigma}\partial_{\beta}\overline{\partial_\alpha^k L_0\tilde{\sigma}}d{\beta}
\end{aligned}
$$
and the total energy
$$
\mathfrak{E}_s^{L_0}=\sum_{k=0}^{s-1}(E_k^{L_0\tilde{\theta}}+E_k^{L_0\tilde{\sigma}}).
$$
In analogy with (\ref{d_t_E}), there holds
\begin{equation}\label{d_tE^L_0}
\begin{aligned}
\frac{d}{dt}E^{L_0\tilde{\theta}}_k=&\ \operatorname{Re}\left\{\int\frac{2D_t\partial_\alpha^kL_0\bar{\tilde{\theta}}}{A}G^{L_0\tilde{\theta}}_kd{\beta}\right\}-\int\frac{|D_t\partial_\alpha^kL_0{\tilde{\theta}}|^2}{A}\left(\frac{a_t}{a}\circ\kappa^{-1}\right)d\beta,\\
\frac{d}{dt}E^{L_0\tilde{\sigma}}_k=&\ \operatorname{Re}\left\{\int\frac{2D_t\partial_\alpha^kL_0\bar{\tilde{\sigma}}}{A}G^{L_0\tilde{\sigma}}_kd{\beta}\right\}-\int\frac{|D_t\partial_\alpha^kL_0{\tilde{\sigma}}|^2}{A}\left(\frac{a_t}{a}\circ\kappa^{-1}\right)d\beta.
\end{aligned}
\end{equation}

The following lemma reveals the relation between the vector fields and the energy:
\begin{lemma}\label{lem_L_0D_t_zeta_E}
    For $t\in[0,T]$, we have
$$
    \norm{L_0D_t{\zeta}}_{H^{s-1}}+\norm{L_0D_t^2\zeta}_{H^{s-1}}+\norm{L_0\zeta_\alpha}_{H^{s-1}}\lesssim\sqrt{\mathfrak{E}_s^{L_0}}+\frac{\epsilon^2}{t^{\frac{1}{6}-\delta_0}}\ln (2+t)+\frac{(\epsilon+\lambda)^2(1+t)}{H^{\frac{5}{2}}}+\frac{\epsilon (1+t)^{\frac{1}{2}}}{H^{\frac{3}{2}}}.
    $$
\end{lemma}
\begin{proof}
    By Referring to Lemma \ref{lem_L_0D_t_theta}, the result follows:
    $$
    \norm{L_0D_t{\zeta}}_{H^{s-1}}+\norm{L_0D_t^2\zeta}_{H^{s-1}}\lesssim\sqrt{\mathfrak{E}_s^{L_0}}+\frac{\epsilon^2}{t^{\frac{1}{6}-\delta_0}}\ln (2+t)+\frac{\lambda(\epsilon+\lambda) t}{H^{\frac{5}{2}}}.
    $$
    For $L_0\zeta_\alpha$, we recall that
    \begin{equation}\label{L_0D_t^2_zeta}
    L_0D_t^2\zeta=iL_0A\zeta_\alpha+iAL_0\zeta_\alpha.
    \end{equation}
    We need to investigate $L_0A$:
    $$
    \begin{aligned}
    (I-\mathfrak{H})L_0A=&\ [L_0,\mathfrak{H}](A-1)+L_0(I-\mathfrak{H})(A-1)\\
    =&\ [L_0,\mathfrak{H}](A-1)+iL_0(I-\mathfrak{H})\left(D_t^2\zeta\frac{\bar{\zeta}_\alpha-1}{\zeta_\alpha}+D_t\zeta\frac{\partial_\alpha D_t\bar{\zeta}}{\zeta_\alpha}\right)\\
    &-\frac{\lambda}{\pi}L_0\left[\frac{\dot{z_2}}{(\zeta-z_2)^2}-\frac{\dot{z_1}}{(\zeta-z_1)^2}\right],
    \end{aligned}
    $$
    where
    $$
    L_0\left[\frac{\dot{z_2}}{(\zeta-z_2)^2}-\frac{\dot{z_1}}{(\zeta-z_1)^2}\right]=\frac{t\ddot{z_2}}{2(\zeta-z_2)^2}-\frac{t\ddot{z_1}}{2(\zeta-z_1)^2}-\frac{\dot{z_2}(2L_0\zeta-t\dot{z_2})}{(\zeta-z_2)^3}+\frac{\dot{z_1}(2L_0\zeta-t\dot{z_1})}{(\zeta-z_1)^3}.
    $$
    Therefore,
    $$
    \norm{L_0\left[\frac{\dot{z_2}}{(\zeta-z_2)^2}-\frac{\dot{z_1}}{(\zeta-z_1)^2}\right]}_{H^{s-1}}\lesssim\frac{(\epsilon+\lambda)^2t}{H^3}+\frac{\epsilon t^{\frac{1}{2}}}{H^{\frac{3}{2}}}.
    $$
    Also,
    $$
    \begin{aligned}
    \norm{[L_0,\mathfrak{H}](A-1)}_{H^{s-1}}\lesssim&\norm{A-1}_{H^{s-1}}+\norm{L_0\zeta_\alpha}_{H^{s-1}}\norm{A-1}_{W^{s-1,\infty}}\\
    \lesssim&\ \frac{\epsilon^2}{t^{\frac{2}{3}-\delta_0}}\ln (2+t)+\frac{\lambda(\epsilon+\lambda)}{H^{\frac{3}{2}}}+\epsilon^2\norm{L_0\zeta_\alpha}_{H^{s-1}}
    \end{aligned}
    $$
    and
    $$
    \begin{aligned}
    \norm{L_0(I-\mathfrak{H})\left(D_t^2\zeta\frac{\bar{\zeta}_\alpha-1}{\zeta_\alpha}\right)}_{H^{s-1}}\lesssim&\norm{L_0D_t^2\zeta}_{H^{s-1}}\norm{\zeta_\alpha-1}_{W^{s-1,\infty}}+\norm{D_t^2\zeta}_{W^{s-1,\infty}}\norm{L_0\zeta_\alpha}_{H^{s-1}}\\
    \lesssim&\ \frac{\epsilon\sqrt{\mathfrak{E}_s^{L_0}}}{t^\frac{1}{6}}\ln (2+t)+\frac{\epsilon^2}{t^{\frac{1}{6}-\delta_0}}\ln (2+t)+\frac{\lambda\epsilon(\epsilon+\lambda) t^{\frac{5}{6}}}{H^{\frac{5}{2}}}\ln (2+t).
    \end{aligned}
    $$
    To handle the rest part of $(I-\mathfrak{H})L_0A$, we consider the most extreme scenario where $L_0$ and all the derivatives are performed on $\partial_\alpha D_t\bar\zeta$ in $(I-\mathfrak{H})\left(D_t\zeta\frac{\partial_\alpha D_t\bar\zeta}{\zeta_\alpha}\right)$. In other words, we now proceed to examine the function
    $$
    (I-\mathfrak{H})(D_t\zeta\partial_\alpha^sL_0D_t\bar\zeta)=[D_t\zeta,\mathfrak{H}]\partial_\alpha^sL_0D_t\bar\zeta+D_t\zeta[\partial_\alpha^s,\mathfrak{H}]L_0D_t\bar\zeta+D_t\zeta\partial_\alpha^s[L_0,\mathfrak{H}]D_t\bar\zeta.
    $$
    By the utilization of integration by part, we deduce that
    $$
    \norm{(I-\mathfrak{H})(D_t\zeta\partial_\alpha^sL_0D_t\bar\zeta)}_{L^2}\lesssim\frac{\epsilon}{t^\frac{1}{2}}\left(\sqrt{\mathfrak{E}_s^{L_0}}+\frac{\epsilon^2}{t^{\frac{1}{6}-\delta_0}}\ln (2+t)+\frac{\lambda(\epsilon+\lambda) t}{H^{\frac{5}{2}}}\right)
    $$
    which, combined with the estimates above, has proved that
    $$
    \norm{L_0A}_{H^{s-1}}\lesssim\frac{\epsilon^2}{t^{\frac{1}{6}-\delta_0}}\ln (2+t)+\frac{\epsilon}{t^\frac{1}{6}}\ln (2+t)\sqrt{\mathfrak{E}_s^{L_0}}+\epsilon^2\norm{L_0\zeta_\alpha}_{H^{s-1}}+\frac{(\epsilon+\lambda)^2t}{H^{\frac{5}{2}}}+\frac{\epsilon t^{\frac{1}{2}}}{H^{\frac{3}{2}}}.
    $$
    From (\ref{L_0D_t^2_zeta}), we obtain
    $$
    \begin{aligned}
    \norm{L_0\zeta_\alpha}_{H^{s-1}}\leq&\ \norm{A-1}_{H^{s-1}}\norm{L_0\zeta_\alpha}_{H^{s-1}}+\norm{L_0D_t^2\zeta}_{H^{s-1}}\\
    &+\norm{L_0A}_{H^{s-1}}\norm{\zeta_\alpha-1}_{H^{s-1}}+\norm{L_0A}_{H^{s-1}}\\
    \lesssim&\ \frac{\epsilon^2}{t^{\frac{1}{6}-\delta_0}}\ln (2+t)+\sqrt{\mathfrak{E}_s^{L_0}}+\epsilon^2\norm{L_0\zeta_\alpha}_{H^{s-1}}+\frac{(\epsilon+\lambda)^2t}{H^{\frac{5}{2}}}+\frac{\epsilon t^{\frac{1}{2}}}{H^{\frac{3}{2}}}.
    \end{aligned}
    $$
    Since $\epsilon$ is small enough, the lemma follows.
\end{proof}
It is necessary to establish several preliminary facts.
\begin{lemma}\label{lem_[P,d_alpha]L_0}
    For $0\leq k\leq s-1$, there holds
$$
    \norm{[D^2_t-iA\partial_\alpha,\partial_\alpha^k]L_0\tilde{\theta}}_{L^2}+\norm{[D^2_t-iA\partial_\alpha,\partial_\alpha^k]L_0\tilde{\sigma}}_{L^2}\lesssim\frac{\epsilon^3}{t^{1+\delta}}+\frac{\lambda\epsilon(\epsilon+\lambda) t^{\delta_0}}{H^2}+\frac{\lambda\epsilon t^{\delta_0}}{H^3}.
    $$
\end{lemma}
\begin{proof}
    Imitating the proof of Lemma \ref{lem_[P,partial_alpha]}, we claim that
    $$\begin{aligned}
        \norm{[D^2_t-iA\partial_\alpha,\partial_\alpha^k]L_0\tilde{\theta}}_{L^2}\lesssim&\ (\norm{b_\alpha}_{W^{s-2,\infty}}+\norm{D_tb}_{W^{s-1,\infty}})\left(\norm{\partial_\alpha L_0\tilde{\theta}}_{H^{s-2}}+\norm{\partial_\alpha D_tL_0\tilde{\theta}}_{H^{s-2}}\right)\\
        &+\norm{A-1}_{W^{s-1,\infty}}\norm{\partial_\alpha L_0\tilde{\theta}}_{H^{s-2}}\\
        \lesssim&\ (\norm{b_\alpha}_{W^{s-2,\infty}}+\norm{D_tb}_{W^{s-1,\infty}})\left(\norm{L_0\tilde{\theta}_\alpha}_{H^{s-2}}+\norm{\partial_\alpha L_0D_t\tilde{\theta}}_{H^{s-2}}\right)\\
        &+(\norm{b_\alpha}_{W^{s-2,\infty}}+\norm{D_tb}_{W^{s-1,\infty}})\left(\norm{[L_0,\partial_\alpha]\tilde{\theta}}_{H^{s-2}}+\norm{\partial_\alpha [L_0,D_t]\tilde{\theta}}_{H^{s-2}}\right)\\
        &+\norm{A-1}_{W^{s-1,\infty}}\norm{L_0\tilde{\theta}_\alpha}_{H^{s-2}}+\norm{A-1}_{W^{s-1,\infty}}\norm{[L_0,\partial_\alpha]\tilde{\theta}}_{H^{s-2}}\\
        \lesssim&\ \frac{\epsilon^3}{t^{1+\delta}}+\frac{\lambda\epsilon(\epsilon+\lambda) t^{\delta_0}}{H^2}+\frac{\lambda\epsilon t^{\delta_0}}{H^3}.
        \end{aligned}
    $$
Likewise,
$$
    \norm{[D^2_t-iA\partial_\alpha,\partial_\alpha^k]L_0\tilde{\sigma}}_{L^2}\lesssim\frac{\epsilon^3}{t^{1+\delta}}+\frac{\lambda\epsilon(\epsilon+\lambda) t^{\delta_0}}{H^2}+\frac{\lambda\epsilon t^{\delta_0}}{H^3}.
    $$
\end{proof}
\begin{lemma}\label{lem_[P,L_0]}
    For $0\leq k\leq s-1$, we have
    $$
    \norm{\partial_\alpha^k[D^2_t-iA\partial_\alpha,L_0]\tilde{\theta}}_{L^2}+\norm{\partial_\alpha^k[D^2_t-iA\partial_\alpha,L_0]\tilde{\sigma}}_{L^2}\lesssim\left(\frac{\epsilon^2}{t}+\frac{\lambda}{H^2}\right)\sqrt{\mathfrak{E}_s^{L_0}}+\frac{\epsilon^3}{t}+\frac{\lambda\epsilon}{H^2}.
    $$
\end{lemma}
\begin{proof}
    A straightforward calculation indicates that
\begin{equation}\label{[P,L_0]}
    \begin{aligned}
    \partial_\alpha^k[D^2_t-iA\partial_\alpha,L_0]\tilde{\theta}=&\ \partial_\alpha^k(D_t[D_t,L_0]+[D_t,L_0]D_t-i[A\partial_\alpha,L_0])\tilde{\theta}\\
    =&\ \partial_\alpha^k(D_t(\frac{1}{2}D_t+\frac{1}{2}b\partial_\alpha-L_0b\partial_\alpha)+(\frac{1}{2}D_t+\frac{1}{2}b\partial_\alpha-L_0b\partial_\alpha)D_t)\tilde{\theta}\\
    &-\partial_\alpha^k(iA\partial_\alpha-iL_0A\partial_\alpha)\tilde{\theta}\\
    =&\ \partial_\alpha^kG_0^{\tilde{\theta}}+\frac{1}{2}\partial_\alpha^k(D_t(b\tilde{\theta}_\alpha)+b\partial_\alpha D_t\tilde{\theta})\\
    &-\partial_\alpha^k(D_t(L_0b\tilde{\theta}_\alpha)+L_0b\partial_\alpha D_t\tilde{\theta}-iL_0A\tilde{\theta}_\alpha).
    \end{aligned}
    \end{equation}
Recall Lemma \ref{lem_b}:
$$
    \norm{\partial_\alpha^k(D_t(b\tilde{\theta}_\alpha)+b\partial_\alpha D_t\tilde{\theta})}_{L^2}\lesssim\frac{\epsilon^3}{t}+\frac{\lambda\epsilon}{H^2}.
    $$
From (\ref{G_c_main}) we deduce that
$$
    \norm{\partial_\alpha^kG_c-C_k1_{t^\frac{3}{4}\leq|\alpha|\leq t^\frac{5}{4}}D_t\tilde{\theta}_{k+1}|\zeta_\alpha-1|^2}_{L^2}\lesssim\frac{\epsilon^3}{t^{1+\delta}}.
    $$
Hence
$$
    \norm{\partial_\alpha^kG_c}_{L^2}\lesssim\frac{\epsilon^3}{t}.
    $$
Next we investigate $\partial_\alpha^{s-1}D_t(L_0b\tilde{\theta}_\alpha)$:
$$
    \begin{aligned}
    \partial_\alpha^{s-1} D_t(L_0b\tilde{\theta}_\alpha)=&\ \partial_\alpha^{s-1}D_tL_0b\tilde{\theta}_\alpha+\partial_\alpha^{s-1}L_0bD_t\tilde{\theta}_\alpha+D_tL_0b\partial_\alpha^s\tilde{\theta}+L_0b\partial_\alpha^{s-1}D_t\tilde{\theta}_\alpha\\
    &+\sum_{j=1}^{s-3}\binom{s-1}{j}\partial_\alpha^jD_tL_0b\partial_\alpha^{s-j-1}\tilde{\theta}_\alpha\\
    &+\sum_{j=2}^{s-2}\binom{s-1}{j}\partial_\alpha^jL_0b\partial_\alpha^{s-j-1}D_t\tilde{\theta}_\alpha\\
    &+(s-1)\partial_\alpha L_0b\partial_\alpha^{s-2}D_t\tilde{\theta}_\alpha+(s-1)\partial_\alpha^{s-2}D_tL_0b\partial_\alpha^2\tilde{\theta}.
    \end{aligned}
    $$ 
To estimate $L_0b$, we refer to Lemma \ref{lem_real_proj} and Lemma \ref{lem_L_0D_t_zeta_E}:
\begin{equation}\label{L_0_b}
    \begin{aligned}
    \norm{L_0b}_{H^{s-2}}\lesssim\ &\norm{(I-\mathfrak{H})L_0b}_{H^{s-2}}\\
    \leq\ &\norm{[L_0,\mathfrak{H}]b}_{H^{s-2}}+\norm{L_0(I-\mathfrak{H})b}_{H^{s-2}}\\
    \lesssim\ &\norm{\partial_\alpha L_0\zeta}_{H^{s-2}}\norm{b}_{W^{s-2,\infty}}+\norm{b}_{H^{s-2}}+\norm{L_0D_t\zeta}_{H^{s-2}}\norm{\zeta_\alpha-1}_{W^{s-2,\infty}}\\
    &+\norm{L_0\zeta_\alpha}_{H^{s-2}}\norm{\partial_\alpha D_t\zeta}_{W^{s-3,\infty}}+\norm{L_0q}_{H^{s-2}}\\
    \lesssim\ &\left(\frac{\epsilon}{t^\frac{1}{2}}+\frac{\lambda}{H^2}\right)\sqrt{\mathfrak{E}_s^{L_0}}+\frac{\epsilon^2}{t^\frac{1}{2}}+\frac{\lambda}{H^{\frac{3}{2}}}.
    \end{aligned}
\end{equation}
Similarly,
$$
\norm{D_tL_0b}_{H^{s-3}}\lesssim\left(\frac{\epsilon}{t^\frac{1}{2}}+\frac{\lambda}{H^2}\right)\sqrt{\mathfrak{E}_s^{L_0}}+\frac{\epsilon^2}{t^\frac{1}{2}}+\frac{\lambda}{H^{\frac{3}{2}}}.
$$
For $\partial_\alpha^{s-1}D_tL_0b\tilde{\theta}_\alpha$, more delicate manipulations are needed:
$$\begin{aligned}
\partial_\alpha^{s-1}D_tL_0b\tilde{\theta}_\alpha=&\ \frac{1}{2}\partial_\alpha^{s-1}D_tL_0(I-\mathfrak{H})b\tilde{\theta}_\alpha+\frac{1}{2}\partial_\alpha^{s-1}D_tL_0\overline{(I-\mathfrak{H})b}\tilde{\theta}_\alpha+\frac{1}{2}\partial_\alpha^{s-1}D_tL_0(\mathfrak{H}+\bar{\mathfrak{H}})b\tilde{\theta}_\alpha\\
=&\ K_1+K_2+K_3.
\end{aligned}$$
For $K_1$, there holds
$$
K_1=\frac{1}{2}\partial_\alpha^{s-1}D_tL_0\left\{-(I-\mathfrak{H})D_t\zeta\frac{\bar{\zeta}_\alpha-1}{\zeta_\alpha}+2q\right\}\tilde{\theta}_\alpha.
$$
To handle $K_1$, we consider three functions which represent the most extreme scenarios: 
$$
\begin{array}{c}
f_1=(I-\mathfrak{H})(\partial_\alpha^{s-1}D_t^2\zeta L_0\bar\zeta_\alpha)(\zeta_\alpha-1),\\
f_2=(I-\mathfrak{H})(D_t\zeta\partial_\alpha^sL_0D_t\bar\zeta)(\zeta_\alpha-1),\\
f_3=(I-\mathfrak{H})(L_0D_t\zeta\partial_\alpha^sD_t\bar\zeta)(\zeta_\alpha-1).
\end{array}
$$ The technique developed here is capable of handling other cases as well. Firstly,
$$
\begin{aligned}
f_1=&\ [L_0\bar\zeta_\alpha,\mathfrak{H}]\partial_\alpha^{s-1}D_t^2\zeta(\zeta_\alpha-1)+L_0\bar\zeta_\alpha[\partial_\alpha^{s-1}D_t^2\zeta(\zeta_\alpha-1)]-L_0\bar\zeta_\alpha(\mathfrak{H}\partial_\alpha^{s-1}D_t^2\zeta)(\zeta_\alpha-1)\\
=&\ [L_0\bar\zeta_\alpha,\mathfrak{H}]\partial_\alpha^{s-1}D_t^2\zeta(\zeta_\alpha-1)+L_0\bar\zeta_\alpha[\partial_\alpha^{s-1}D_t^2\zeta(\zeta_\alpha-1)]-L_0\bar\zeta_\alpha(\mathfrak{H}+\bar{\mathfrak{H}})\partial_\alpha^{s-1}D_t^2\zeta(\zeta_\alpha-1)\\
&+L_0\bar\zeta_\alpha(\bar{\mathfrak{H}}\partial_\alpha^{s-1}D_t^2\zeta)(\zeta_\alpha-1)\\
=&\ [L_0\bar\zeta_\alpha,\mathfrak{H}]\partial_\alpha^{s-1}D_t^2\zeta(\zeta_\alpha-1)+2L_0\bar\zeta_\alpha[\partial_\alpha^{s-1}D_t^2\zeta(\zeta_\alpha-1)]-L_0\bar\zeta_\alpha(\mathfrak{H}+\bar{\mathfrak{H}})\partial_\alpha^{s-1}D_t^2\zeta(\zeta_\alpha-1)\\
&+L_0\bar\zeta_\alpha[\bar{\mathfrak{H}},\partial_\alpha^{s-1}D_t]D_t\zeta(\zeta_\alpha-1).
\end{aligned}
$$
With the application of Lemma \ref{app_S1} and Lemma \ref{lem_transfer_d}, we establish that
$$
\norm{f_1}_{L^2}\lesssim\frac{\epsilon^2}{t}\sqrt{\mathfrak{E}_s^{L_0}}+\frac{\epsilon^3}{t^{1+\delta}}.
$$
For $f_2$,
$$
f_2=([D_t\zeta,\mathfrak{H}]\partial_\alpha^sL_0D_t\bar\zeta)(\zeta_\alpha-1)+D_t\zeta([\partial_\alpha^sL_0,\mathfrak{H}]D_t\bar\zeta)(\zeta_\alpha-1).
$$
Then
$$
\norm{f_2}_{L^2}\lesssim\norm{L_0D_t\zeta}_{H^{s-1}}\norm{\partial_\alpha D_t\zeta}_{W^{2,\infty}}\norm{\zeta_\alpha-1}_{L^\infty}\lesssim\frac{\epsilon^2}{t}\sqrt{\mathfrak{E}_s^{L_0}}+\frac{\epsilon^3}{t^{1+\delta}}.
$$
Imitating the estimate of $f_1$, we get that
$$
\norm{f_3}_{L^2}\lesssim\frac{\epsilon^2}{t}\sqrt{\mathfrak{E}_s^{L_0}}+\frac{\epsilon^3}{t^{1+\delta}}.
$$
After treating some commutators which possess better time-decay and recalling Proposition \ref{pro_L_0_q}, we claim that
$$
\norm{K_1}_{L^2}\lesssim\left(\frac{\epsilon^2}{t}+\frac{\lambda}{H^2}\right)\sqrt{\mathfrak{E}_s^{L_0}}+\frac{\epsilon^3}{t^{1+\delta}}+\frac{\lambda\epsilon(\epsilon+\lambda)t^{\frac{1}{2}}}{H^{\frac{5}{2}}}.
$$
Similarly,
$$
\norm{K_2}_{L^2}\lesssim\left(\frac{\epsilon^2}{t}+\frac{\lambda}{H^2}\right)\sqrt{\mathfrak{E}_s^{L_0}}+\frac{\epsilon^3}{t^{1+\delta}}+\frac{\lambda\epsilon(\epsilon+\lambda)t^{\frac{1}{2}}}{H^{\frac{5}{2}}}.
$$
Reviewing the estimate of $K_1$, we can follow similar steps to estimate $\norm{D_tL_0b}_{H^{s-1}}$, which means that
\begin{equation}\label{D_t_L_0_b}
\norm{D_tL_0b}_{H^{s-1}}\lesssim\frac{\epsilon}{t^\frac{1}{6}}\ln (2+t)\sqrt{\mathfrak{E}_s^{L_0}}+\frac{\lambda(\epsilon+\lambda)t}{H^3
    }+\frac{\lambda(\epsilon+\lambda) t^{\frac{1}{2}}}{H^{\frac{5}{2}}}.
\end{equation}
For $K_3$, we lack control of $\norm{\partial_\alpha^sL_0D_t\zeta}_{L^2}$. Notice that
$$
\begin{aligned}
K_3=&\ \frac{1}{2}[\partial_\alpha^{s-1},D_tL_0](\mathfrak{H}+\bar{\mathfrak{H}})b\tilde{\theta}_\alpha+\frac{1}{2}D_tL_0\partial_\alpha^{s-1}\left(\frac{1}{\pi i}\int\partial_\beta\left(\ln \frac{\bar\zeta(\alpha,t)-\bar\zeta(\beta,t)}{\zeta(\alpha,t)-\zeta(\beta,t)}\right)bd\beta\right)\tilde{\theta}_\alpha\\
=&\ \frac{1}{2}[\partial_\alpha^{s-1},D_tL_0](\mathfrak{H}+\bar{\mathfrak{H}})b\tilde{\theta}_\alpha+\frac{1}{2}D_tL_0\partial_\alpha^{s-2}\left(\frac{1}{\pi i}\int\partial_\alpha\left(\ln\frac{\zeta(\alpha,t)-\zeta(\beta,t)}{\bar\zeta(\alpha,t)-\bar\zeta(\beta,t)}\right)b_\beta d\beta\right)\tilde{\theta}_\alpha\\
=&\ \frac{1}{2}[\partial_\alpha^{s-1},D_tL_0](\mathfrak{H}+\bar{\mathfrak{H}})b\tilde{\theta}_\alpha+\frac{1}{2}D_tL_0\partial_\alpha^{s-2}[\zeta_\alpha\mathfrak{H}(\zeta_\alpha b_\alpha)+\bar\zeta_\alpha\bar{\mathfrak{H}}(\bar\zeta_\alpha b_\alpha)]\tilde{\theta}_\alpha.
\end{aligned}
$$
Thus, thanks to (\ref{L_0_b}) and (\ref{D_t_L_0_b}),
$$
\begin{aligned}
\norm{K_3}_{L^2}\lesssim&\ (\norm{D_t\zeta}_{W^{s-2,\infty}}\norm{L_0b}_{H^{s-1}}+\norm{L_0D_t\zeta}_{H^{s-1}}\norm{b}_{W^{s-2,\infty}}\\&+\norm{L_0\zeta_\alpha}_{H^{s-1}}\norm{D_tb}_{W^{s-3,\infty}}+\norm{\zeta_\alpha-1}_{W^{s-2,\infty}}\norm{D_tL_0b}_{H^{s-1}})\norm{\tilde{\theta}_\alpha}_{L^\infty}\\
\lesssim&\left(\frac{\epsilon^3}{t^{\frac{7}{6}}}\ln (2+t)+\frac{\lambda\epsilon}{H^2t^{\frac{1}{2}}}\right)\sqrt{\mathfrak{E}_s^{L_0}}+\frac{\epsilon^4}{t^{\frac{3}{2}}}+\frac{\lambda\epsilon^2}{H^{\frac{3}{2}}t}+\frac{\lambda(\epsilon+\lambda)\epsilon^2}{H^3}+\frac{\lambda(\epsilon+\lambda)\epsilon^2}{H^{\frac{5}{2}}t^\frac{1}{2}}.
\end{aligned}
$$
In summary, the above implies that
$$
\norm{\partial_\alpha^{s-1}D_tL_0b\tilde{\theta}_\alpha}_{L^2}\lesssim\left(\frac{\epsilon^2}{t}+\frac{\lambda}{H^2}\right)\sqrt{\mathfrak{E}_s^{L_0}}+\frac{\epsilon^3}{t^{1+\delta}}+\frac{\lambda\epsilon(\epsilon+\lambda)t^{\frac{1}{2}}}{H^{\frac{5}{2}}}.
$$
For the other terms, the argument proceeds similarly to the previous one. More precisely, we split b into $\frac{1}{2}(I-\mathfrak{H})b$ and $\frac{1}{2}(I+\mathfrak{H})b$, then transfer the derivatives via Lemma \ref{lem_transfer_d} if necessary to gain enough time-decay. Till now we have deduced that
$$
\norm{\partial_\alpha^{s-1}D_t(L_0b\tilde{\theta}_\alpha)}_{L^2}\lesssim\left(\frac{\epsilon^2}{t}+\frac{\lambda}{H^2}\right)\sqrt{\mathfrak{E}_s^{L_0}}+\frac{\epsilon^3}{t^{1+\delta}}+\frac{\lambda\epsilon(\epsilon+\lambda)t^{\frac{1}{2}}}{H^{\frac{5}{2}}}.
$$

The remaining terms in (\ref{[P,L_0]}) can be treated by an analogous argument. The situation where $\tilde{\theta}$ is replaced by $\tilde{\sigma}$ is handled in the same way. Therefore the proof is completed.
\end{proof}

To close the bootstrap assumptions when $H_0\sim 1$, we need to provide a more delicate estimate of $D_tL_0q$ and $D_t^2L_0q$. The direct application of Lemma \ref{pro_L_0_q} requires stronger assumptions, e.g., $H_0\geq\epsilon^{-2}$.

\begin{lemma}
    For $t\in[0,T]$, there holds
    $$
    \norm{D_tL_0q}_{H^{s-1}}+\norm{D_t^2L_0q}_{H^{s-1}}\lesssim\frac{\lambda}{H^2}\sqrt{\mathfrak{E}^{L_0}_s}+\frac{\lambda(\epsilon+\lambda)^2(1+t)}{H^{3}}+\frac{\lambda(\epsilon+\lambda)}{H^2}+\frac{\lambda\epsilon(\epsilon+\lambda)(1+t)^{\frac{1}{2}+\delta_0}}{H^2}.
    $$
\end{lemma}
\begin{proof}
We only prove the situation where $s=1$ since other cases can be treated similarly. Recall that
    $$
    \begin{aligned}
    D_tL_0q=&-\frac{\lambda i}{4\pi}\frac{\dot{z_1}-\dot{z_2}+t(\ddot{z_1}-\ddot{z_2})}{(\zeta-z_1)(\zeta-z_2)}+\frac{\lambda it}{4\pi}(\dot{z_1}-\dot{z_2})\left(\frac{D_t\zeta-\dot{z_1}}{(\zeta-z_1)^2(\zeta-z_2)}+\frac{D_t\zeta-\dot{z_2}}{(\zeta-z_1)(\zeta-z_2)^2}\right)\\
    &+\frac{\lambda i}{2\pi}(\dot{z_1}-\dot{z_2})\left(\frac{L_0\zeta-\frac{1}{2}t\dot{z_1}}{(\zeta-z_1)^2(\zeta-z_2)}+\frac{L_0\zeta-\frac{1}{2}t\dot{z_2}}{(\zeta-z_1)(\zeta-z_2)^2}\right)\\
    &+\frac{\lambda i}{2\pi}(z_1-z_2)D_t\left(\frac{L_0\zeta-\frac{1}{2}t\dot{z_1}}{(\zeta-z_1)^2(\zeta-z_2)}+\frac{L_0\zeta-\frac{1}{2}t\dot{z_2}}{(\zeta-z_1)(\zeta-z_2)^2}\right).
    \end{aligned}
    $$
    We need to treat the terms $t\ddot{z_1}$ and $t\ddot{z_2}$ with caution. Observe the equation
    $$
    \begin{aligned}
    t\ddot{z_1}=&\ t\left(\frac{\lambda i(\bar{\dot{z_2}}-\bar{\dot{z_1}})}{2\pi(\bar{z_1}-\bar{z_2})^2}+\overline{\mathfrak{U}_z}(z_1,t)\bar{\dot{z_1}}+\bar{\mathfrak{U}}_t(z_1,t)\right)\\
    =&\ t\left(\frac{\lambda i(\bar{\dot{z_2}}-\bar{\dot{z_1}})}{2\pi(\bar{z_1}-\bar{z_2})^2}+\overline{\mathfrak{U}_z}(z_1,t)\bar{\dot{z_1}}\right)+\overline{(t{\mathfrak{U}}_t(z_1,t)+2z_1\mathfrak{U}_z(z_1,t))}-2\bar{z_1}\overline{\mathfrak{U}_z}(z_1,t).
    \end{aligned}
    $$
    To handle the second term on the right side, we restrict the function on $\Sigma(t)$ and denote
    $$
    \tilde{L}\mathfrak{F}=\left.\left(\frac{1}{2}t{\mathfrak{U}}_t(z,t)+z\mathfrak{U}_z(z,t)\right)\right|_{z\in\Sigma(t)}=\frac{1}{2}t\left(\partial_t-\frac{\zeta_t}{\zeta_\alpha}\partial_\alpha\right)\mathfrak{F}+\frac{\zeta}{\zeta_\alpha}\mathfrak{F}_\alpha.
    $$
    Since $\mathfrak{U}$ is holomorphic, $\tilde{L}\mathfrak{F}$ is the boundary value of a holomorphic function defined on $\Omega(t)$. From the definition of $\tilde{L}$ we verify that $[\tilde{L},\mathfrak{H}]=0$. Then we manage to estimate the value $\norm{\tilde{L}\mathfrak{F}}_{L^2}$. Direct calculation implies that
    $$
    \tilde{L}\mathfrak{F}=\tilde{L}\frac{I+\mathfrak{H}}{2}D_t\bar \zeta=\frac{I+\mathfrak{H}}{2}(\tilde{L}-L_0+L_0)D_t\bar\zeta=\frac{I+\mathfrak{H}}{2}L_0D_t\bar\zeta+\frac{I+\mathfrak{H}}{2}\left[\frac{(-L_0+I)(\zeta-\alpha)}{\zeta_\alpha}\partial_\alpha D_t\bar \zeta\right].
    $$
    Note that
    $$
    \partial_t(I-L_0)(\zeta-\alpha)=-\frac{1}{2}\zeta_t+(I-L_0)(D_t\zeta-b\zeta_\alpha)
    $$
    which implies
    $$
    \begin{aligned}
    &\partial_t\norm{(I-L_0)(\zeta-\alpha)}_{L^2}^2\\
    \lesssim\ &\norm{(I-L_0)(\zeta-\alpha)}_{L^2}(\norm{D_t\zeta}_{L^2}+\norm{L_0D_t\zeta}_{L^2}+\norm{(I-L_0)b}_{L^2}+\norm{b}_{L^\infty}\norm{L_0(\zeta_\alpha-1)}_{L^2}).
    \end{aligned}
    $$
    Consider (\ref{L_0_b}) and we obtain
    $$
    \partial_t\norm{(I-L_0)(\zeta-\alpha)}_{L^2}^2\lesssim(\epsilon+\lambda) t^{\delta_0}\norm{(I-L_0)(\zeta-\alpha)}_{L^2}.
    $$
    The integration with respect to $t$ suggests that
    $$
    \norm{(I-L_0)(\zeta-\alpha)}_{L^2}\lesssim(\epsilon+\lambda) t^{1+\delta_0}.
    $$
    Therefore,
    $$
    \begin{aligned}
    \norm{\tilde{L}\mathfrak{F}}_{L^2}\lesssim&\norm{L_0D_t\zeta}_{L^2}+\norm{(I-L_0)(\zeta-\alpha)}_{L^2}\norm{\partial_\alpha D_t\bar \zeta}_{L^\infty})\\
    \lesssim&\ \sqrt{\mathfrak{E}^{L_0}_s}+\epsilon(\epsilon+\lambda)t^{\frac{1}{2}+\delta_0}.
    \end{aligned}
    $$
    Cauchy's integral formula shows that
    $$
    |t{\mathfrak{U}}_t(z_1,t)+2z_1\mathfrak{U}_z(z_1,t)|\lesssim\frac{\norm{\tilde{L}\mathfrak{F}}_{L^2}}{H^{\frac{1}{2}}}\lesssim \frac{1}{H^{\frac{1}{2}}}( \sqrt{\mathfrak{E}^{L_0}_s}+\epsilon(\epsilon+\lambda)t^{\frac{1}{2}+\delta_0}).
    $$
    and
    $$
    \begin{aligned}
    |t\ddot{z_1}|\lesssim &\ \lambda t|\dot{z_1}-\dot{z_2}|+|\dot{z_1}||\mathfrak{U}_z(z_1,t)|+|t{\mathfrak{U}}_t(z_1,t)+2z_1\mathfrak{U}_z(z_1,t)|+|z_1||\mathfrak{U}_z(z_1,t)|\\
    \lesssim&\ [H_0+(\epsilon+\lambda)t]\frac{\epsilon+\lambda}{H^{\frac{3}{2}}}+\frac{1}{H^{\frac{1}{2}}}( \sqrt{\mathfrak{E}^{L_0}_s}+\epsilon(\epsilon+\lambda)t^{\frac{1}{2}+\delta_0}).
    \end{aligned}
    $$
    (For $|\mathfrak{U}_z(z_1,t)|$ we refer to ({\ref{sup_U_z}})) With respect to $|t\ddot{z_2}|$ we can obtain the same estimate. Return to (\ref{D_t_L_0_q}):
    $$
    \begin{aligned}
    \norm{D_tL_0q}_{L^2}\lesssim&\ \frac{\lambda}{H^{\frac{3}{2}}}|\dot{z_1}-\dot{z_2}|+\frac{\lambda}{H^{\frac{3}{2}}}(|t\ddot{z_1}|+|t\ddot{z_2}|)\\
    &+\frac{\lambda t}{H^{\frac{5}{2}}}|\dot{z_1}-\dot{z_2}|(\norm{D_t\zeta}_{L^\infty}+|\dot{z_1}|+|\dot{z_2}|)+\frac{\lambda}{H^{\frac{5}{2}}}\norm{D_tL_0\zeta}_{L^\infty}\\
    \lesssim&\ \frac{\lambda}{H^2}\sqrt{\mathfrak{E}^{L_0}_s}+\frac{\lambda(\epsilon+\lambda)^2t}{H^{3}}+\frac{\lambda(\epsilon+\lambda)}{H^2}+\frac{\lambda\epsilon(\epsilon+\lambda)t^{\frac{1}{2}+\delta_0}}{H^2}.
    \end{aligned}
    $$
    An analogous argument works for $\norm{D_t^2L_0q}_{L^2}$.
\end{proof}

\begin{lemma}\label{lem_L_0G_d}
    For $t\in[0,T]$, we have
    $$
    \norm{L_0G_d}_{H^{s-1}}+\norm{L_0D_tG_d}_{H^{s-1}}\lesssim\frac{\lambda}{H^2}\sqrt{\mathfrak{E}_s^{L_0}}+\frac{\lambda(\epsilon+\lambda)^2(1+t)}{H^{3}}+\frac{\lambda(\epsilon+\lambda)}{H^2}+\frac{\lambda\epsilon(\epsilon+\lambda)(1+t)^{\frac{1}{2}+\delta_0}}{H^2}.
    $$
\end{lemma}
\begin{proof}
    Recall that
    $$
    G_d=-2[\bar{q},\mathfrak{H}]\frac{\bar{\mathfrak{F}}_{\alpha}}{\zeta_{\alpha}}-2[\bar{\mathfrak{F}},\mathfrak{H}]\frac{\bar{q}_{\alpha}}{\zeta_{\alpha}}-2[\bar{q},\mathfrak{H}]\frac{\bar{q}_{\alpha}}{\zeta_{\alpha}}-4D_tq.
    $$
    and
    $$
    \mathfrak{F}=D_t\bar\zeta-q,
    $$
    we derive that
    $$
    \begin{aligned}
    \norm{L_0G_d}_{H^{s-1}}\lesssim&\norm{L_0q_\alpha}_{H^{s-2}}\norm{\mathfrak{F}}_{W^{s-1,\infty}}+\norm{q}_{W^{s-1,\infty}}\norm{L_0\mathfrak{F}_\alpha}_{H^{s-2}}\\
    &+\norm{L_0q_\alpha}_{H^{s-2}}\norm{q}_{W^{s-1,\infty}}+\norm{L_0D_tq}_{H^{s-1}}\\
    \lesssim&\ \frac{\lambda}{H^2}\sqrt{\mathfrak{E}_s^{L_0}}+\frac{\lambda(\epsilon+\lambda)^2t}{H^{3}}+\frac{\lambda(\epsilon+\lambda)}{H^2}+\frac{\lambda\epsilon(\epsilon+\lambda)t^{\frac{1}{2}+\delta_0}}{H^2}.
    \end{aligned}
    $$
    The same holds for $\norm{L_0D_tG_d}_{H^{s-1}}$.
\end{proof}

\begin{lemma}\label{lem_L_0G_0}
    For $t\in [0,T]$, we have
    $$
    \norm{L_0G_0^{\tilde{\theta}}}_{H^{s-1}}+\norm{L_0G_0^{\tilde{\sigma}}}_{H^{s-1}}\lesssim\left(\frac{\epsilon^2}{t}+\frac{\lambda}{H^2}\right)\sqrt{\mathfrak{E}_s^{L_0}}+\frac{\epsilon^3}{t^{1+\delta}}+\frac{\lambda(\epsilon+\lambda)^2t}{H^{3}}+\frac{\lambda(\epsilon+\lambda)}{H^2}+\frac{\lambda\epsilon(\epsilon+\lambda)t^{\frac{1}{2}+\delta_0}}{H^2}.
    $$
\end{lemma}
\begin{proof}
    Rewrite $G_c$ as:
    $$
    G_c=-2[D_t\zeta-\bar q,\mathfrak{H}\frac{1}{\zeta_\alpha}+\bar{\mathfrak{H}}\frac{1}{\bar\zeta_\alpha}](D_t\zeta-\bar q)_{\alpha}+\frac{1}{\pi i}\int\left(\frac{D_t\zeta(\alpha,t)-D_t\zeta(\beta,t)}{\zeta(\alpha,t)-\zeta(\beta,t)}\right)^2(\zeta-\bar\zeta)_{\beta}d\beta.
    $$
    Denote
    $$
    G_{c1}=2[D_t\zeta,\mathfrak{H}\frac{1}{\zeta_\alpha}+\bar{\mathfrak{H}}\frac{1}{\bar\zeta_\alpha}]\bar q_\alpha+2[\bar q,\mathfrak{H}\frac{1}{\zeta_\alpha}+\bar{\mathfrak{H}}\frac{1}{\bar\zeta_\alpha}]\partial_\alpha D_t\zeta-2[\bar q,\mathfrak{H}\frac{1}{\zeta_\alpha}+\bar{\mathfrak{H}}\frac{1}{\bar\zeta_\alpha}]\bar q_\alpha
    $$
    and
    $$
    G_{c2}=-2[D_t\zeta,\mathfrak{H}\frac{1}{\zeta_\alpha}+\bar{\mathfrak{H}}\frac{1}{\bar\zeta_\alpha}]\partial_\alpha D_t\zeta+\frac{1}{\pi i}\int\left(\frac{D_t\zeta(\alpha,t)-D_t\zeta(\beta,t)}{\zeta(\alpha,t)-\zeta(\beta,t)}\right)^2(\zeta-\bar\zeta)_{\beta}d\beta
    $$
    so that
    $$
    G_c=G_{c1}+G_{c2}.
    $$
    There holds
    \begin{equation}\label{L_0G_c1}
    \begin{aligned}
    \norm{L_0G_{c1}}_{H^{s-1}}\lesssim&\norm{L_0D_t\zeta}_{H^{s-1}}\norm{\zeta_\alpha-1}_{W^{s-2,\infty}}\norm{q}_{W^{s-1,\infty}}+\norm{D_t\zeta}_{W^{s-1,\infty}}\norm{L_0\zeta_\alpha}_{H^{s-1}}\norm{q}_{W^{s-1,\infty}}\\
    &+\norm{D_t\zeta}_{W^{s-1,\infty}}\norm{\zeta_\alpha-1}_{W^{s-2,\infty}}\norm{L_0q}_{H^{s-1}}+\norm{L_0q}_{H^{s-1}}\norm{\zeta_\alpha-1}_{W^{s-2,\infty}}\norm{q}_{W^{s-1,\infty}}\\
    &+\norm{q}_{W^{s-1,\infty}}\norm{L_0\zeta_\alpha}_{H^{s-1}}\norm{q}_{W^{s-1,\infty}}\\
    \lesssim&\ \frac{\lambda\epsilon^2}{H^2t^{\frac{1}{4}-\delta_0}}\ln (2+t)+\frac{\lambda\epsilon^2(\epsilon+\lambda)t^{\frac{1}{4}}}{H^{\frac{5}{2}}}\ln (2+t)+\frac{\lambda^2\epsilon(\epsilon+\lambda)t^{\frac{1}{2}}}{H^{\frac{9}{2}}}+\frac{\lambda^2\epsilon t^{\delta_0}}{H^4}.
    \end{aligned}
    \end{equation}
    For $L_0G_{c2}$, things are getting a bit complicated. Since the time-decay of $\norm{\partial_\alpha^{s-1}D_t\zeta}_{L^\infty}$ and $\norm{\partial_\alpha^s\zeta}_{L^\infty}$ fails to achieve the rate $t^{-\frac{1}{2}}$, a more delicate analysis is needed. For instance, we examine the function
    $$
    g=\int\frac{(\partial_\alpha^{s-1}\zeta_t(\alpha,t)-\partial_\beta^{s-1}\zeta_t(\beta,t))(L_0\zeta(\alpha,t)-L_0\bar{\zeta}(\alpha,t)-L_0\zeta(\beta,t)+L_0\bar{\zeta}(\beta,t))}{|\zeta(\alpha,t)-\zeta(\beta,t)|^2}\partial_\beta \zeta_t d\beta
    $$
    which represents the situation where all the spatial derivatives act on $D_t\zeta$ of the first term in $G_{c2}$. To prove the lemma, we aim to transfer a t-derivative from $\partial_\beta^{s-1}\zeta_t$ to $\partial_\beta \zeta_t$. More precisely, denote $L_0\zeta-L_0\bar\zeta$ by $h$:
    $$
    \begin{aligned}
    g=&\int\frac{h(\alpha,t)-h(\beta,t)}{|\zeta(\alpha,t)-\zeta(\beta,t)|^2}[(\partial_\alpha^{s-1}\zeta_t(\alpha,t)-\partial_\beta^{s-1}\zeta_t(\beta,t))\partial_\beta\zeta_{t}- (\partial_\alpha^{s-1}\zeta(\alpha,t)-\partial_\beta^{s-1}\zeta(\beta,t))\partial_\beta\zeta_{tt}]d\beta\\
    &+\int\frac{(\partial_\alpha^{s-1}\zeta(\alpha,t)-\partial_\beta^{s-1}\zeta(\beta,t))(h(\alpha,t)-h(\beta,t))}{|\zeta(\alpha,t)-\zeta(\beta,t)|^2}\partial_\beta\zeta_{tt} d\beta.\\
    \end{aligned}
    $$
    \eqref{transfer_d_var_2} suggests that
    $$
    \begin{aligned}
    \norm{g}_{L^2}\lesssim&\ \frac{1}{t}(\norm{\Omega_0\zeta_\alpha}_{H^{s-2}}\norm{L_0\zeta_\alpha}_{H^1}\norm{\partial_\alpha D_t^2\zeta}_{L^\infty}+\norm{D_t\zeta}_{W^{s-1,\infty}}\norm{L_0\zeta_\alpha}_{L^2}\norm{\partial_\alpha D_t^2\zeta}_{L^\infty})\\
    &+\frac{1}{t}\norm{D_t\zeta}_{W^{s-1,\infty}}\norm{L_0\zeta_\alpha}_{H^1}\norm{\Omega_0\partial_\alpha D_t\zeta}_{L^2}+\norm{\zeta_\alpha-1}_{W^{s-2,\infty}}\norm{L_0\zeta_\alpha}_{L^2}\norm{\partial_\alpha D_t^2\zeta}_{L^\infty}\\
    \lesssim&\ \frac{\epsilon^2}{t}\sqrt{\mathfrak{E}_s^{L_0}}+\frac{\epsilon^3}{t^{1+\delta}}.
    \end{aligned}
    $$
    The other cases can be handled as well. In a word, there exists
    \begin{equation}\label{L_0G_c2}
    \norm{L_0G_{c2}}_{H^{s-1}}\lesssim\frac{\epsilon^2}{t}\sqrt{\mathfrak{E}_s^{L_0}}+\frac{\epsilon^3}{t^{1+\delta}}.
    \end{equation}
    Utilizing Lemma \ref{lem_L_0G_d}, (\ref{L_0G_c1}) and (\ref{L_0G_c2}), we get
    $$
    \begin{aligned}
    \norm{L_0G_0^{\tilde{\theta}}}_{H^{s-1}}\lesssim&\left(\frac{\epsilon^2}{t}+\frac{\lambda}{H^2}\right)\sqrt{\mathfrak{E}_s^{L_0}}+\frac{\epsilon^3}{t^{1+\delta}}+\frac{\lambda(\epsilon+\lambda)^2t}{H^{3}}+\frac{\lambda(\epsilon+\lambda)}{H^2}+\frac{\lambda\epsilon(\epsilon+\lambda)t^{\frac{1}{2}+\delta_0}}{H^2}.
    \end{aligned}
    $$
    It remains to estimate $\norm{L_0G_0^{\tilde{\sigma}}}_{H^{s-1}}$. Imitating the preceding argument leads to
    $$
    \norm{L_0D_tG_0^{\tilde{\theta}}}_{H^{s-1}}\lesssim\left(\frac{\epsilon^2}{t}+\frac{\lambda}{H^2}\right)\sqrt{\mathfrak{E}_s^{L_0}}+\frac{\epsilon^3}{t^{1+\delta}}+\frac{\lambda(\epsilon+\lambda)^2t}{H^{3}}+\frac{\lambda(\epsilon+\lambda)}{H^2}+\frac{\lambda\epsilon(\epsilon+\lambda)t^{\frac{1}{2}+\delta_0}}{H^2}.
    $$
    Next we examine the term $L_0(\frac{a_t}{a}\circ\kappa^{-1}A\tilde{\theta}_\alpha)$ and denote it by $\mathfrak{A}$:
    $$
    \mathfrak{A}=L_0(\frac{a_t}{a}\circ\kappa^{-1}A)\tilde{\theta}_\alpha+\frac{a_t}{a}\circ\kappa^{-1}AL_0\tilde{\theta}_\alpha.
    $$
    Referring to Lemma \ref{lem_a_t/a_est}, we see that
    $$
    \begin{aligned}
    \norm{\frac{a_t}{a}\circ\kappa^{-1}AL_0\tilde{\theta}_\alpha}_{H^{s-1}}\lesssim&\norm{\frac{a_t}{a}\circ\kappa^{-1}A}_{W^{s-1,\infty}}\norm{L_0\tilde{\theta}_\alpha}_{H^{s-1}}\\
    \lesssim&\left(\frac{\epsilon^2}{t^{\frac{7}{6}-\delta_0}}(\ln (2+t))^2+\frac{\lambda\epsilon}{H^2t^{\frac{1}{2}}}+\frac{\lambda}{H^3}\right)(\sqrt{\mathfrak{E}_s^{L_0}}+\frac{\epsilon^2}{t^{\frac{1}{6}-\delta_0}}\ln (2+t)).
    \end{aligned}
    $$
    Meanwhile, we can make use of the method which plays a pivotal role in the estimate of $\partial_\alpha^{s-1}D_tL_0b\tilde{\theta}_\alpha$ to control $\norm{L_0(\frac{a_t}{a}\circ\kappa^{-1}A)\tilde{\theta}_\alpha}_{H^{s-1}}$. Hence we deduce that
    $$
    \norm{L_0(\frac{a_t}{a}\circ\kappa^{-1}A)\tilde{\theta}_\alpha}_{H^{s-1}}\lesssim\frac{\epsilon^2}{t}\sqrt{\mathfrak{E}_s^{L_0}}+\frac{\epsilon^3}{t^{1+\delta}}+\frac{\lambda\epsilon}{H^2t^{\frac{1}{6}}}\ln (2+t).
    $$
    Since
    $$
    L_0G_0^{\tilde{\sigma}}=\mathfrak{A}+L_0D_tG_0^{\tilde{\theta}},
    $$
    we conclude that
    $$
    \norm{L_0G_0^{\tilde{\sigma}}}_{H^{s-1}}\lesssim\left(\frac{\epsilon^2}{t}+\frac{\lambda}{H^2}\right)\sqrt{\mathfrak{E}_s^{L_0}}+\frac{\epsilon^3}{t^{1+\delta}}+\frac{\lambda(\epsilon+\lambda)^2t}{H^{3}}+\frac{\lambda(\epsilon+\lambda)}{H^2}+\frac{\lambda\epsilon(\epsilon+\lambda)t^{\frac{1}{2}+\delta_0}}{H^2}.
    $$
\end{proof}

With all the estimates in place, we are ready to analyze the behavior of $\mathfrak{E}_s^{L_0}$.
\begin{thm}\label{thm_d_t_E^L_0}
    For $t\in[0,T]$, we claim that
    \begin{equation}\label{d_t_E^L_0}
    \begin{aligned}
    &\left|\frac{d}{dt}{\mathfrak{E}_s^{L_0}}\right|\\
    &\leq C\left[\left(\frac{\epsilon^2}{t}+\frac{\lambda}{H^2}\right)\sqrt{\mathfrak{E}_s^{L_0}}+\frac{\epsilon^3}{t}+\frac{\lambda(\epsilon+\lambda)^2(1+t)}{H^{3}}+\frac{\lambda(\epsilon+\lambda)}{H^2}+\frac{\lambda\epsilon(\epsilon+\lambda)(1+t)^{\frac{1}{2}+\delta_0}}{H^2}\right]\sqrt{\mathfrak{E}_s^{L_0}}.
    \end{aligned}
    \end{equation}
\end{thm}
\begin{proof}
    From (\ref{d_tE^L_0}) we have
    $$
    \left|\frac{d}{dt}{\mathfrak{E}_s^{L_0}}\right|\lesssim\sum_{k=0}^{s-1}\left(\norm{D_t\partial_\alpha^kL_0\tilde{\theta}}_{L^2}\norm{G_k^{L_0\tilde{\theta}}}_{L^2}+\norm{D_t\partial_\alpha^kL_0\tilde{\theta}}_{L^2}^2\norm{\frac{a_t}{a}\circ\kappa^{-1}}_{L^\infty}\right)
    $$
    Appealing to Lemma \ref{lem_[P,d_alpha]L_0}, Lemma \ref{lem_[P,L_0]} and Lemma \ref{lem_L_0G_0} yields (\ref{d_t_E^L_0}).
\end{proof}

\begin{cor}\label{cor_E^L_0}
    For $t\in[0,T]$, there holds
    $$
    \sqrt{\mathfrak{E}_s^{L_0}}\leq e^{Cc_s^2}(1+t)^{C\epsilon^2}\left(\sqrt{\mathfrak{E}_s^{L_0}(0)}+C\epsilon^3\ln (1+t)+Cc_s\epsilon\right).
    $$ 
\end{cor}
\begin{proof}
    Theorem \ref{thm_d_t_E^L_0} implies that
    $$
    \frac{d}{dt}\sqrt{\mathfrak{E}_s^{L_0}}\leq\frac{C}{2}\left[\left(\frac{\epsilon^2}{1+t}+\frac{\lambda}{H^2}\right)\sqrt{\mathfrak{E}_s^{L_0}}+\frac{\epsilon^3}{1+t}+\frac{\lambda(\epsilon+\lambda)^2t}{H^{3}}+\frac{\lambda(\epsilon+\lambda)}{H^2}+\frac{\lambda\epsilon(\epsilon+\lambda)t^{\frac{1}{2}+\delta_0}}{H^2}\right].
    $$
    Denote
    $$
    \Theta_{L_0}=\frac{C}{2}\left(\frac{\epsilon^2}{1+t}+\frac{\lambda}{H^2}\right).
    $$
    Multiply both sides of the equation by the integrating factor $e^{-\int_0^t\Theta_{L_0}(\tau)d\tau}$:
    $$
    \begin{aligned}
    \frac{d}{dt}&\left[e^{-\int_0^t\Theta_{L_0}(\tau)d\tau}\sqrt{\mathfrak{E}_s^{L_0}}\right]\\
    &\leq \frac{C}{2}e^{-\int_0^t\Theta_{L_0}(\tau)d\tau}\left(\frac{\epsilon^3}{1+t}+\frac{\lambda(\epsilon+\lambda)^2t}{H^{3}}+\frac{\lambda(\epsilon+\lambda)}{H^2}+\frac{\lambda\epsilon(\epsilon+\lambda)t^{\frac{1}{2}+\delta_0}}{H^2}\right).
    \end{aligned}
    $$
    Integrating both sides yields that
    $$
    \begin{aligned}
    &\sqrt{\mathfrak{E}_s^{L_0}}\\
    &\leq e^{\int_0^t\Theta_{L_0}(\tau)d\tau}\left[\sqrt{\mathfrak{E}_s^{L_0}(0)}
    +\frac{C}{2}\int_0^t\left(\frac{\epsilon^3}{1+\tau}+\frac{\lambda(\epsilon+\lambda)^2\tau}{H^3}+\frac{\lambda(\epsilon+\lambda)}{H^{2}}+\frac{\lambda\epsilon(\epsilon+\lambda)\tau^{\frac{1}{2}+\delta_0}}{H^2}\right)d\tau\right].
    \end{aligned}
    $$
    Appealing to \eqref{boot_H}, we obtain
    $$
    \int_0^t\Theta_{L_0}(\tau)d\tau\leq\frac{C}{2}\epsilon^2\ln (1+t)+C\frac{1}{H_0}.
    $$
    Next we check the integral 
    $$
    \int_0^t\frac{\lambda(\epsilon+\lambda)^2\tau}{H^3}d\tau\lesssim\int_0^\infty\frac{\lambda(\epsilon+\lambda)^2\tau}{(H_0+\lambda \tau)^3}d\tau\lesssim\frac{(\epsilon+\lambda)^2}{\lambda H_0}\lesssim\frac{\epsilon}{H_0^{\frac{1}{2}}}.
    $$
    Similar calculus yields that
    $$
    \sqrt{\mathfrak{E}_s^{L_0}}\leq(1+t)^{C\epsilon^2}e^{\frac{C}{H_0}}\left[\sqrt{\mathfrak{E}_s^{L_0}(0)}+C\epsilon^3\ln (1+t)+C\left(\frac{\epsilon}{H_0^{\frac{1}{2}}}+\frac{\epsilon^{\frac{3}{2}-\delta_0}}{H_0^{\frac{1}{4}-\frac{3}{2}\delta_0}}\right)\right].
    $$
    Due to the assumptions we made about $H_0$ and the smallness of $\epsilon$, there holds
    $$
    \sqrt{\mathfrak{E}_s^{L_0}}\leq e^{Cc_s^2}(1+t)^{C\epsilon^2}\left(\sqrt{\mathfrak{E}_s^{L_0}(0)}+C\epsilon^3\ln (1+t)+Cc_s\epsilon\right).
    $$
    
\end{proof}
\section{Decay estimates}\label{sec_decay}
We devote this section to the decay estimates of $\norm{(\partial_\alpha D_t\zeta,\zeta_\alpha-1,D_t^2\zeta)}_{W^{s-3,\infty}\times W^{s-2,\infty}\times W^{s-2,\infty}}$. Recall the notation $S(t)=\{\alpha\in\mathbb{R}: t^{1-\mu}\leq|\alpha|\leq t^{1+\mu}\}, \mu\in(0,\frac{1}{2})$ used in the preceding section. In Lemma \ref{lem_away_infty}, we handle the situation where $\alpha$ is away from $t$, i.e., $\alpha\in S^c(t)$. For the rest of the section, we introduce a functional $\tilde{E}$ which helps us study the behavior of unknowns in $S(t)$. Note that the wave packet $u$
 we introduced, to some extent, suppresses the oscillations of the unknowns.
\begin{lemma}\label{lem_away_infty}
    Assume the bootstrap assumptions. There holds
    \begin{equation}\label{S(t)^c_L^infty}
    \norm{(\partial_\alpha D_t\zeta,\zeta_\alpha-1,D_t^2\zeta)}_{W^{s-3,\infty}(S(t)^c)\times W^{s-2,\infty}(S(t)^c)\times W^{s-2,\infty}(S(t)^c)}\lesssim\frac{\epsilon}{t^{1-\delta_0}}+\frac{\epsilon}{t^{\frac{1}{4}+{2\mu}}}\ln (2+t)+\frac{\epsilon}{t^{\frac{1}{2}+\frac{1}{2}\mu-\delta_0}}
    \end{equation}
    for $t\in[0,T]$.
\end{lemma}
\begin{proof}
    We prove the inequality about $\partial_\alpha D_t\zeta$ in detail. The remainder of the lemma can be treated by a similar argument. For $1\leq k\leq s-2$, employing Lemma \ref{lem_D_t_zeta_L^infty} and (\ref{transfer_partial_alpha}) yields that
    $$
    \begin{aligned}
    \norm{\partial_\alpha^kD_t\zeta}_{L^\infty(|\alpha|<t^{1-\mu})}\lesssim&\ \frac{1}{t}(\norm{L_0\partial_\alpha^kD_t\zeta}_{L^\infty}+\norm{\Omega_0\partial_\alpha^kD_t\zeta}_{L^\infty})+\frac{1}{t^{2\mu}}\norm{\partial_\alpha^{k+1}D_t\zeta}_{L^\infty}\\
    \lesssim&\ \frac{\epsilon}{t^{1-\delta_0}}+\frac{\epsilon}{t^{\frac{1}{4}+{2\mu}}}\ln (2+t).
    \end{aligned}
    $$
    For $|\alpha|>t^{1+\mu}$, we consider $\partial_\alpha^kD_t\tilde{\theta}$ instead and adopt (3.6) of \cite{Wu2009}:
    $$
    |\partial_\alpha^kD_t\tilde{\theta}|\leq\frac{2}{|\alpha|^{\frac{1}{2}}}\norm{L_0\partial_\alpha^kD_t\tilde{\theta}}_{L^2}+\frac{t}{|\alpha|^{\frac{3}{2}}}\norm{\Omega_0\partial_\alpha^kD_t\tilde{\theta}}_{L^2}\lesssim\frac{\epsilon}{t^{\frac{1}{2}+\frac{1}{2}\mu-\delta_0}}.
    $$
    Then Lemma \ref{lem_theta_t_zeta_t} reveals (\ref{S(t)^c_L^infty}).
\end{proof}

Now we focus on the situation where $|\alpha|$ is near $t$. More precisely, construct the function $u(\alpha,t,\nu)=\chi\left(\frac{\alpha-\nu t}{\sqrt{t}}\right)e^{-\frac{it^2}{4\alpha}}$ where $\chi(x)\in C^\infty_c(\mathbb{R})$ is a real smooth bump function satisfying $\int \chi(x)dx=1$, $\operatorname{supp}(\chi)\subset[-2,2]
$, and $\nu\in\mathbb{R}\setminus \{0\}$. Simple calculation indicates that
\begin{equation}\label{u_t}
\partial_tu=-\frac{it}{2\alpha}e^{-\frac{it^2}{4\alpha}}\chi\left(\frac{\alpha-\nu t}{\sqrt{t}}\right)-\left(\frac{\alpha-\nu t}{2t^{\frac{3}{2}}}+\frac{\nu}{t^{\frac{1}{2}}}\right)e^{-\frac{it^2}{4\alpha}}\chi'\left(\frac{\alpha-\nu t}{\sqrt{t}}\right),
\end{equation}
\begin{equation}
\partial_\alpha u=\frac{it^2}{4\alpha^2}e^{-\frac{it^2}{4\alpha}}\chi\left(\frac{\alpha-\nu t}{\sqrt{t}}\right)+\frac{1}{\sqrt{t}}e^{-\frac{it^2}{4\alpha}}\chi'\left(\frac{\alpha-\nu t}{\sqrt{t}}\right)
\end{equation}
and
\begin{equation}
\Omega_0u=-\alpha\left(\frac{\alpha-\nu t}{2t^{\frac{3}{2}}}+\frac{\nu}{t^{\frac{1}{2}}}\right)e^{-\frac{it^2}{4\alpha}}\chi'\left(\frac{\alpha-\nu t}{\sqrt{t}}\right).
\end{equation}
Then we can calculate the $L^2$ norm:
$$
\begin{aligned}
\norm{\partial_tu}_{L^2}\leq&\left(\int\frac{t^2}{4\alpha^2}\chi^2\left(\frac{\alpha-\nu t}{\sqrt{t}}\right)d\alpha\right)^{\frac{1}{2}}+\left(\int\left(\frac{\alpha-\nu t}{2t^{\frac{3}{2}}}+\frac{\nu}{t^{\frac{1}{2}}}\right)^2\left(\chi'\right)^2\left(\frac{\alpha-\nu t}{\sqrt{t}}\right)d\alpha\right)^{\frac{1}{2}}\\
\leq&\left(\int\frac{t^{\frac{3}{2}}}{4\alpha^2}\chi^2(\alpha-\nu\sqrt{t})d\alpha\right)^{\frac{1}{2}}+\left(\int\left(\frac{\alpha}{2t}+\frac{\nu}{t^{\frac{1}{2}}}\right)^2t^{\frac{1}{2}}\left(\chi'\right)^2(\alpha)d\alpha\right)^{\frac{1}{2}}.
\end{aligned}
$$
Since the support of $\chi$ is compact, we can restrict the integration region to $[\nu\sqrt{t}-2,\nu\sqrt{t}+2]$ in the first integral and $[-2,2]$ in the second integral, which means that
\begin{equation}\label{u_t_L^2}
\norm{\partial_tu}_{L^2}\lesssim\frac{t^{\frac{1}{4}}}{|\nu|}\norm{\chi}_{L^2}+\left(\frac{1}{t^{\frac{3}{4}}}+\frac{|\nu|}{t^{\frac{1}{4}}}\right)\norm{\chi'}_{L^2}.
\end{equation}
Similarly we have
\begin{equation}\label{u_t_L^1}
    \norm{\partial_tu}_{L^1}\lesssim\frac{t^{\frac{1}{2}}}{|\nu|}\norm{\chi}_{L^1}+\left(\frac{1}{t^{\frac{1}{2}}}+|\nu|\right)\norm{\chi'}_{L^1},
\end{equation}
\begin{equation}\label{u_alpha_L^2}
\norm{\partial_\alpha u}_{L^2}\lesssim\frac{t^{\frac{1}{4}}}{\nu^2}\norm{\chi}_{L^2}+\frac{1}{t^{\frac{1}{4}}}\norm{\chi'}_{L^2},
\end{equation}
\begin{equation}\label{Omega_0u_L^2}
\norm{\Omega_0u}_{L^2}\lesssim(1+|\nu|\sqrt{t})(t^{-\frac{1}{4}}+|\nu| t^{\frac{1}{4}})\norm{\chi'}_{L^2}
\end{equation}
and
$$
\norm{\Omega_0u}_{L^1}\lesssim\nu^2t\norm{\chi'}_{L^1}.
$$
In the remainder of this section we demonstrate that $D_t^2\tilde{\theta}\leq \tilde{C}\epsilon t^{-\frac{1}{2}}$. The strategies can be applied to the higher-order derivatives of $D_t^2\tilde{\theta},\ \partial_\alpha D_t\tilde{\theta}$ and $\tilde{\theta}_\alpha$. Define the functional
\begin{equation}\label{tilde_E}
\tilde{E}(\nu,t)=\int D_t^2\tilde{\theta}\bar u-D_t\tilde{\theta}D_t\bar u.
\end{equation}
Subsequently we will explore the relation between $\tilde{E}$ and $D_t^2\tilde{\theta}$:
\begin{lemma}\label{lem_tilde_E}
Suppose $0<\mu<\frac{1}{4} $ and $|\nu|\in[t^{-\mu},t^{\mu}]$. For $t\in[1,T]$ we have
    \begin{equation}\label{tilde_E_theta_alpha}
    \left|\tilde{E}-2ie^{\frac{it}{4\nu}}\sqrt{t}\tilde{\theta}_\alpha(\nu t,t)\right|\lesssim\frac{\epsilon}{t^{\frac{1}{2}-2\mu}}+\frac{\epsilon}{t^{\frac{1}{4}-\mu-\delta_0}}.
    \end{equation}
\end{lemma}
\begin{proof}
    We rewrite the expression of $\tilde{E}$ as below:
    \begin{equation}\label{tilde_E_exp}
    \begin{aligned}
    \tilde{E}=&\int2D_t^2\tilde{\theta}\bar u-D_t(D_t\tilde{\theta}\bar u)\\
    =&\int2D_t^2\tilde{\theta}\bar u-b\partial_\alpha(D_t\tilde{\theta}\bar u)d\alpha-\partial_t\int \left(D_t\tilde{\theta}(\alpha,t)-D_t\tilde{\theta}(\nu t, t)\right)\bar u(\alpha,t)d\alpha-\partial_t\left(D_t\tilde{\theta}(\nu t,t)\int\bar u\right).
    \end{aligned}
    \end{equation}
    For the third term we have
    $$
    \partial_t\left(D_t\tilde{\theta}(\nu t,t)\int\bar u\right)=\left(\nu\partial_\alpha D_t\tilde{\theta}(\nu t,t)+\partial_t D_t\tilde{\theta}(\nu t,t)\right)\int\bar u+D_t\tilde{\theta}(\nu t,t)\int \bar u_t.
    $$
    By integration by parts,
    $$
    \int\bar u=-\int\chi\left(\frac{\alpha-\nu t}{\sqrt{t}}\right)\frac{4\alpha^2}{it^2}\partial_\alpha\left(e^{\frac{it^2}{4\alpha}}\right)d\alpha=\int\frac{8\alpha}{it^2}\bar u+\frac{4\alpha^2}{it^\frac{5}{2}}\chi'\left(\frac{\alpha-\nu t}{\sqrt{t}}\right)e^{\frac{it^2}{4\alpha}}d\alpha
    $$
    which implies that
    $$
    \left|\int\bar u\right|\lesssim t^{2\mu}.
    $$
    In fact, repeating the process several times shows that for all $k>0$
    \begin{equation}\label{int_u}
    \left|\int\bar u\right|\lesssim t^{-k}.
    \end{equation}
    The same holds for $\int \bar u_t$. Therefore $\partial_t\left(D_t\tilde{\theta}(\nu t,t)\int\bar u\right)$ is negligible. For the second term of $\tilde{E}$:
    $$
    \begin{aligned}
    &\ -\partial_t\int \left(D_t\tilde{\theta}(\alpha,t)-D_t\tilde{\theta}(\nu t, t)\right)\bar u(\alpha,t)d\alpha\\
    =&\int \left(\partial_t D_t\tilde{\theta}(\nu t,t)-\partial_t D_t\tilde{\theta}(\alpha, t)\right)\bar u(\alpha,t)+ \left(D_t\tilde{\theta}(\nu t,t)-D_t\tilde{\theta}(\alpha, t)\right)\bar u_t(\alpha,t)d\alpha\\
    &+\nu\partial_\alpha D_t\tilde{\theta}(\nu t,t)\int\bar u\\
    =&\ \frac{2}{it}\int\frac{\Omega_0D_t\tilde{\theta}(\nu t,t)-\Omega_0 D_t\tilde{\theta}(\alpha,t)}{\nu t-\alpha}(\nu t-\alpha)\bar u_t-(\partial_t D_t\tilde{\theta}(\nu t,t)-\partial_t D_t\tilde{\theta}(\alpha,t))\overline{\Omega_0u(\alpha,t)}d\alpha\\
    &\ +\frac{2}{it}\partial_t D_t\tilde{\theta}(\nu t,t)\int(\alpha-\nu t)\bar u_td\alpha+\nu\partial_\alpha D_t\tilde{\theta}(\nu t,t)\int\bar u.
    \end{aligned}
    $$
    The second equation follows from \eqref{transfer_d_var_1}. By adopting \eqref{int_u} and Hardy's inequality,
    \begin{equation}\label{tilde_E_D_t_theta_u_1}
    \begin{aligned}
    &\left|-\partial_t\int \left(D_t\tilde{\theta}(\alpha,t)-D_t\tilde{\theta}(\nu t, t)\right)\bar u(\alpha,t)d\alpha\right|\\
    \lesssim\ &\frac{1}{t}\left(\norm{\frac{\Omega_0D_t\tilde{\theta}(\nu t,t)-\Omega_0 D_t\tilde{\theta}(\alpha,t)}{\nu t-\alpha}}_{L^2}\norm{(\nu t-\alpha)\bar u_t}_{L^2}+\norm{\partial_t D_t\tilde{\theta}}_{L^\infty}\left(\norm{\Omega_0u}_{L^1}+\norm{(\alpha-\nu t)\bar u_t}_{L^1}\right)\right)\\
    \lesssim\ &\frac{1}{t}\left(\norm{\partial_\alpha\Omega_0 D_t\tilde{\theta}}_{L^2}\norm{(\nu t-\alpha)\bar u_t}_{L^2}+\norm{\partial_t D_t\tilde{\theta}}_{L^\infty}\left(\norm{\Omega_0u}_{L^1}+\norm{(\alpha-\nu t)\bar u_t}_{L^1}\right)\right)\\
    \lesssim\ &\frac{\epsilon}{t^{\frac{1}{4}-\mu-\delta_0}}+\frac{\epsilon}{t^{\frac{1}{2}-2\mu}}.
    \end{aligned}
    \end{equation}
    Obviously,
    $$
    \norm{u}_{L^2}=\left(\int\chi^2(\alpha)\sqrt{t}d\alpha\right)^{\frac{1}{2}}\lesssim t^\frac{1}{4}.
    $$
    By utilizing (\ref{b_L^2}) and (\ref{u_alpha_L^2}) we note that
    \begin{equation}\label{tilde_E_D_t_theta_u_2}
    \left|\int b\partial_\alpha(D_t\tilde{\theta}\bar u)\right|\lesssim\norm{b}_{L^2}(\norm{D_t\tilde{\theta}}_{L^{\infty}}\norm{u_\alpha}_{L^2}+\norm{\partial_\alpha D_t\tilde{\theta}}_{L^{\infty}}\norm{u}_{L^2})\lesssim\frac{\epsilon^3}{t^{\frac{1}{2}-2\mu}}+\frac{\lambda\epsilon t^{2\mu}}{H^{\frac{3}{2}}}\lesssim\frac{\epsilon}{t^{\frac{1}{2}-2\mu}}.
    \end{equation}
    From \eqref{tilde_E_exp}, \eqref{tilde_E_D_t_theta_u_1} and \eqref{tilde_E_D_t_theta_u_2} we have proved that
    \begin{equation}\label{tilde_E_D_t_theta_u}
        \left|\tilde{E}-\int2D_t^2\tilde{\theta}\bar u\right|\lesssim\frac{\epsilon}{t^{\frac{1}{2}-2\mu}}+\frac{\epsilon}{t^{\frac{1}{4}-\mu-\delta_0}}.
    \end{equation}
    For the remaining term $\int 2D_t^2\tilde{\theta}\bar u$, we proceed to compute its leading term:
    \begin{equation}\label{D_t^2_theta_u}
    \begin{aligned}
    \int 2D_t^2\tilde{\theta}\bar u=&\int2i\tilde{\theta}_\alpha\bar u+2i(A-1)\tilde{\theta}_\alpha\bar u+2(G_c+G_d)\bar u\\
    =&\ 2ie^{\frac{it}{4\nu}}\sqrt{t}\tilde{\theta}_\alpha(\nu t,t)+\int2i(e^{\frac{it^2}{4\alpha}}\tilde{\theta}_\alpha(\alpha,t)-e^{\frac{it}{4\nu}}\tilde{\theta}_\alpha(\nu t,t))\chi\left(\frac{\alpha-\nu t}{\sqrt{t}}\right)d\alpha\\
    &+\int2i(A-1)\tilde{\theta}_\alpha\bar u+2(G_c+G_d)\bar u.
    \end{aligned}
    \end{equation}
    With H\"older's inequality and Hardy's inequality we deduce that
    $$\begin{aligned}
    &\left|\int2i(e^{\frac{it^2}{4\alpha}}\tilde{\theta}_\alpha(\alpha,t)-e^{\frac{it}{4\nu}}\tilde{\theta}_\alpha(\nu t,t))\chi\left(\frac{\alpha-\nu t}{\sqrt{t}}\right)d\alpha\right|\\
    \lesssim&\norm{\frac{e^{\frac{it^2}{4\alpha}}\tilde{\theta}_\alpha(\alpha,t)-e^{\frac{it}{4\nu}}\tilde{\theta}_\alpha(\nu t,t)}{\alpha-\nu t}}_{L^2(|\alpha-\nu t |\lesssim \sqrt{t})}\norm{(\alpha-\nu t)\chi\left(\frac{\alpha-\nu t}{\sqrt{t}}\right)}_{L^2}\\
    \lesssim&\ t^{\frac{3}{4}}\norm{\partial_\alpha(e^{\frac{it^2}{4\alpha}}\tilde{\theta}_\alpha)}_{L^2(|\alpha-\nu t|\lesssim\sqrt{t})}\\
    =&\ t^{\frac{3}{4}}\norm{\partial_\alpha\left[e^{\frac{it^2}{4\alpha}}\left(\frac{4\alpha^2}{it^2}\tilde{\theta}_{\alpha\alpha}-\frac{4\alpha}{it^2}L_0\tilde{\theta}_\alpha+\frac{2}{it}\Omega_0\tilde{\theta}_\alpha\right)\right]}_{L^2(|\alpha-\nu t|\lesssim\sqrt{t})}\\
    \lesssim&\ t^{\frac{3}{4}}\norm{-\frac{2}{t}\Omega_0\tilde{\theta}_{\alpha\alpha}+\frac{4\alpha}{t^2}L_0\tilde{\theta}_{\alpha\alpha}}_{L^2(|\alpha-\nu t|\lesssim\sqrt{t})}\\
    &+t^{\frac{3}{4}}\norm{\frac{8\alpha}{t^2}\tilde{\theta}_{\alpha\alpha}-\frac{4}{t^2}L_0\tilde{\theta}_{\alpha}-\frac{4\alpha}{t^2}\partial_\alpha L_0\tilde{\theta}_\alpha+\frac{2}{t}\partial_\alpha\Omega_0\tilde{\theta}_\alpha}_{L^2(|\alpha-\nu t|\lesssim\sqrt{t})}\\
    \lesssim&\ \frac{\epsilon}{t^{\frac{1}{4}-\mu-\delta_0}}.
    \end{aligned}
    $$
    Obviously the last integral in (\ref{D_t^2_theta_u}) decays faster. Combining \eqref{tilde_E_D_t_theta_u} and (\ref{D_t^2_theta_u}) proves (\ref{tilde_E_theta_alpha}).
\end{proof}

Now we investigate $\frac{d}{dt}\tilde{E}$. Since Lemma \ref{lem_away_infty} handles the case where 
$|\alpha|$ is far from $t$, we will focus our attention on the case where 
$|\alpha|$
 is close to 
$t$. From now on, we set $t^{-\mu}\leq|\nu|\leq t^{\mu}$. The following sequence of inequalities hold uniformly for $\nu$. From (\ref{tilde_E}) we have
$$
\tilde{E}(\nu,t)=\int(\theta_{tt}\bar u\circ\kappa-\theta_t\partial_t(\bar u\circ\kappa))\kappa_\alpha d\alpha.
$$
Thus
$$\begin{aligned}
\frac{d}{dt}\left(\tilde{E}(\nu,t)\right)=&\int(\theta_{tt}\bar u\circ\kappa-\theta_t\partial_t(\bar u\circ\kappa))\partial_t\kappa_\alpha d\alpha+\int(\theta_{ttt}\bar u\circ\kappa-\theta_t\partial_t^2(\bar u\circ\kappa))\kappa_\alpha d\alpha\\
=&\int(\theta_{tt}\bar u\circ\kappa-\theta_t\partial_t(\bar u\circ\kappa))\partial_t\kappa_\alpha d\alpha\\
&+\int[(\theta_{ttt}-ia\theta_{t\alpha})\bar u\circ\kappa-\theta_t(\partial_t^2(\bar u\circ\kappa)+ia\partial_\alpha(\bar u\circ\kappa)]\kappa_\alpha d\alpha\\
&+\int-i\partial_\alpha(a\kappa_\alpha)\theta_t\bar u\circ\kappa d\alpha\\
=&\int(D_t^2\tilde{\theta}\bar u-D_t\tilde{\theta}D_t\bar u)b_\alpha d\alpha\\
&+\int[(D_t^3\tilde{\theta}-iA\partial_\alpha D_t\tilde{\theta})\bar u-D_t\tilde{\theta}(D_t^2\bar u+iA\bar u_\alpha)]d\alpha\\
&+\int-iA_\alpha D_t\tilde{\theta}\bar ud\alpha.
\end{aligned}
$$
The utilization of Lemma \ref{lem_b} and Lemma \ref{lem_A} reveals that
\begin{equation}\label{d_t_tilde_E_one}
\begin{aligned}
\left|\int D_t^2\tilde{\theta}\bar ub_\alpha d\alpha\right|
\lesssim\norm{D_t^2\tilde{\theta}}_{L^\infty}\norm{u}_{L^1}\norm{b_\alpha}_{L^\infty}\lesssim\frac{\epsilon^3}{t^{\frac{5}{4}-\mu-\delta_0}}(\ln (2+t))^2+\frac{\lambda\epsilon}{H^3}
\end{aligned}
\end{equation}
and
\begin{equation}\label{d_t_tilde_E_two}
\left|\int-iA_\alpha D_t\tilde{\theta}\bar ud\alpha\right|\lesssim\norm{A_\alpha}_{L^\infty}\norm{D_t\tilde{\theta}}_{L^2}\norm{u}_{L^2}\lesssim\frac{\epsilon^3}{t^{\frac{5}{4}-\delta_0}}(\ln (2+t))^2+\frac{\lambda\epsilon(\epsilon+\lambda) (1+t)^{\frac{1}{4}}}{H^2}.
\end{equation}
For the term $\int D_t\tilde{\theta}D_t\bar ub_\alpha d\alpha$, we decompose it into
$$
\int D_t\tilde{\theta}D_t\bar ub_\alpha d\alpha=\int D_t\tilde{\theta}D_t\bar u\partial_\alpha\left(\frac{1}{2}(I-\mathfrak{H})b+\frac{1}{2}(I-\bar{\mathfrak{H}})b+\frac{1}{2}(\mathfrak{H}+\bar{\mathfrak{H}})b\right)d\alpha.
$$
Then we consider the term $\int D_t\tilde{\theta}D_t\bar u\partial_\alpha(I-\mathfrak{H})b d\alpha$ and recall \eqref{mul_b_alpha}:
$$
\begin{aligned}
&\ \left|\int D_t\tilde{\theta}D_t\bar u\partial_\alpha(I-\mathfrak{H})b d\alpha\right|\\
\leq&\ \norm{D_t\tilde{\theta}}_{L^2}\norm{D_tu}_{L^2}\norm{\partial_\alpha(I-\mathfrak{H})b-2q_\alpha}_{L^\infty}+\norm{D_t\tilde{\theta}}_{L^\infty}\norm{D_tu}_{L^\infty}\norm{2q_\alpha}_{L^1}\\
\lesssim&\frac{\epsilon^3}{t^{\frac{5}{4}-\mu-\delta_0}}\left(\ln (2+t)\right)^2+\frac{\lambda\epsilon}{H^2t^{\frac{1}{4}-\mu}}.
\end{aligned}
$$
Therefore,
\begin{equation}\label{d_t_tilde_E_three}
\left|\int D_t\tilde{\theta}D_t\bar ub_\alpha d\alpha\right|\lesssim\frac{\epsilon^3}{t^{\frac{5}{4}-\mu-\delta_0}}\left(\ln (2+t)\right)^2+\frac{\lambda\epsilon}{H^2t^{\frac{1}{4}-\mu}}.
\end{equation}
Set
$$
\tilde{I}_1=\int(D_t^3\tilde{\theta}-iA\partial_\alpha D_t\tilde{\theta})\bar u d\alpha=\int G_0^{\tilde{\sigma}}\bar u
$$
and
$$
\tilde{I}_2=\int D_t\tilde{\theta}(D_t^2\bar u+iA\bar u_\alpha)d\alpha.
$$
The following two lemmas clarify the asymptotic behavior of integrals above and help us complete the estimate of $\frac{d}{dt}\tilde{E}$:
\begin{lemma}\label{lem_tilde_I_1}
    For $t\in [1,T]$ and  $t^{-\mu}\leq|\nu|\leq t^{\mu}$, we have
    \begin{equation}\label{I_1_tilde_E}
    \left|\tilde{I}_1-\frac{i|D_t^2\zeta(\nu t,t)|^2}{4\nu}\tilde{E}\right|\lesssim\frac{\epsilon^3}{t^{\frac{3}{2}-2\mu}}+\frac{\epsilon^3}{t^{\frac{5}{4}-\mu-\delta_0}}+\frac{\epsilon^3}{t^{\frac{7}{6}}}+\frac{\lambda(\epsilon+\lambda)}{H^{\frac{3}{2}}t^{\frac{1}{4}}}.
    \end{equation}
\end{lemma}
\begin{proof}
Recall (\ref{G_0^sigma}):
$$
G^{\tilde{\sigma}}_0=i\frac{a_t}{a}\circ\kappa^{-1} A\tilde{\theta}_{\alpha}+D_t(G_c+G_d).
$$
Lemma \ref{lem_a_t/a_est} shows that
$$
\norm{i\frac{a_t}{a}\circ\kappa^{-1} A\tilde{\theta}_{\alpha}}_{L^2}\lesssim\norm{\frac{a_t}{a}\circ\kappa^{-1} A}_{L^2}\norm{\tilde{\theta}_\alpha}_{L^\infty}\lesssim\left(\frac{\epsilon^2}{t^{1-\delta_0}}\ln (2+t)+\frac{\lambda\epsilon}{H^{\frac{3}{2}}t^{\frac{1}{2}}}+\frac{\lambda}{H^{\frac{5}{2}}}\right)\frac{\epsilon}{t^{\frac{1}{2}}}.
$$
For $D_tG_d$, we note that
$$
D_tG_d=-2D_t[\bar q,\mathfrak{H}]\frac{\bar{\mathfrak{F}}_\alpha}{\zeta_\alpha}-2D_t[D_t\zeta,\mathfrak{H}]\frac{\bar q_\alpha}{\zeta_\alpha}-4D_t^2q.
$$
So according to Proposition \ref{pro_q},
$$
\begin{aligned}
\norm{D_tG_d}_{L^2}\lesssim&\norm{\partial_\alpha D_tq}_{L^2}\norm{\mathfrak{F}}_{L^\infty}+\norm{q_\alpha}_{L^2}\norm{D_t\mathfrak{F}}_{L^\infty}\\
&+\norm{\partial_\alpha D_t^2\zeta}_{L^\infty}\norm{q}_{L^2}+\norm{\partial_\alpha D_t\zeta}_{L^\infty}\norm{D_tq}_{L^2}+\norm{D_t^2q}_{L^2}\\
\lesssim&\ \frac{\lambda(\epsilon+\lambda)^2}{H^3}+\frac{\lambda\epsilon}{H^{\frac{3}{2}}t^{\frac{1}{2}}}.
\end{aligned}
$$
We now turn to $D_tG_c$. Revisit (\ref{G_c}):
$$\begin{aligned}
G_c=&-\frac{2}{\pi i}\int\frac{[(D_t\zeta-\bar q)(\alpha,t)-(D_t\zeta-\bar q)(\beta,t)][(\bar\zeta-\zeta)(\alpha,t)-(\bar\zeta-\zeta)(\beta,t)]}{|\zeta(\alpha,t)-\zeta(\beta,t)|^2}\partial_\beta(D_t\zeta-\bar q)d\beta\\
&+\frac{1}{\pi i}\int\left(\frac{D_{t}\zeta(\alpha,t)-D_{t}\zeta(\beta,t)}{\zeta(\alpha,t)-\zeta(\beta,t)}\right)^2(\zeta-\bar{\zeta})_{\beta}d\beta\\
=&\ \frac{2}{\pi i}\int\frac{[(D_t\zeta-\bar q)(\alpha,t)-(D_t\zeta-\bar q)(\beta,t)][(\bar\zeta-\zeta)(\alpha,t)-(\bar\zeta-\zeta)(\beta,t)]}{|\zeta(\alpha,t)-\zeta(\beta,t)|^2}\partial_\beta\bar qd\beta\\
&+\frac{2}{\pi i}\int\frac{[\bar q(\alpha,t)-\bar q(\beta,t)][(\bar\zeta-\zeta)(\alpha,t)-(\bar\zeta-\zeta)(\beta,t)]}{|\zeta(\alpha,t)-\zeta(\beta,t)|^2}\partial_\beta D_t\zeta d\beta\\
&-\frac{2}{\pi i}\int\frac{[D_t\zeta(\alpha,t)-D_t\zeta(\beta,t)][(\bar\zeta-\zeta)(\alpha,t)-(\bar\zeta-\zeta)(\beta,t)]}{|\zeta(\alpha,t)-\zeta(\beta,t)|^2}\partial_\beta D_t\zeta d\beta\\
&+\frac{1}{\pi i}\int\left(\frac{D_{t}\zeta(\alpha,t)-D_{t}\zeta(\beta,t)}{\zeta(\alpha,t)-\zeta(\beta,t)}\right)^2(\zeta-\bar{\zeta})_{\beta}d\beta.\\
\end{aligned}
$$
Let
$$\begin{aligned}
\mathfrak{R}=&\ \frac{2}{\pi i}\int\frac{[(D_t\zeta-\bar q)(\alpha,t)-(D_t\zeta-\bar q)(\beta,t)][(\bar\zeta-\zeta)(\alpha,t)-(\bar\zeta-\zeta)(\beta,t)]}{|\zeta(\alpha,t)-\zeta(\beta,t)|^2}\partial_\beta\bar qd\beta\\
&+\frac{2}{\pi i}\int\frac{[\bar q(\alpha,t)-\bar q(\beta,t)][(\bar\zeta-\zeta)(\alpha,t)-(\bar\zeta-\zeta)(\beta,t)]}{|\zeta(\alpha,t)-\zeta(\beta,t)|^2}\partial_\beta D_t\zeta d\beta.
\end{aligned}$$
Then
$$
\norm{D_t\mathfrak{R}}_{L^2}\lesssim(\norm{\partial_\alpha D_t\zeta}_{L^\infty}+\norm{q_\alpha}_{L^{\infty}})\norm{\zeta_\alpha-1}_{L^\infty}\norm{q}_{L^2}\lesssim\frac{\lambda\epsilon^2}{H^{\frac{3}{2}}t}+\frac{\lambda^2\epsilon}{H^{\frac{9}{2}}t^{\frac{1}{2}}}.
$$
For the remaining terms, we consider one situation. Let
$$
\mathfrak{I}=\frac{1}{\pi i}\int\frac{(D_{t}\zeta(\alpha,t)-D_{t}\zeta(\beta,t))(D_{t}^2\zeta(\alpha,t)-D_{t}^2\zeta(\beta,t))}{(\zeta(\alpha,t)-\zeta(\beta,t))^2}(\zeta-\bar{\zeta})_{\beta}d\beta.
$$
Appealing to Corollary \ref{cor_integral_three} and Remark {\ref{rem_S(t)}} yields that
$$
\norm{\mathfrak{I}-\frac{t}{\alpha}\partial_tD_t\zeta D_t^2\zeta(\bar\zeta_\alpha-1)}_{L^2\left(\frac{1}{2}t^{1-\mu}\leq|\alpha|\leq2t^{1+\mu}\right)}\lesssim\frac{\epsilon^3}{t^{\frac{3}{2}-\mu}}.
$$
The other situations can be treated like this. In summary, there holds
$$
\begin{aligned}
\norm{G_0^{\tilde{\sigma}}-\frac{t}{\alpha}\partial_tD_t\zeta D_t^2\zeta(\bar\zeta_\alpha-1)}_{L^2\left(\frac{1}{2}t^{1-\mu}\leq|\alpha|\leq2t^{1+\mu}\right)}&\\
\lesssim\left(\frac{\epsilon^2}{t^{1-\delta_0}}\ln (2+t)+\frac{\lambda\epsilon}{H^{\frac{3}{2}}t^{\frac{1}{2}}}+\frac{\lambda}{H^{\frac{3}{2}}}\right)\frac{\epsilon}{t^{\frac{1}{2}}}&+\frac{\lambda(\epsilon+\lambda)^2}{H^3}+\frac{\epsilon^3}{t^{\frac{3}{2}-\mu}}.
\end{aligned}
$$
Therefore,
\begin{equation}\label{I_1_one}
\begin{aligned}
\left|\tilde{I}_1-\int\frac{t}{\alpha}\partial_tD_t\zeta D_t^2\zeta(\bar\zeta_\alpha-1)\bar ud\alpha\right|\lesssim&\norm{G_0^{\tilde{\sigma}}-\frac{t}{\alpha}\partial_tD_t\zeta D_t^2\zeta(\bar\zeta_\alpha-1)}_{L^2\left(\frac{1}{2}t^{1-\mu}\leq|\alpha|\leq2t^{1+\mu}\right)}\norm{u}_{L^2}\\
\lesssim&\left(\frac{\epsilon^2}{t^{1-\delta_0}}\ln (2+t)+\frac{\lambda\epsilon}{H^{\frac{3}{2}}t^{\frac{1}{2}}}+\frac{\lambda}{H^{\frac{3}{2}}}\right)\frac{\epsilon}{t^{\frac{1}{4}}}+\frac{\lambda(\epsilon+\lambda)^2t^{\frac{1}{4}}}{H^3}+\frac{\epsilon^3}{t^{\frac{5}{4}-\mu}}.
\end{aligned}
\end{equation}
Note that
$$
\begin{aligned}
&\int\frac{t}{\alpha}\partial_tD_t\zeta D_t^2\zeta(\bar\zeta_\alpha-1)\bar ud\alpha+2\partial_\alpha D_t\zeta(\nu t,t)(\bar\zeta_\alpha(\nu t,t)-1)\int D_t^2\zeta\bar u\\
=&\int\left[\frac{t}{\alpha}\partial_tD_t\zeta (\alpha,t)(\bar\zeta_\alpha(\alpha,t)-1)-\frac{1}{\nu}\partial_tD_t\zeta (\nu t,t)(\bar\zeta_\alpha(\nu t,t)-1)\right]D_t^2\zeta\bar ud\alpha\\
&+\frac{2}{\nu t}(\frac{t}{2}\partial_tD_t\zeta (\nu t,t)+\nu t\partial_\alpha D_t\zeta(\nu t,t))(\bar\zeta_\alpha(\nu t,t)-1)\int D_t^2\zeta\bar u\\
\end{aligned}
$$
which implies that
$$
\begin{aligned}
&\left|\int\frac{t}{\alpha}\partial_tD_t\zeta D_t^2\zeta(\bar\zeta_\alpha-1)\bar ud\alpha+2\partial_\alpha D_t\zeta(\nu t,t)(\bar\zeta_\alpha(\nu t,t)-1)\int D_t^2\zeta\bar u\right|\\
\lesssim&\norm{\partial_\alpha\left[\frac{t}{\alpha}\partial_tD_t\zeta(\bar\zeta_\alpha-1)\right]}_{L^2\left(\frac{1}{2}t^{1-\mu}\leq|\alpha|\leq2t^{1+\mu}\right)}\norm{D_t^2\zeta}_{L^\infty}\norm{u}_{L^2}\\
&+\frac{1}{t^{1-\mu}}\norm{L_0D_t\zeta}_{L^\infty}\norm{\zeta_\alpha-1}_{L^\infty}\norm{D_t^2\zeta}_{L^2}\norm{u}_{L^2}\\
\lesssim&\ \frac{\epsilon^3}{t^{\frac{5}{4}-\mu}}.
\end{aligned}
$$
Referring to (\ref{tilde_E_D_t_theta_u}) leads that
$$
\left|\tilde{E}-4\int D_t^2\zeta\bar u\right|\lesssim\left|\tilde{E}-2\int D_t^2\tilde{\theta}\bar u\right|+\norm{D_t^2\tilde{\theta}-2D_t^2\zeta}_{L^\infty}\norm{u}_{L^1}\lesssim\frac{\epsilon}{t^{\frac{1}{2}-2\mu}}+\frac{\epsilon}{t^{\frac{1}{2}-\mu-\delta_0}}+\frac{\epsilon^2}{t^{\frac{1}{6}}}+\frac{\lambda(\epsilon+\lambda)t^{\frac{1}{2}}}{H^3}.
$$
So
\begin{equation}\label{I_1_two}
\begin{aligned}
&\left|\int\frac{t}{\alpha}\partial_tD_t\zeta D_t^2\zeta(\bar\zeta_\alpha-1)\bar ud\alpha+\frac{1}{2}\partial_\alpha D_t\zeta(\nu t,t)(\bar\zeta_\alpha(\nu t,t)-1)\tilde{E}\right|\\
\lesssim&\ \frac{\epsilon^3}{t^{\frac{5}{4}-\mu}}+\norm{\partial_\alpha D_t\zeta}_{L^\infty}\norm{\zeta_\alpha-1}_{L^\infty}\left|\tilde{E}-4\int D_t^2\zeta\bar u\right|\\
\lesssim&\ \frac{\epsilon^3}{t^{\frac{5}{4}-\mu}}+\frac{\epsilon^2}{t}\left(\frac{\epsilon}{t^{\frac{1}{2}-2\mu}}+\frac{\epsilon}{t^{\frac{1}{2}-\mu-\delta_0}}+\frac{\epsilon^2}{t^{\frac{1}{6}}}+\frac{\lambda(\epsilon+\lambda)t^{\frac{1}{2}}}{H^3}\right).
\end{aligned}
\end{equation}
Recall that
$$
\partial_\alpha D_t\zeta=-\frac{t}{2\alpha}D_t^2\zeta+\frac{t}{2\alpha}b\partial_\alpha D_t\zeta+\frac{1}{\alpha}L_0D_t\zeta
$$
and
$$
\zeta_\alpha-1=\zeta_\alpha(1-A)-iD_t^2\zeta.
$$
Combining (\ref{I_1_one}) and (\ref{I_1_two}) yields (\ref{I_1_tilde_E}).
\end{proof}
\begin{lemma}\label{lem_tilde_I_2}
    For $t\in[1,T]$ and $t^{-\mu}\leq|\nu|\leq t^{\mu}$ there holds
    \begin{equation}\label{I_2_tilde_E}
    \begin{aligned}
    &\left|\tilde{I_2}+\frac{\hat{b}(\nu t,t)}{2\pi i\nu}\tilde{E}\right|\\
    \leq&\ \tilde{C}\frac{\epsilon}{t^{\frac{5}{4}-\mu-\delta_0}}+C\left(\frac{\epsilon^3}{t^{\frac{5}{4}-\mu-\delta_0}}\ln (2+t)+\frac{\epsilon^3}{t^{\frac{3}{2}-2\mu}}+\frac{\lambda(\lambda+\epsilon)\epsilon(1+t)^{\frac{1}{4}}}{H^3}+\frac{\lambda\epsilon}{H^2t^{\frac{1}{4}-\mu}}+\frac{\lambda\epsilon}{Ht^{\frac{1}{2}-\mu}}\right)
    \end{aligned}
    \end{equation}
    where
    $$
    \hat{b}(\alpha,t)=\operatorname{Re}{\int\frac{D_t\zeta(\alpha,t)-D_t\zeta(\beta,t)}{\alpha-\beta}\bar\zeta_{\beta t}d\beta}.
    $$
\end{lemma}
\begin{proof}
    Since
    $$\begin{aligned}
    \tilde{I_2}=&\int D_t\tilde{\theta}(D_t^2\bar u+iA\bar u_\alpha)\\
    =&\int D_t\tilde{\theta}(\partial_t^2+i\partial_\alpha)\bar u+\int D_t\tilde{\theta}[b\partial_\alpha\bar u_t+b\partial_\alpha(b\bar u_\alpha)+\partial_t(b\bar u_\alpha)+i(A-1)\bar u_\alpha]\\
    =&\int D_t\tilde{\theta}(\partial_t^2+i\partial_\alpha)\bar u+\int\left(-2\partial_\alpha D_t\tilde{\theta}b\bar u_t-2D_t\tilde{\theta}b_\alpha\bar u_t-\partial_\alpha D_t\tilde{\theta}b^2\bar u_\alpha-D_t\tilde{\theta}b_\alpha b\bar u_\alpha\right)\\
    &+\int D_t\tilde{\theta}[b_t\bar u_\alpha+i(A-1)\bar u_\alpha]\\
    =&\int D_t\tilde{\theta}(\partial_t^2+i\partial_\alpha)\bar u+\int-2\partial_\alpha D_t\tilde{\theta}b\bar u_t\\
    &+\int-2D_t\tilde{\theta}b_\alpha\bar u_t-\int\partial_\alpha(-\partial_\alpha D_t\tilde{\theta}b^2-D_t\tilde{\theta}b_\alpha b+D_t\tilde{\theta}b_t+iD_t\tilde{\theta}(A-1))\bar u,
    \end{aligned}
    $$
    H\"older's inequality implies that
    \begin{equation}\label{I_2_four}
    \begin{aligned}
    &\left|\int-\partial_\alpha[-\partial_\alpha D_t\tilde{\theta}b^2-D_t\tilde{\theta}b_\alpha b+D_t\tilde{\theta}b_t+iD_t\tilde{\theta}(A-1)]\bar u\right|\\
    \lesssim&\left[\norm{D_t\tilde{\theta}}_{H^2}(\norm{b}_{W^{2,\infty}}^2+\norm{b_t}_{W^{1,\infty}}+\norm{A-1}_{W^{1,\infty}})\right]\norm{u}_{L^2}\\
    \lesssim&\ \frac{\epsilon^3}{t^{\frac{5}{4}-\delta_0}}(\ln (2+t))^2+\frac{\lambda(\lambda+\epsilon)\epsilon (1+t)^{\frac{1}{4}}}{H^3}.
    \end{aligned}
    \end{equation}
    and from \eqref{d_t_tilde_E_three}
    \begin{equation}\label{I_2_three}
    \begin{aligned}
    \left|\int D_t\tilde{\theta}b_\alpha\bar u_t\right|\lesssim\frac{\epsilon^3}{t^{\frac{5}{4}-\mu-\delta_0}}\left(\ln (2+t)\right)^2+\frac{\lambda\epsilon}{H^2t^{\frac{1}{4}-\mu}}.
    \end{aligned}
    \end{equation}
    For the term $\int-2\partial_\alpha D_t\tilde{\theta}b\bar u_t$ it requires more delicate operations. Decompose $b$ as $\frac{1}{2}[(I-\mathfrak{H})b+(I-\bar{\mathfrak{H}})b+(\mathfrak{H}+\bar{\mathfrak{H}})b]$ and we focus on the term
    $$
    \begin{aligned}
    \int-\partial_\alpha D_t\tilde{\theta}(I-\mathfrak{H})b\bar u_t=\ &\frac{1}{\pi i}\iint\partial_\alpha D_t\tilde{\theta}(\alpha,t)\frac{D_t\zeta(\alpha,t)-D_t\zeta(\beta,t)}{\zeta(\alpha,t)-\zeta(\beta,t)}(\bar\zeta_\beta(\beta,t)-1)\bar u_t(\alpha,t)d\alpha d\beta\\
    &-2\int\partial_\alpha D_t\tilde{\theta}q\bar u_t
    \end{aligned}
    $$
    where
    \begin{equation}\label{b_u_t_1}
    \left|\int\partial_\alpha D_t\tilde{\theta}q\bar u_t\right|\lesssim\norm{\partial_\alpha D_t\tilde{\theta}}_{L^\infty}\norm{q}_{L^1}\norm{u_t}_{L^\infty}\lesssim\frac{\lambda\epsilon}{Ht^{\frac{1}{2}-\mu}}.
    \end{equation}
    Meanwhile,
    \begin{equation}\label{b_u_t_2}
    \begin{aligned}
    &\left|\iint\partial_\alpha D_t\tilde{\theta}(\alpha,t)\left(\frac{1}{\zeta(\alpha,t)-\zeta(\beta,t)}-\frac{1}{\alpha-\beta}\right)(D_t\zeta(\alpha,t)-D_t\zeta(\beta,t))(\bar\zeta_\beta(\beta,t)-1)\bar u_t(\alpha,t)d\alpha d\beta\right|\\
    \lesssim\ &\norm{\partial_\alpha D_t\tilde{\theta}}_{L^\infty}\norm{u_t}_{L^2}\norm{D_t\zeta}_{L^2}\norm{\zeta_\alpha-1}^2_{L^\infty}\\
    \lesssim\ &\frac{\epsilon^4}{t^{\frac{5}{4}-\mu}}.
    \end{aligned}
    \end{equation}
    So we turn to
    $$
    \begin{aligned}
    &\frac{1}{\pi i}\iint\partial_\alpha D_t\tilde{\theta}(\alpha,t)\frac{D_t\zeta(\alpha,t)-D_t\zeta(\beta,t)}{\alpha-\beta}(\bar\zeta_\beta(\beta,t)-1)\bar u_t(\alpha,t)d\alpha d\beta\\
    =\ &\frac{1}{\pi i}\iint\frac{D_t\zeta(\alpha,t)-D_t\zeta(\beta,t)}{\alpha-\beta}[(\bar\zeta_\beta(\beta,t)-1)\partial_\alpha D_t\tilde{\theta}(\alpha,t)+\bar\zeta_{\beta t}(\beta,t)\tilde{\theta}_\alpha(\alpha,t)]\bar u_t(\alpha,t)d\alpha d\beta\\
    &-\frac{1}{\pi i}\iint\tilde{\theta}_\alpha(\alpha,t)\frac{D_t\zeta(\alpha,t)-D_t\zeta(\beta,t)}{\alpha-\beta}\bar\zeta_{\beta t}(\beta,t)\bar u_t(\alpha,t)d\alpha d\beta.
    \end{aligned}
    $$
    Adding \eqref{transfer_d_3} and \eqref{transfer_d_var_1} yields that
    $$
    f_t(\alpha,t)\bar g(\beta,t)+f(\alpha,t)\bar g_t(\beta,t)=\frac{2}{it}(\Omega_0f(\alpha,t)\bar g_t(\beta,t)-f_t(\alpha,t)\overline{\Omega_0g(\beta,t)})-\frac{2}{it}(\alpha-\beta)f_t(\alpha,t)\bar g_t(\beta,t).
    $$
    Hence,
    $$
    \begin{aligned}
    &\frac{1}{\pi i}\iint\partial_\alpha D_t\tilde{\theta}(\alpha,t)\frac{D_t\zeta(\alpha,t)-D_t\zeta(\beta,t)}{\alpha-\beta}(\bar\zeta_\beta(\beta,t)-1)\bar u_t(\alpha,t)d\alpha d\beta\\
    =\ &-\frac{1}{\pi i}\iint\tilde{\theta}_\alpha(\alpha,t)\frac{D_t\zeta(\alpha,t)-D_t\zeta(\beta,t)}{\alpha-\beta}\bar\zeta_{\beta t}(\beta,t)\bar u_t(\alpha,t)d\alpha d\beta\\
    &+\frac{2}{\pi t}\iint\tilde{\theta}_{\alpha t}(\alpha,t)(D_t\zeta(\alpha,t)-D_t\zeta(\beta,t))\bar\zeta_{\beta t}(\beta,t)\bar u_t(\alpha,t)d\alpha d\beta+R\\
    =\ &-\frac{1}{\pi i}\iint\tilde{\theta}_\alpha(\alpha,t)\frac{D_t\zeta(\alpha,t)-D_t\zeta(\beta,t)}{\alpha-\beta}\bar\zeta_{\beta t}(\beta,t)\bar u_t(\alpha,t)d\alpha d\beta\\
    &-\frac{2}{\pi t}\int\tilde{\theta}_{\alpha t}\bar u_td\alpha\int D_t\zeta\bar\zeta_{\beta t} d\beta+R
    \end{aligned}
    $$
    where
    \begin{equation}\label{b_u_t_3}
    \begin{aligned}
    |R|\lesssim\ &\frac{1}{t}\norm{u_t}_{L^2}\norm{D_t\zeta}_{L^\infty}\left(\norm{\Omega_0(\zeta_\alpha-1)}_{L^2}\norm{\tilde{\theta}_{\alpha t}}_{L^\infty}+\norm{\partial_\alpha \zeta_t}_{L^\infty}\norm{\Omega_0 \tilde{\theta}_\alpha}_{L^2}\right)\ln (2+t)\\
    &+\norm{b}_{L^\infty}\norm{\tilde{\theta}_{\alpha\alpha}}_{L^\infty}\norm{D_t\zeta}_{L^2}\norm{\zeta_\alpha-1}_{L^\infty}\norm{u_t}_{L^2}\\
    \lesssim\ &\frac{\epsilon^3}{t^{\frac{5}{4}-\mu-\delta_0}}\ln (2+t).
    \end{aligned}
    \end{equation}
    Combining \eqref{b_u_t_1}-\eqref{b_u_t_3} leads that
    $$
    \begin{aligned}
    \int-\partial_\alpha D_t\tilde{\theta}(I-\mathfrak{H})b\bar u_t=&-\frac{1}{\pi i}\iint\tilde{\theta}_\alpha(\alpha,t)\frac{D_t\zeta(\alpha,t)-D_t\zeta(\beta,t)}{\alpha-\beta}\bar\zeta_{\beta t}(\beta,t)\bar u_t(\alpha,t)d\alpha d\beta\\
    &-\frac{2}{\pi t}\int\tilde{\theta}_{\alpha t}\bar u_td\alpha\int D_t\zeta\bar\zeta_{\beta t} d\beta+R_1
    \end{aligned}
    $$
    where
    $$
    |R_1|\lesssim\frac{\epsilon^3}{t^{\frac{5}{4}-\mu-\delta_0}}\ln (2+t)+\frac{\lambda\epsilon}{Ht^{\frac{1}{2}-\mu}}.
    $$
    Similarly,
    $$
    \begin{aligned}
    \int-\partial_\alpha D_t\tilde{\theta}(I-\bar{ \mathfrak{H}})b\bar u_t=&-\frac{1}{\pi i}\iint\tilde{\theta}_\alpha(\alpha,t)\frac{D_t\bar \zeta(\alpha,t)-D_t\bar \zeta(\beta,t)}{\alpha-\beta}\zeta_{\beta t}(\beta,t)\bar u_t(\alpha,t)d\alpha d\beta\\
    &-\frac{2}{\pi t}\int\tilde{\theta}_{\alpha t}\bar u_td\alpha\int D_t\bar \zeta\zeta_{\beta t} d\beta+R_2
    \end{aligned}
    $$
    where $R_2$ satisfies the same inequality of $R_1$. Set
    $$
    \hat{b}(\alpha,t)=\operatorname{Re}{\int\frac{D_t\zeta(\alpha,t)-D_t\zeta(\beta,t)}{\alpha-\beta}\bar\zeta_{\beta t}d\beta}
    $$
    Imitating the proof of Lemma \ref{lem_b} shows that
    \begin{equation}\label{hat_b_alpha_L^2}
    \norm{\hat{b}_\alpha}_{L^2}\lesssim\frac{\epsilon^2}{t^{\frac{3}{2}-\delta_0}},\qquad \norm{\hat{b}}_{L^2}\lesssim\frac{\epsilon^2}{t^{\frac{1}{2}}}.
    \end{equation}
    Since in $\hat{b}$ we can transfer the t-derivative of $\bar\zeta_{\beta t}$ to $D_t\zeta$ via \eqref{transfer_d_var_1}, the following inequality holds:
    \begin{equation}\label{hat_b_L^infty}
    \norm{\hat{b}}_{L^\infty}\lesssim\frac{\epsilon^2}{t}\ln (2+t).
    \end{equation}
    Note that
    $$
    \begin{aligned}
    \int-2\partial_\alpha D_t\tilde{\theta}b\bar u_t=&-\int\partial_\alpha D_t\tilde{\theta}(I-\mathfrak{H})b\bar u_t-\int\partial_\alpha D_t\tilde{\theta}(I-\bar{\mathfrak{H}})b\bar u_t-\int\partial_\alpha D_t\tilde{\theta}(\mathfrak{H}+\bar{\mathfrak{H}})b\bar u_t\\
    =&-\frac{2}{\pi i}\int \hat{b}\tilde{\theta}_\alpha \bar u_t-\frac{2}{\pi t}\int\tilde{\theta}_{\alpha t}\bar u_td\alpha\int b\zeta_\beta\bar\zeta_{\beta t}+b\bar \zeta_\beta\zeta_{\beta t}+\partial_\beta(|\zeta_t|^2) d\beta+R_1+R_2\\
    &-\int\partial_\alpha D_t\tilde{\theta}(\mathfrak{H}+\bar{\mathfrak{H}})b\bar u_t
    \end{aligned}
    $$
    which implies that
    \begin{equation}\label{hat_b_1}
    \left|\int-2\partial_\alpha D_t\tilde{\theta}b\bar u_t+\frac{2}{\pi i}\int \hat{b}\tilde{\theta}_\alpha \bar u_t\right|\lesssim\frac{\epsilon^3}{t^{\frac{5}{4}-\mu-\delta_0}}\ln (2+t)+\frac{\lambda\epsilon}{Ht^{\frac{1}{2}-\mu}}.
    \end{equation}
    We claim that
    \begin{equation}\label{hat_b_2}
    \left|\int \hat{b}\tilde{\theta}_\alpha \bar u_t-\frac{\hat{b}(\nu t,t)}{4\nu}\tilde{E}\right|\lesssim\frac{\epsilon^3}{t^{\frac{3}{2}-2\mu}}+\frac{\epsilon^3}{t^{\frac{5}{4}-\mu-\delta_0}}.
    \end{equation}
    Since
    \begin{equation}\label{hat_b_2_1}
    \begin{aligned}
    \int \hat{b}\tilde{\theta}_\alpha \bar u_t=&\int \hat{b}\left(\tilde{\theta}_\alpha+iD_t^2\tilde{\theta} \right)\bar u_t-i\int \left(\hat{b}-\hat{b}(\nu t,t)\right)D_t^2\tilde{\theta} \bar u_t-i\hat{b}(\nu t,t)\int D_t^2\tilde{\theta} \left(\bar u_t-\frac{it}{2\alpha}\bar u\right)\\
    &+\hat{b}(\nu t,t)\int D_t^2\tilde{\theta}\left(\frac{t}{2\alpha}-\frac{1}{2\nu}\right)\bar u+\frac{\hat{b}(\nu t,t)}{2\nu}\left(\int D_t^2\tilde{\theta}\bar u-\frac{1}{2}\tilde{E}\right)+\frac{\hat{b}(\nu t,t)}{4\nu}\tilde{E},
    \end{aligned}
    \end{equation}
    we need to control $\frac{\hat{b}(\nu t,t)}{\nu}$. The definition of $\hat{b}$ shows that
    $$
    \begin{aligned}
    \hat{b}=\ &\operatorname{Re}\int\frac{1}{\alpha-\beta}\partial_t[(D_t\zeta(\alpha,t)-D_t\zeta(\beta,t))(\bar\zeta_\beta-1)]d\beta\\
     &-\operatorname{Re}\int\frac{D_t^2\zeta(\alpha,t)-D_t^2\zeta(\beta,t)}{\alpha-\beta}(\bar\zeta_\beta-1-iD_t^2\bar\zeta)(\beta,t)d\beta\\
     &+\operatorname{Re}\int\frac{(b\partial_\alpha D_t\zeta)(\alpha,t)-(b\partial_\alpha D_t\zeta)(\beta,t)}{\alpha-\beta}(\bar\zeta_\beta-1)d\beta\\
     &-\operatorname{Re}\int\frac{D_t^2\zeta(\alpha,t)-D_t^2\zeta(\beta,t)}{\alpha-\beta}iD_t^2\bar\zeta(\beta,t)d\beta.
    \end{aligned}
    $$
    Therefore,
    $$
    \begin{aligned}\norm{\hat{b}+\operatorname{Re}\left(iD_t^2\zeta\int\frac{D_t^2\bar\zeta(\beta,t)}{\alpha-\beta}d\beta\right)}_{L^\infty}=&\ \norm{\hat{b}+\operatorname{Re}\int\frac{D_t^2\zeta(\alpha,t)-D_t^2\zeta(\beta,t)}{\alpha-\beta}iD_t^2\bar\zeta(\beta,t)d\beta}_{L^\infty}\\
    \lesssim&\ \frac{\epsilon^2}{t^{\frac{5}{4}-\delta_0}}\ln (2+t).
    \end{aligned}
    $$
    We return to $\frac{\hat{b}(\nu t,t)}{\nu}$:
    $$
    \begin{aligned}
    \frac{\hat{b}(\nu t,t)}{\nu}=\ &\frac{1}{\nu}\left[\hat{b}(\nu t,t)+\operatorname{Re}\left(iD_t^2\zeta(\nu t,t)\int\frac{D_t^2\bar\zeta(\beta,t)}{\nu t-\beta}d\beta\right)\right]\\
    &-\operatorname{Re}\left(i\left(\frac{1}{\nu}(b\partial_\alpha D_t\zeta)(\nu t,t)+\frac{2}{\nu t}L_0D_t\zeta(\nu t,t)-{2}\partial_\alpha D_t\zeta(\nu t,t)\right)\int\frac{D_t^2\bar\zeta(\beta,t)}{\nu t-\beta}d\beta\right).
    \end{aligned}
    $$
    Thus,
    \begin{equation}\label{hat_b_2_2}
    \left|\frac{\hat{b}(\nu t,t)}{\nu}\right|\lesssim\frac{\epsilon^2}{t}\ln (2+t)+\frac{\epsilon^2}{t^{\frac{5}{4}-\mu-\delta_0}}\ln (2+t).
    \end{equation}
    From (\ref{u_t}) we see
    \begin{equation}\label{hat_b_2_3}
        \norm{u_t+\frac{it}{2\alpha}u}_{L^2}\lesssim\frac{1}{t^{\frac{1}{4}-\mu}}.
    \end{equation}
    By H\"older's inequality and Hardy's inequality:
    \begin{equation}\label{hat_b_2_4}
    \begin{aligned}
    \left|\int \left(\hat{b}-\hat{b}(\nu t,t)\right)D_t^2\tilde{\theta} \bar u_t\right|\lesssim\ &\norm{\frac{\hat{b}(\alpha,t)-\hat{b}(\nu t,t)}{\alpha-\nu t}}_{L^2}\norm{D_t^2\tilde{\theta}}_{L^\infty}\norm{(\alpha-\nu t)\bar u_t}_{L^2}\\
    \lesssim\ &\norm{\hat{b}_\alpha}_{L^2}\norm{D_t^2\tilde{\theta}}_{L^\infty}\norm{(\alpha-\nu t)\bar u_t}_{L^2}\\
    \lesssim\ &\frac{\epsilon^3}{t^{\frac{5}{4}-\mu-\delta_0}}
    \end{aligned}
    \end{equation}
    Combining \eqref{hat_b_2_1}-\eqref{hat_b_2_4} and \eqref{tilde_E_D_t_theta_u} yields \eqref{hat_b_2}. By \eqref{hat_b_1} and \eqref{hat_b_2} we deduce that
    \begin{equation}\label{I_2_two}
    \left|\int-2\partial_\alpha D_t\tilde{\theta}b\bar u_t+\frac{\hat{b}(\nu t,t)}{2\pi i\nu}\tilde{E}\right|\lesssim\frac{\epsilon^3}{t^{\frac{3}{2}-2\mu}}+\frac{\epsilon^3}{t^{\frac{5}{4}-\mu-\delta_0}}\ln (2+t)+\frac{\lambda\epsilon}{Ht^{\frac{1}{2}-\mu}}.
    \end{equation}
    For the term $\int D_t\tilde{\theta}(\bar u_{tt}+i\bar u_\alpha)$ we rely on Lemma 7.1 in \cite{su2025new}:
    \begin{equation}\label{I_2_one}
    \left|\int D_t\tilde{\theta}(\bar u_{tt}+i\bar u_\alpha)\right|\leq\tilde{C}\frac{\epsilon}{t^{\frac{5}{4}-\mu-\delta_0}}.
    \end{equation}
    The combination of (\ref{I_2_four}), (\ref{I_2_three}), (\ref{I_2_two}) and (\ref{I_2_one}) verifies (\ref{I_2_tilde_E}).
\end{proof}
Take (\ref{d_t_tilde_E_one}), (\ref{d_t_tilde_E_two}), Lemma \ref{lem_tilde_I_1} and Lemma \ref{lem_tilde_I_2} into consideration and we demonstrate the following theorem:
\begin{thm}\label{thm_d_t_tilde_E}
    For $t\in[1,T]$ and $t^{-\mu}\leq|\nu|\leq t^\mu$, if $\mu=\frac{1}{5}$,
    \begin{equation}
    \begin{aligned}
    &\left|\frac{d}{dt}\left(\tilde{E}(\nu,t)\right)-\frac{i|D_t^2\zeta(\nu t,t)|^2}{4\nu}\tilde{E}+i\frac{\hat{b}(\nu t,t)}{2\pi \nu}\tilde{E}\right|\\
    \leq&\ \tilde{C}\frac{\epsilon}{t^{1+\frac{1}{20}-\delta_0}}+C\left(\frac{\epsilon^3}{t^{1+\frac{1}{20}-\delta_0}}(\ln (2+t))^2+\frac{\lambda(\lambda+\epsilon)\epsilon(1+t)^{\frac{1}{4}}}{H^3}+\frac{\lambda\epsilon}{H^2t^{\frac{1}{20}}}+\frac{\lambda\epsilon}{Ht^{\frac{3}{10}}}+\frac{\lambda(\epsilon+\lambda)}{H^{\frac{3}{2}}t^{\frac{1}{4}}}\right).
    \end{aligned}
    \end{equation}
\end{thm}
\begin{cor}\label{cor_tilde_E}
    \begin{equation}\label{tilde_E_beh}
    \norm{\tilde{E}(\cdot,t)}_{L^\infty(D(t))}\lesssim\tilde{C}\epsilon+C\epsilon^3.
    \end{equation}
    where $D(t)=\{\nu\in \mathbb{R}\setminus \{0\}\mid t^{-\frac{1}{5}}\leq|\nu|\leq t^{\frac{1}{5}}\}$
\end{cor}
\begin{proof}
    Theorem \ref{thm_d_t_tilde_E} indicates that
    $$
    \begin{aligned}
    &\left|\frac{d}{dt}\left(e^{-i\int_1^t\frac{|D_t^2\zeta(\nu\tau,\tau)|^2}{4\nu}-\frac{\hat{b}(\nu \tau,\tau)}{2\pi \nu}d\tau}\tilde{E}(\nu,t)\right)\right|\\
    \leq&\ C\left(\frac{\epsilon^3}{(1+t)^{\frac{21}{20}-\delta_0}}(\ln (2+t))^2+\frac{\lambda(\lambda+\epsilon)\epsilon(1+t)^{\frac{1}{4}}}{H^3}+\frac{\lambda\epsilon}{H^2t^{\frac{1}{20}}}+\frac{\lambda\epsilon}{Ht^{\frac{3}{10}}}+\frac{\lambda(\epsilon+\lambda)}{H^{\frac{3}{2}}(1+t)^{\frac{1}{4}}}\right)\\
    &+\tilde{C}\frac{\epsilon}{(1+t)^{\frac{21}{20}-\delta_0}}.
    \end{aligned}
    $$
    The inequality holds in the domain $\mathcal{D}=\{(\nu,t)\mid t^{-\frac{1}{5}}\leq|\nu|\leq t^{\frac{1}{5}}\}$. If $|\nu|>1$, we integrate from $|\nu|^{\frac{1}{\mu}}$
 to $t$. If $|\nu|\leq1$, we integrate from $|\nu|^{-\frac{1}{\mu}}$
 to $t$. Then if $|\nu|>1$,
    $$
    \tilde{E}(\nu,t)\lesssim\tilde{E}(\nu,|\nu|^{\frac{1}{\mu}})+\tilde{C}\epsilon+C\epsilon^3+C\frac{\epsilon+\lambda}{H_0^{\frac{1}{2}}}+C\frac{\lambda^{\frac{3}{10}}\epsilon}{H_0^{\frac{3}{10}}}\lesssim\tilde{E}(\nu,|\nu|^{\frac{1}{\mu}})+(\tilde{C}+1)\epsilon+C\epsilon^3
    $$
    where the last inequality follows from \eqref{lambda_H_0}. 
    If $|\nu|\leq1$,
    $$\tilde{E}(\nu,t)\lesssim\tilde{E}(\nu,|\nu|^{-\frac{1}{\mu}})+(\tilde{C}+1)\epsilon+C\epsilon^3.$$
    From Lemma \ref{lem_away_infty} and Lemma \ref{lem_tilde_E} we obtain \eqref{tilde_E_beh}.
\end{proof}

\appendix

\section{Singular integrals and commutators}

Let $m\geq 1$ be an integer. Denote
\begin{equation}
S_1(A,f)=p.v.\int \prod_{j=1}^m \frac{A_j(\alpha)-A_j(\beta)}{\gamma_j(\alpha)-\gamma_j(\beta)} \frac{f(\beta)}{\gamma_0(\alpha)-\gamma_0(\beta)}d\beta.
\end{equation}

\begin{equation}
S_2(A,f)=\int \prod_{j=1}^m \frac{A_j(\alpha)-A_j(\beta)}{\gamma_j(\alpha)-\gamma_j(\beta)} f_{\beta}(\beta)d\beta.
\end{equation}
\noindent We have the following commutator estimates, which can be found in \cite{Totz2012}, \cite{Wu2009}.
\begin{lemma}\label{app_S1}
We have the following results:

(1) Assume each $\gamma_j (j=0,...,m)$ satisfies 
\begin{equation}\label{gammaj}
C_{0,j}|\alpha-\beta|\leq |\gamma_j(\alpha)-\gamma_j(\beta)|\leq C_{1,j}|\alpha-\beta|.
\end{equation}
Then both $\|S_1(A,f)\|_{L^2}$ and $\|S_2(A,f)\|_{L^2}$ are bounded by
$$C\prod_{j=1}^m \|A_j'\|_{X_j}\|f\|_{X_0},$$
where one of the $X_0, X_1, ...X_m$ is equal to $L^2$ and the rest are $L^{\infty}$. The constant $C$ depends on $\|\gamma_j'\|_{L^{\infty}}^{-1}, j=0,1,..,m$.

(2) Let $s\geq 3$ be given, and assume (\ref{gammaj}) for each $\gamma_j$, then 
$$\|S_2(A,f)\|_{H^s}\leq C\prod_{j=1}^m \|A_j'\|_{Y_j}\|f\|_Z,$$
where for all $j=1,...,m$, $Y_j=H^{s-1}$ or $W^{[s-1]+1,\infty}$ and $Z=H^s$ or $W^{[s]+1,\infty}$. At most one of these $Y_j, Z$ norms is in $H^k$ ($k=s-1$ for $Y_j$ or $s$ for $Z$). The constant $C$ depends on $\norm{\partial_\alpha\gamma_j-1}_{H^{s-1}},\ j=1,...,m$.

(3) Assume (\ref{gammaj}) for each $\gamma_j$, then
\[
\norm{S_1(A,f)}_{L^\infty}\le C\left(\prod_{j=1}^m\norm{A'_j}_{W^{1,\infty}}\norm{f}_{W^{1,\infty}}+\prod_{j=1}^m\norm{A_j'}_{L^\infty}\norm{f}_{L^\infty}\ln r+\prod_{j=1}^m\norm{A_j'}_{L^\infty}\norm{f}_{L^2}r^{-\frac{1}{2}}\right)
\]
for any $r>1$. The constant $C$ depends on $\|\gamma_j'\|_{L^{\infty}}^{-1}, j=0,1,..,m$.
\end{lemma}

\begin{lemma}\label{derivative_K_com}
    Let $K$ be such that $K$ or $(\alpha-\beta)K(\alpha,\beta;t)$ is continuous and bounded, and $K$ is smooth away from the diagonal $\Delta=\{(\alpha,\beta)| \alpha=\beta\}$. Denote
    \begin{equation}
        \boldsymbol{K}f(\alpha,t)=p.v. \int K(\alpha,\beta;t)f(\beta,t)d\beta
    \end{equation}
    Then we have 
    \begin{equation}
        \begin{split}
            [\partial_t, \boldsymbol{K}]f(\alpha,t)=&\int \partial_t K(\alpha,\beta;t)f(\beta,t)d\beta.\\
            [\partial_{\alpha},\boldsymbol{K}]f(\alpha,t)=& \int (\partial_{\alpha}+\partial_{\beta})K(\alpha,\beta;t)f(\beta,t)d\beta.\\
            [L_0, \boldsymbol{K}]f(\alpha,t)=& \int (\alpha\partial_{\alpha}+\beta\partial_{\beta}+\frac{1}{2}t\partial_t)K(\alpha,\beta;t)f(\beta,t)d\beta+\boldsymbol{K}f(\alpha,t)
        \end{split}
    \end{equation}
    for $f\in C^1(\mathbb{R};\mathbb{R}^d)$ which vanishes as $|\alpha|\rightarrow\infty$.
\end{lemma}

\end{document}